\documentclass[11pt]{article}
\pdfoutput=1 %for arxiv
\usepackage[utf8]{inputenc}

\usepackage[margin=1in]{geometry}
\usepackage{amsfonts}
\usepackage{amsmath}
\usepackage{amssymb}
\usepackage{amsthm}
\usepackage{float}
\usepackage[font=small,labelfont=bf]{caption} %customise text style on fig captions
\usepackage{subfigure}
\usepackage[ruled,vlined]{algorithm2e}

\allowdisplaybreaks

\usepackage{xspace}
\usepackage{mathtools}
\usepackage{mleftright}
\usepackage{xparse}

\usepackage{bold-extra} %adds bold+smallcaps (for standard fonts)

\usepackage{hyperref}
\usepackage[svgnames]{xcolor}
\hypersetup{colorlinks={true},urlcolor={blue},linkcolor={DarkBlue},citecolor=[named]{DarkGreen}}
\usepackage[authoryear,square]{natbib}
\usepackage{doi}

\usepackage{tcolorbox}
\tcbuselibrary{breakable}
\tcbuselibrary{skins}

\usepackage{microtype}
\usepackage[capitalise,nameinlink,noabbrev]{cleveref}

\usepackage{tikz}
\usetikzlibrary{arrows.meta}

\usepackage{graphicx}

\usepackage{enumitem}

\definecolor{blueContr}{RGB}{116,144,202}
\definecolor{redContr}{RGB}{234,92,72}
\definecolor{redDark}{RGB}{139,40,14}
\definecolor{blueDark}{RGB}{41,61,104}

\newcommand{\vecx}{\ensuremath{\boldsymbol{x}}}
\newcommand{\vecy}{\ensuremath{\boldsymbol{y}}}

\newcommand{\N}{\ensuremath{\mathbb{N}}}

\newcommand{\R}{\ensuremath{\mathbb{R}}}

\newcommand{\calA}{\ensuremath{\mathcal{A}}}
\newcommand{\calB}{\ensuremath{\mathcal{B}}}
\newcommand{\calC}{\ensuremath{\mathcal{C}}}
\newcommand{\calD}{\ensuremath{\mathcal{D}}}

\newcommand{\calI}{\ensuremath{\mathcal{I}}}

\newcommand{\calL}{\ensuremath{\mathcal{L}}}

\newcommand{\calQ}{\ensuremath{\mathcal{Q}}}

\newcommand{\calT}{\ensuremath{\mathcal{T}}}

\newcommand{\calV}{\ensuremath{\mathcal{V}}}

\NewDocumentCommand\xDeclarePairedDelimiter{mmm}
 {%
  \NewDocumentCommand#1{som}{%
   \IfNoValueTF{##2}
    {\IfBooleanTF{##1}{#2##3#3}{\mleft#2##3\mright#3}}
    {\mathopen{##2#2}##3\mathclose{##2#3}}%
  }%
 }

\xDeclarePairedDelimiter\inner{\langle}{\rangle}
\xDeclarePairedDelimiter\ang{\langle}{\rangle}
\xDeclarePairedDelimiter\abs{\lvert}{\rvert}
\xDeclarePairedDelimiter\set{\{}{\}}
\xDeclarePairedDelimiter\p{(}{)}
\let\b=\relax
\xDeclarePairedDelimiter\b{[}{]}
\xDeclarePairedDelimiter\round{\lfloor}{\rceil}
\xDeclarePairedDelimiter\floor{\lfloor}{\rfloor}
\xDeclarePairedDelimiter\ceil{\lceil}{\rceil}
\xDeclarePairedDelimiter\norm{\lVert}{\rVert}
\newcommand{\class}[1]{\ensuremath{\mathsf{#1}}\xspace}

\def\poly{\mathrm{poly}}

\newcommand{\FP}{\class{FP}}

\newcommand{\NP}{\class{NP}}
\newcommand{\FNP}{\class{FNP}}
\newcommand{\TFNP}{\class{TFNP}}

\newcommand{\PLS}{\class{PLS}}
\newcommand{\PPAD}{\class{PPAD}}

\newcommand{\CLS}{\class{CLS}}

\newcommand{\eps}{\varepsilon}

\newcommand{\dist}{\textup{dist}}
\newcommand{\len}{\mathrm{len}}

\DeclareMathOperator*{\argmax}{argmax}
\DeclareMathOperator*{\argmin}{argmin}

\def\eol/{\textup{\textsc{End-of-Line}}}
\def\clo/{\textup{\textsc{Continuous-Localopt}}}
\def\gclo/{\textup{\textsc{General-Continuous-Localopt}}}
\def\kkt/{\textup{\textsc{KKT}}}
\def\gdls/{\textup{\textsc{GD-Local-Search}}}
\def\gdfp/{\textup{\textsc{GD-Fixpoint}}}
\def\gbrouwer/{\textup{\textsc{General-Brouwer}}}
\def\grlo/{\textup{\textsc{General-Real-Localopt}}}
\def\linclo/{\textup{\textsc{Linear-Continuous-Localopt}}}
\def\iter/{\textup{\textsc{Iter}}}
\def\stationary/{\textup{\textsc{Stationary}}}

\theoremstyle{definition}
\newtheorem{definition}{Definition}[section]
\newtheorem{remark}{Remark}

\theoremstyle{plain}
\newtheorem{theorem}{Theorem}[section]

\newtheorem{lemma}{Lemma}
\newtheorem{corollary}{Corollary}
\newtheorem{claim}{Claim}

\title{The Computational Complexity of Finding Stationary Points in Non-Convex Optimization\thanks{An extended abstract appeared in the proceedings of COLT 2023.}}

\author{
\begin{tabular}{c c c}
& \hspace{1.7cm} & \\ \textbf{Alexandros Hollender} & & \textbf{Manolis Zampetakis}\\
\small{University of Oxford} & & \small{Yale University}\\
\href{mailto:alexandros.hollender@cs.ox.ac.uk}{\small{\texttt{alexandros.hollender@cs.ox.ac.uk}}} & & \href{mailto:emmanouil.zampetakis@yale.edu}{\small{\texttt{emmanouil.zampetakis@yale.edu}}}
\end{tabular}
}

\date{}

\begin{document}

\maketitle

\begin{abstract}
Finding approximate stationary points, i.e., points where the gradient is approximately zero, of non-convex but smooth objective functions $f$ over unrestricted $d$-dimensional domains is one of the most fundamental problems in classical non-convex optimization. Nevertheless, the computational and query complexity of this problem are still not well understood when the dimension $d$ of the problem is independent of the approximation error. In this paper, we show the following computational and query complexity results:
  \begin{enumerate}
    \item The problem of finding approximate stationary points over unrestricted domains is $\PLS$-complete.
    \item For $d = 2$, we provide a zero-order algorithm for finding $\eps$-approximate stationary points that requires at most $O(1/\eps)$ value queries to the objective function. 
    \item We show that any algorithm needs at least $\Omega(1/\eps)$ queries to the objective function and/or its gradient to find $\eps$-approximate stationary points when $d=2$. Combined with the above, this characterizes the query complexity of this problem to be $\Theta(1/\eps)$.
    \item For $d = 2$, we provide a zero-order algorithm for finding $\eps$-KKT points in constrained optimization problems that requires at most $O(1/\sqrt{\eps})$ value queries to the objective function. This closes the gap between the works of \cite{BubeckM20-trap-gradient} and \cite{Vavasis93-local-min} and characterizes the query complexity of this problem to be $\Theta(1/\sqrt{\eps})$.
    \item Combining our results with the recent result of \cite{FearnleyGHS22-gradient}, we show that finding approximate KKT points in constrained optimization is reducible to finding approximate stationary points in unconstrained optimization but the converse is impossible.
  \end{enumerate}
\end{abstract}

\section{Introduction}

One of the most fundamental problems in optimization is the following unconstrained problem which arises in the whole spectrum of scientific research and engineering and is an essential part of modern machine learning systems (\cite{boyd2004convex, nesterov2003introductory})
\begin{equation} \label{eq:problemDefinition}
  \min_{x \in \R^d} f(x).
\end{equation}
Depending on the structure of $f$, the above optimization problem takes different forms. In this paper, we focus on the case where $f$ is possibly non-convex, it has Lipschitz gradients, i.e., is \textit{smooth}, and its value is lower bounded. This formulation of \eqref{eq:problemDefinition} corresponds to one of the standard formulations of optimization problems in non-convex optimization. The global optimization of \eqref{eq:problemDefinition} is intractable even if we allow approximation errors (\cite{nemirovskij1983problem}). Following the classical paradigm, we consider the problem of finding \textit{approximate stationary points (or $\eps$-stationary points)} of \eqref{eq:problemDefinition}, i.e., we want to find $x^{\star} \in \R^d$ such that
\begin{equation} \label{eq:stationaryDefinition:intro}
  \norm{\nabla f\p{x^{\star}}}_2 \le \eps \quad \quad \text{ where $\quad \eps > 0$.}
\end{equation}

For this problem, it is well-known that gradient descent with appropriate but constant step size converges to an $\eps$-stationary point $x^{\star}$ after $O(1/\eps^2)$ iterations, if we assume that $f$ has constant smoothness, e.g., see \cite{nesterov2003introductory}.

Despite the significance of this problem, essentially no lower bounds were known until recently when \cite{CarmonDHS20-stationary} showed that if $d > 1/\eps^2$ then any algorithm that only has query access to $f$ and $\nabla f$ needs time $\Omega(1/\eps^2)$ to find an $\eps$-stationary point. This breakthrough result settles the query complexity of finding approximate stationary points when the number of dimensions grows faster that $1/\eps^2$.
\medskip

In this paper our goal is to understand the complexity of finding stationary points when the condition $d \ge 1/\eps^2$ does not hold. In particular, we want to answer the following two questions;

\begin{description}
  \item[Question 1:] \textit{What is the computational complexity of finding $\eps$-stationary points for any $d > 1$?}
  \item[Question 2:] \textit{What is the query complexity of finding approximate stationary points when $d = 2$?}
\end{description}

We note that for $d=1$, the problem can be solved using only $O(\log(1/\eps))$ function value queries \citep{chewi23a}. There has also been a lot of interest in the constrained version of the problem \eqref{eq:problemDefinition}, which corresponds to the following constrained optimization task
\begin{equation} \label{eq:problemDefinition:constrained}
  \min_{x \in [0, 1]^d} f(x).
\end{equation}
For this problem, the corresponding notion of approximate stationary points is finding \textit{approximate KKT points}. From the definition of stationary points and KKT points, it is not clear if finding stationary points in unconstrained optimization or finding KKT points in constrained optimization is harder, since the solution concepts are different. For this constrained problem we know the following:
\begin{itemize}
  \item \cite{FearnleyGHS22-gradient}: finding approximate KKT points is $\CLS$-complete.
  \item \cite{Vavasis93-local-min}: finding $\eps$-approximate KKT points in $d = 2$ requires at least $\Omega(1/\sqrt{\eps})$ value and gradient queries.
  \item \cite{BubeckM20-trap-gradient}: there exist an algorithm for finding $\eps$-approximate KKT points in $d = 2$ that requires at most $O((1/\sqrt{\eps})\sqrt{\log(1/\eps)})$ value and gradient queries.
\end{itemize}
Although the above results draw an almost complete picture for answering Questions 1 and 2 in the constrained optimization setting the following two questions still remain open.
\begin{description}
  \item[Question 3:] \textit{Is there an algorithm with complexity $O(1/\sqrt{\eps})$ for finding $\eps$-approximate KKT points in constrained optimization for $d = 2$?}
  \item[Question 4:] \textit{What is the relationship between finding approximate KKT points in constrained optimization and finding approximate stationary points in unconstrained optimization?}
\end{description}

In this paper we provide a complete answer for Questions 1 and 2 in unconstrained optimization and Questions 3 and 4 for constrained optimization.

\subsection{Our Results} \label{sec:results}

Our two main results are the following:
\begin{enumerate}
  \item \textbf{$\PLS$-completeness (Theorem \ref{thm:completeness}).} We show that finding $\eps$-stationary points of smooth and bounded functions in unconstrained optimization is $\PLS$-complete for any $d > 1$.
  
  \item \textbf{Unconstrained algorithm for $d = 2$ (Theorem \ref{thm:algo-unconstrained}).} For $d = 2$ we provide a zero-order algorithm that finds $\eps$-stationary points of smooth and bounded functions in unconstrained optimization using at most $O(1/\eps)$ value queries to the objective. This result resolves a recent open problem of \cite{chewi23a}.
\end{enumerate}

Once we establish these results we show that our $\PLS$-completeness result has the following important corollaries.
\begin{enumerate}
  \item[(a)] \textbf{Query lower bound for $d = 2$. (Theorem \ref{thm:query2D}).} We show that any algorithm for finding $\eps$-stationary points of smooth and bounded functions in unconstrained optimization when $d = 2$ needs at least $\Omega(1/\eps)$ value and/or gradient queries.
  
  \item[(b)] \textbf{Constrained vs Unconstrained optimization (Corollary \ref{cor:unconstrainedVSconstrained}).} It is possible to reduce finding approximate KKT points in constrained optimization to finding approximate stationary points in unconstrained optimization but the converse is impossible.
\end{enumerate}

  Additionally, the techniques that we use to show our second main result have the following corollary.
\begin{enumerate}
  \item[(c)] \textbf{Constrained algorithm for $d = 2$ (Theorem \ref{thm:algo-constrained}).} For $d = 2$ we provide a zero-order algorithm that finds $\eps$-KKT points of smooth and bounded functions in constrained optimization using at most $O(1/\sqrt{\eps})$ value queries to the objective. This result closes the gap between the upper bound of \cite{BubeckM20-trap-gradient} and the lower bound of \cite{Vavasis93-local-min}.
\end{enumerate}

The above results resolve our Questions 1 - 4. In particular: result 1. resolves Question 1, the combination of 2. and (a) resolve Question 2, (b) resolves Question 4, and (c) closes the gap of Question 3.

A particularly surprising result, at least for the authors, is that (b) provides the relationship between finding stationary points in unconstrained optimization and finding KKT points in constrained optimization. It is almost a reflex to believe that constrained optimization is harder than unconstrained optimization but the notion of stationary points in unconstrained optimization is provably harder than the notion of KKT points in constrained optimization. We can see this in two ways. First, finding approximate stationary points is $\PLS$-complete as per Theorem \ref{thm:completeness}, whereas finding approximate KKT points is $\CLS$-complete as per \cite{FearnleyGHS22-gradient} and $\CLS \subseteq \PLS$. Second, finding approximate KKT points is reducible to finding approximate stationary points in a black-box way, but the converse is in fact impossible. 

We believe that our results significantly increase our understanding of the complexity of finding stationary points in unconstrained optimization. Furthermore, we depart from the classical methods of showing lower bounds in optimization that often use the resisting oracle technique and we give one more example where tools from complexity theory can be utilized to provide new lower bounds in classical optimization problems, e.g, \cite{BabichenkoR21-congestion, DaskalakisSZ21-min-max, FearnleyGHS22-gradient}.

\subsection{Open Questions}

The following couple of open questions arise from our work:
\begin{enumerate}
  \item Characterize the query complexity of the problem (both in its unconstrained and constrained versions) in fixed dimension $d \geq 3$. In particular, any progress in dimension $d = 3$ is interesting.
  \item Characterize the query and computational complexity of higher order methods, both in the unconstrained and in the constrained setting.
\end{enumerate}

\section{Preliminaries} \label{sec:model}

\noindent \textbf{Notation.} We define $[[n, m]]$ for $n \le m$, to be the set of all integers between $n$ and $m$. We use $e_i$ to represent the $i$th unit vector in the standard basis of $\R^d$. We use $\ell(x, y)$ to denote the line segment between two points $x, y \in \R^d$. For any number  $z \in \R$ we use $\len(z)$ to denote the number of bits in the binary representation of $z$. We can extend the domain of $\len$ to vectors as follows: let $x \in \R^d$, then $\len(x) = \sum_{i = 1}^d \len(x_i)$.
\medskip

\noindent Our main object of study is a real-valued, non-convex, but smooth function $f : \R^d \to \R$. We assume that we have access to $f$ in two different ways:
\begin{itemize}[leftmargin=15pt]
  \item \textbf{Black Box Model.} This is the classical model used in optimization theory where we assume that for every point $x \in \R^d$ in the space there is a way to evaluate $f(x)$ and/or $\nabla f(x)$. If we only have access to $f(x)$ then we say that we have \textit{zero-order black box access} to $f$, whereas if we have access to both $f(x)$ and $\nabla f(x)$ we say that we have \textit{first-order black box access} to $f$.
  \item \textbf{White Box Model.} This is the classical model used in complexity theory to characterize the computational complexity of optimization problems. In this model we are given the description of a polynomial-time Turing machine $\calC_f$ that computes $f(x)$ and $\nabla f(x)$. More precisely, given some input $x \in \R^d$, described using $b$ bits, $\calC_f$ runs in time upper bounded by some polynomial in $b$ and outputs approximate values for $f(x)$ and $\nabla f(x)$. We note that a running time upper bound on a given Turing machine can be enforced syntactically by stopping the computation and outputting a fixed output whenever the computation exceeds the bound. See also Section 2 of \cite{DaskalakisSZ21-min-max} for more details about how to formally study the computational complexity of problems that take as input a polynomial-time Turing machine.
\end{itemize}

\noindent \textbf{Lipschitzness, Smoothness, and Normalization.} Our main objects of study are continuously differentiable Lipschitz, smooth, and bounded functions $f : \R^d \to \R$. A continuously differentiable function $f$ is called \textit{$L$-Lipschitz} if $\abs{f(x) - f(y)} \le L \norm{x - y}_2$, for all $x, y \in \R^d$, \textit{$L$-smooth} if $\norm{\nabla f(x) - \nabla f(y)}_2 \le L \norm{x - y}_2$, for all $x, y \in \R^d$, and \textit{$B$-bounded} is $\abs{f(x)} \le B$ for all $x \in \R^d$. In the black box model we will assume that the function satisfies these properties. In the white box model we will allow for violation solutions to handle cases where the properties are not satisfied.
\smallskip

\noindent \textbf{Approximate Stationary points.} As in classical unconstrained non-convex optimization, given the function $f : \R^d \to \R$ our goal is to find a point $x^{\star}$ such that $\norm{\nabla f(x^{\star})}_2 \le \eps$ for some given parameter $\eps > 0$. We call such a point $x^{\star}$ an \textit{$\eps$-stationary point of $f$}.
\medskip

In this paper we have two tasks: (1) to characterize the computational complexity of finding an $\eps$-stationary point when $f$ is given in the white box model, and (2) find tight upper and lower bounds on the number of oracle calls to $f$ that are needed when $d = 2$ and $f$ is given in the black box model. For each one of these results we need some definitions.

\subsection{Complexity of Total Search Problems and the Classes PLS and CLS} \label{sec:model:classes}

An important property of finding stationary points of a function $f$ is that if $f$ is bounded and smooth then there always exists an $\eps$-stationary point of $f$ with bit representation $\poly \log(B, L, 1/\eps)$, as we will see in Section \ref{sec:inclusion}. From this property, it is clear that the correct complexity landscape to study this problem is that of total search problems that is captured by the subclasses of the class $\TFNP$ which we define in Appendix \ref{app:tfnp}. In particular, we are interested in the complexity classes $\PLS$ as defined in \cite{JohnsonPY88-PLS, DaskalakisP11-CLS}. $\PLS$, which stands for Polynomial Local Search, is defined as the set of all $\TFNP$ problems that reduce in polynomial time to the following problem.

\begin{tcolorbox}
	\begin{definition}
		\textsc{Localopt}:
		
		\noindent\textbf{Input}: Circuits $N,P : [2^n] \to [2^n]$ ($N$ is the neighbor function and $P$ the potential function).
		
		\noindent\textbf{Goal}: Find $v \in [2^n]$ such that $P(N(v)) \geq P(v)$.
	\end{definition}
\end{tcolorbox}

\noindent In this paper, we make use of the following total search problem that is known to also characterize the complexity class $\PLS$ \citep{Morioka01-Mthesis-PLS}. See also Appendix \ref{app:complete:stationary} for more details.
\smallskip

\begin{tcolorbox}
	\begin{definition}\label{def:iter}
		\iter/:
		
		\noindent\textbf{Input}: Boolean circuit $C : [2^n] \to [2^n]$ with $C(1) > 1$.
		
		\noindent\textbf{Goal}: Find $v$ such that either
		\begin{itemize}
			\item $C(v) < v$, or
			\item $C(v) > v$ and $C(C(v)) = C(v)$.
		\end{itemize}
	\end{definition}
\end{tcolorbox}

\noindent Next we define the computational problem that we are interested in for the white box model.
\medskip

\begin{tcolorbox}[breakable,enhanced]
	\begin{definition}\label{def:kkt-general-inf}
		\stationary/:
		
		\noindent\textbf{Input}:
		\begin{itemize}
			\item precision parameter $\varepsilon > 0$,
			\item Turing machines $\calC_f$ and $\calC_{\nabla f}$ representing $f: \mathbb{R}^d \to \mathbb{R}$ and $\nabla f: \mathbb{R}^d \to \mathbb{R}^d$,
			\item a boundedness constant $B > 0$, and a smoothness constant $L > 0$.
		\end{itemize}
		
		\noindent\textbf{Goal}: Find $x^{\star} \in \R^d$ such that $\norm{\nabla f(x^{\star})}_2 \le \eps$.
        \smallskip
		
		\noindent Alternatively, we also accept one of the following violations as a solution:
		\begin{itemize}
			\item $x,y \in \R^d$ that violate the Lipschitzness or smoothness of $f$,
            \item $x\in \R^d$ such that $|f(x)| > B,$
			\item $x,y \in \R^d$ that certify that $\calC_{\nabla f}$ computes $\nabla f$ incorrectly.
		\end{itemize}
	\end{definition}
\end{tcolorbox}

\noindent \textbf{\bf Violations and promise-preserving reductions.} There is no known way of syntactically enforcing the Turing machines $\calC_f$, $\calC_{\nabla f}$ to be Lipschitz-continuous and bounded. Thus, to ensure that our problem indeed lies in $\TFNP$, we also allow solutions witnessing violations of these properties. This ``trick'' was also used by \cite{DaskalakisP11-CLS} for the definition of \CLS. Nevertheless, we note that our hardness result for \stationary/ also applies to the promise version of the problem, where we are promised that the input satisfies the properties.

\medskip

\subsection{Query Complexity in the Black Box Model} \label{sec:model:blackBox}

In the black box model we are only interested in the query complexity of the problem. This means that we are interested in understanding how many queries to $f$ and $\nabla f$ an algorithm has to make before it finds an $\eps$-stationary point. Fix a deterministic algorithm $\calA$, an initialization $x_0 \in \R^d$, and a bounded, Lipschitz and smooth function $f$. We define $\calT(\calA, x_0; f)$ to be the number of queries that the deterministic algorithm $\calA$ has to make to $f$ and $\nabla f$ before it outputs an $\eps$-stationary point of $f$.

\begin{remark}[Query lower bound for randomized algorithms.] In this paper we focus for simplicity on proving query lower bounds for deterministic algorithms. Nevertheless, we want to mention that the techniques of Section 6 of \cite{BubeckM20-trap-gradient} can be used to show that our Theorem \ref{thm:query2D} holds even for randomized algorithms with a polylogarithmic loss in the query complexity.
\end{remark}

\section{PLS-Completeness of Finding Stationary Points in Unconstrained Optimization} \label{sec:PLS-completeness}

In this section we characterize the computational complexity of the optimization problem \stationary/. The machinery that we use to prove Theorem \ref{thm:completeness} can be also utilized to show the tight query lower bound for \stationary/ in 2-dimensions, see Theorem \ref{thm:query2D}.

\begin{theorem} \label{thm:completeness}
  The problem \stationary/ is $\PLS$-complete, for any $d\ge2$. The $\PLS$-hardness of \stationary/ holds even if we are promised that the input satisfies the boundedness, the Lipschitzness, and the smoothness properties. Finally, our hardness reduction is black-box preserving.
\end{theorem}

If we combine Theorem \ref{thm:completeness} with the $\CLS$-completeness result of finding KKT points in constrained optimization we get the following very interesting corollary for the relation between stationary points in constrained and unconstrained optimization.

\begin{corollary} \label{cor:unconstrainedVSconstrained}
  There is an efficient black-box reduction from finding approximate KKT points in constrained optimization to finding approximate stationary points in unconstrained optimization. On the other hand, there is no efficient black-box reduction from finding approximate stationary points in unconstrained optimization to finding approximate KKT points in constrained optimization.
\end{corollary}

\begin{proof}
  This is an immediate corollary of the following three results: Theorem \ref{thm:completeness}, the $\CLS$ completeness of finding approximate KKT points of \cite{FearnleyGHS22-gradient}, and the black-box separation $\PLS \nsubseteq \CLS$, which follows from $\PLS \nsubseteq \PPAD$ \citep{BureshM04-NP-search-problems,BussJ12-propositional-NP-search,GoosHJMPRT22-separations}.
\end{proof}

\noindent In Section \ref{sec:inclusion} we show that \stationary/ is in $\PLS$, and in Section \ref{sec:hardness} that it is $\PLS$-hard. Finally, in Section \ref{sec:tight2D} we show a black-box query complexity result that we get from Theorem \ref{thm:completeness} when $d = 2$.

\subsection{Membership in $\PLS$} \label{sec:inclusion}

  The idea for showing the membership of \stationary/ in $\PLS$ is very similar with the well-known argument that gradient descent finds and $\eps$-stationary point in $O(B \cdot L/\eps^2)$ steps. The argument is that if we start from a point $x$ such that $\norm{\nabla f(x)}_2 \ge \eps$ and we apply a gradient descent step with step size $1/L$ then the function value $f(x')$ in the new point $x'$ has to be at least $\eps^2/L$ smaller than $f(x)$. Since $f$ is $B$-bounded this can only happen $O(B \cdot L/\eps^2)$ times and hence within this number of steps gradient descent has to find an $\eps$-stationary point.

  The above argument suggests that the value $f(x)$ can be used as a potential function and the gradient descent update as a neighboring function to show the membership of \stationary/ in $\PLS$. The issue with that is that our domain is unbounded and hence it is not possible at first glance to define a finite domain which would be necessary in order to provide a reduction to \textsc{LocalOpt}. Of course, taking a closer look at the gradient descent argument, we observe that we can provide an upper bound on the distance $R$ that gradient descent needs to travel before it reaches an $\eps$-stationary point. Unfortunately, this is still not enough since truncating all the points that have distance more than $R$ will create fallacious solutions on the boundary.

  To resolve this we need one more idea: focus only on points whose distance from the initial point and their value suggest that they could be points visited by gradient descent starting from the origin. In this way we are able to show that there are no ad-hoc solutions created on the boundary, which proves the following lemma. The complete proof of the lemma is presented in Appendix \ref{app:inclusion}.

  \begin{lemma} \label{lem:inclusion}
    It holds that \stationary/$ \in \PLS$. 
  \end{lemma}

\begin{remark}
  In Appendix \ref{app:inclusion} we in fact prove that \stationary/ is in $\PLS$ even when we are only given the Turing machine $\calC_f$, and we do not require a Turing machine $\calC_{\nabla f}$ that computes the gradient, assuming that we have a promise that the smoothness of $f$ holds. Moreover, for the $\PLS$-membership we do not need the boundedness assumption $\abs{f(x)} \le B$ but we only need that $f(0) - f(x) \le B$ for all $x \in \R^d$, where $0$ can be replaced by any fixed initial point $x_0$. 
\end{remark}

\subsection{$\PLS$-hardness} \label{sec:hardness}

  In this section we show that the problem \stationary/ is $\PLS$-hard even when the number of dimensions is $d = 2$. For $d = 1$ finding stationary points can be done in $O(\log(L B/ \eps))$ as was shown in \cite{chewi23a}. Hence, our result implies that for $d \ge 2$ there is no $O(\poly\log(L B/ \eps))$ algorithm for solving \stationary/ unless $\FP = \PLS$.

  \begin{lemma} \label{lem:hardness}
    The problem \stationary/ is $\PLS$-hard. Moreover, the $\PLS$-hardness of \stationary/ holds even if we are promised that the boundedness, the Lipschitzness, and the smoothness of the input of \stationary/ hold. Finally, our hardness reduction is black-box preserving.
  \end{lemma}  

  We provide the full proof of Lemma \ref{lem:hardness} in Appendix \ref{app:hardness}. Below we provide a high-level description of our construction and mention the key components of our proof.
  \medskip
  
  \noindent \textit{High-level Proof Sketch of Lemma \ref{lem:hardness}.} We construct $f$ periodically, i.e., we define a function $g$ over a square $[0, M]^2$ and then to get $f$ we repeat copies of $g$ in the whole plane $\R^2$ as shown in Figure \ref{fig:periodic}. If $g$ satisfies some boundary properties then this repetition of $g$ yields a bounded and smooth function $f$. So, we can now focus on the construction of $g$.

  \begin{figure}[t]
    \centering
    \includegraphics[scale=0.3]{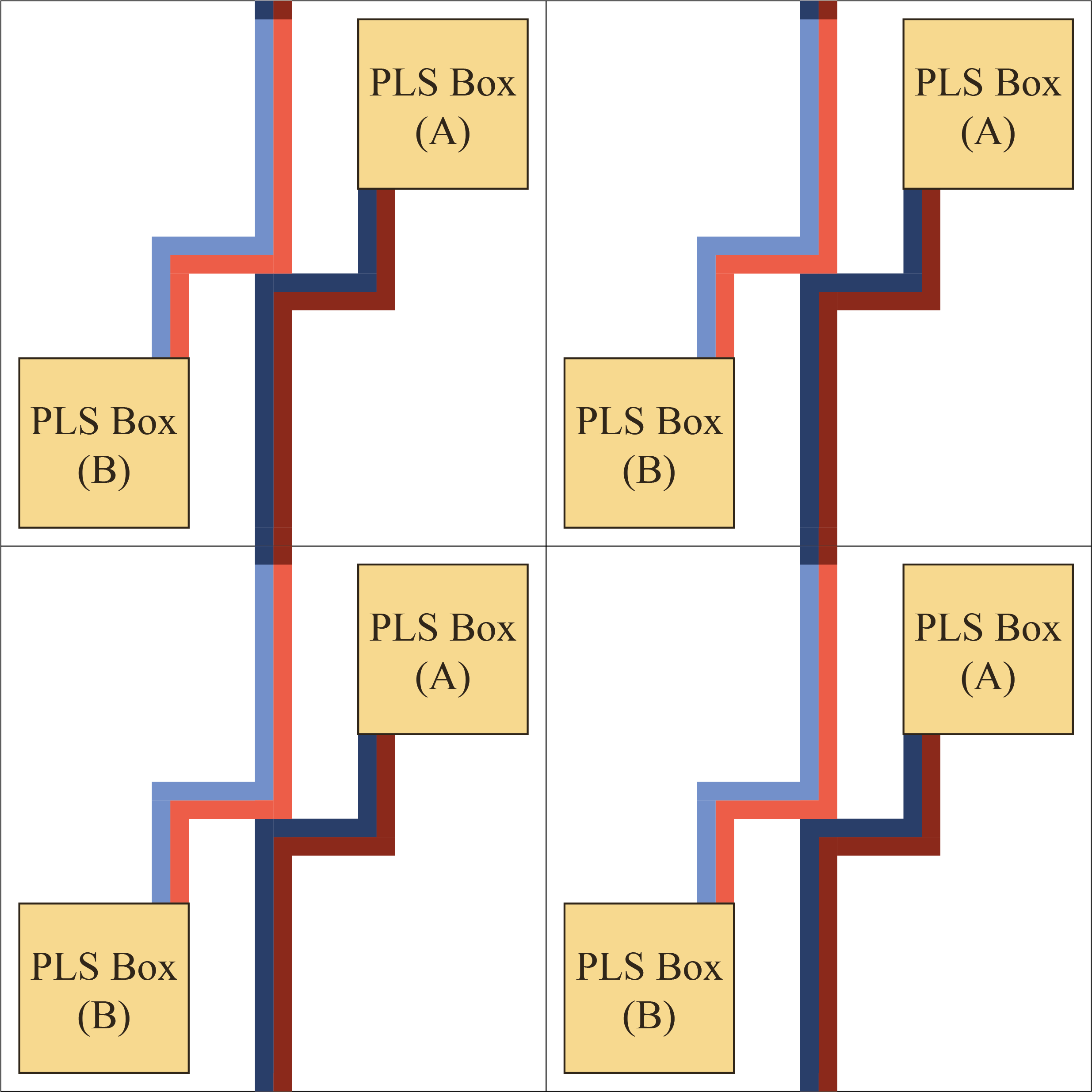}
    \caption{High-level structure of the construction of $f$ in our hardness result of Lemma \ref{lem:hardness}.}
    \label{fig:periodic}
  \end{figure}

  The function $g$ is a continuously differentiable function that is defined over $[0, M]^2$ which means that it will attain a minimum at some point $x_{\min}$ and a maximum at some point $x_{\max}$. Obviously we have to have that $\nabla g(x_{\min}) = 0$ and $\nabla g(x_{\max}) = 0$. Hence we want to ensure two things: (1) that $x_{\min}$ and $x_{\max}$ are in places that correspond to solutions of an \iter/ instance, and (2) all stationary points correspond to local minima and local maxima. For this, we place two boxes in the construction of $g$: the PLS Box (A) and the PLS Box (B) (see Figure \ref{fig:periodic}). The PLS Box (A) is the only place where local minima can be formed and the PLS Box (B) is the only place where local maxima can be formed. We also want to make sure that no solution can arise in the space outside PLS Box (A) and PLS Box (B). To do that we use a structure shown in Figure \ref{fig:periodic} where the blue colors correspond to low function values, the red colors to high function values and the background has colors between blue and red. Both blue and red values are decreasing as each one of the coordinates increase we move up and to the right and the background decreases as we go from left to right. Hence:
  \begin{itemize}[noitemsep,topsep=0pt,parsep=0pt,partopsep=0pt]
    \item[$\triangleright$] If we start from any point in the background and we are looking for a local minimum then we will move to the right until we hit the blue region. Once we are inside the blue region we will move up and to the right until we are inside the PLS Box (A).
    \item[$\triangleright$] Similarly, if we start from any point in the background and we are looking for a local maximum then we will move to the left until we hit the red region. Once we are inside the red region we will move down and to the left until we are inside the PLS Box (B).
  \end{itemize}
  The above observation gives an intuition about why local minima can appear only in PLS Box (A) and local maxima only in PLS Box (B) in our construction as well as about why we should not expect any stationary points outside of these PLS boxes. We formally show these properties about the space outside the PLS boxes in Appendix \ref{sec:structureOfG}.

  \noindent \textbf{PLS Boxes.} This is where we use the \iter/ instance\footnote{Our PLS boxes have many similarities with the PLS Labyrinths of \cite{FearnleyGHS22-gradient}. Nevertheless, our construction is more challenging because of a different background that we need to use since we define $f$ on the whole plane.}. At a very high level, starting from an \iter/ instance we move the blue-red in the bottom left corner of PLS Box (A) following the paths of the circuit $C$ and in this way we make sure that when the blue-red line ends, which is the only place where a local minimum can be formed, will be in places where the \iter/ problem with input $C$ has a solution. The construction of PLS Box (B) is similar except that the orientation of the box is rotated by $180^{o}$. We give the formal construction of the PLS boxes in Appendix \ref{sec:PLSBoxes}.

  In the above sketch we skipped an important part of our proof which is how to construct this piecewise function while satisfying smoothness of $f$. For this we use bi-cubic interpolation techniques (see Appendix \ref{sec:interpolation}) and we need to make sure that all the regions that we construct satisfy the conditions to apply bi-cubic interpolation without creating stationary points in places other than local minima and local maxima.

  We defer all remaining details of the proof of Lemma \ref{lem:hardness} to Appendix \ref{app:hardness}.

\subsection{Tight Query Bounds for 2D} \label{sec:tight2D}

As a corollary of the proof of Lemma \ref{lem:hardness} we can show the following black box lower bound for finding $\eps$-stationary points in two dimensions.

\begin{theorem} \label{thm:query2D}
  For any deterministic algorithm $\calA$ that computes $\eps$-stationary points and any starting point $(x_0, y_0) \in \R^2$ it holds that there exists a $1$-bounded, $1$-Lipschitz, and $1$-smooth function $f : \R^2 \to \R$, such that
  $\calT(\calA, (x_0, y_0); f) \ge  \Omega\p{\frac{1}{\eps}}$.
\end{theorem}

We prove Theorem \ref{thm:query2D} in Appendix \ref{app:query2D}. The proof is based on the proof of Lemma \ref{lem:hardness}. In particular, it is not hard to see that if we only have black box access to \iter/ then any algorithm would need at least $2^n$ time to find a solution. If we combine this tight bound for \iter/ together with the proof of Lemma \ref{lem:hardness} we obtain Theorem \ref{thm:query2D}.

In the next section we present an algorithm with complexity $O(1/\eps)$  which combined with Theorem \ref{thm:query2D} resolves the black box query complexity of finding stationary points when $d = 2$.

\section{The Gradient Flow Parallel Trap Algorithm}

In this section we present the Gradient Flow Parallel Trap (GFPT) algorithm for computing $\eps$-stationary points in both the unconstrained and constrained settings. This algorithm is inspired by the Gradient Flow Trapping (GFT) algorithm proposed by \citet{BubeckM20-trap-gradient} for the constrained setting. For $d=2$ their GFT algorithm yields an upper bound that is almost tight, namely up to $\log(1/\eps)$ factors. The GFPT algorithm we propose here uses some core ideas from the GFT algorithm to achieve a \emph{tight} upper bound for $d=2$. Furthermore, GFPT is in fact simpler than GFT: while GFT relies on two subroutines (\emph{parallel trap} and \emph{edge fixing}), GFPT only uses an improved version of one of the two subroutines (namely, \emph{parallel trap}), without the need for the second subroutine.

We first state our result in the unconstrained setting, which is our main focus in this paper.

\begin{theorem}\label{thm:algo-unconstrained}
Let $d \geq 2$ and $f: \mathbb{R}^d \to [0, \infty)$ be such that $\nabla f$ is $L$-Lipschitz-continuous.
For any $\eps > 0$, the GFPT algorithm with starting point $x_0 \in \mathbb{R}^d$ returns an $\eps$-stationary point using at most
$$O(d)^{\frac{5d-1}{4}} \left(\frac{\sqrt{L f(x_0)}}{\eps}\right)^{d-1}$$
queries. In particular, for $d=2$ the number of queries is $O(\sqrt{L f(x_0)}/\eps)$.
\end{theorem}

Note that if the co-domain is also bounded from above, i.e., $[0,1]$ instead of $[0, \infty)$, then the bound becomes $(\sqrt{L}/\eps)^{d-1}$ for any fixed $d$.

The GFPT algorithm also applies to the constrained setting, which is the setting in which the GFT algorithm was originally stated by \citet{BubeckM20-trap-gradient}.

\begin{theorem}\label{thm:algo-constrained}
Let $d \geq 2$ and $f: [0,1]^d \to \mathbb{R}$ be such that $\nabla f$ is $L$-Lipschitz-continuous on $[0,1]^d$.
For any $\eps > 0$, the GFPT algorithm returns an $\eps$-KKT point (w.r.t. minimization) using at most
$$O(d)^{\frac{5d-1}{4}} \left(\sqrt{\frac{L}{\eps}}\right)^{d-1}$$
queries. In particular, for $d=2$ the number of queries is $O(\sqrt{L/\eps})$, and for $d=3$ it is $O(L/\eps)$.
\end{theorem}

In the next section we present a high-level overview of the algorithm. Then, we proceed with the formal presentation and proof for the unconstrained setting. Finally, we briefly mention how the algorithm can be adapted to the constrained setting.

\subsection{Overview of the GFPT algorithm}

For this overview we consider the 2-dimensional unconstrained setting. In other words, we assume that we have query access to a function $f: \mathbb{R}^2 \to [0, \infty)$, which has $L$-Lipschitz-continuous gradient $\nabla f$. Recall that our goal is to find an $\eps$-stationary point, i.e., a point $x \in \mathbb{R}^d$ such that $\|\nabla f(x)\|_2 \leq \eps$.

\paragraph{Gradient flow.}
The notion of the \emph{gradient flow} is very useful in order to understand the intuition behind the algorithm. Intuitively, the gradient flow is a continuous path that corresponds to the points that gradient descent would visit if it had an infinitesimally small step size. More formally, the gradient flow starting at $x$ is the path $\gamma(t)$ that is the solution of the differential equation $\gamma'(t) = -\nabla f(\gamma(t))$ with initial condition $\gamma(0) = x$.

\paragraph{Initialization.}
Let $x_0 \in \mathbb{R}^d$ be some starting point. If $f(x_0) = 0$, then $x_0$ is a global minimum of the function and thus necessarily a stationary point, so we assume that $f(x_0) > 0$.

Consider the rectangle $R_0 = \{x \in \mathbb{R}^2: \|x-x_0\|_\infty \leq 2f(x_0)/\eps\}$. We claim that $R_0$ must contain an $\eps$-stationary point of $f$. Indeed, assume that $R_0$ does not contain any $\eps$-stationary points and consider the gradient flow starting at $x_0$. Then, as long as the gradient flow has not left $R_0$, the value of $f$ along the gradient flow must decrease by a rate at least $\eps$: if the gradient flow has traveled a distance $\ell$, then the function value has decreased by at least $\eps \cdot \ell$. In order to reach the boundary of $R_0$, the gradient flow must travel a distance at least $2f(x_0)/\eps$ from $x_0$. But this means that, when the gradient flow reaches $\partial R_0$, the function value will be at most $f(x_0) - \eps \cdot 2f(x_0)/\eps = -f(x_0) < 0$, which is impossible! As a result, it follows that $R_0$ must in fact contain an $\eps$-stationary point. See Figure~\ref{fig:gradient-flow-R0} for an illustration of this argument.

\begin{figure}
	\centering
	\begin{minipage}{.45\textwidth}
		\centering
\scalebox{0.6}{
\includegraphics{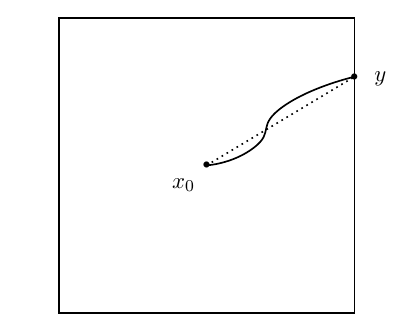}
}
\captionof{figure}{An illustration of the rectangle $R_0$ and the gradient flow starting at $x_0$.
}\label{fig:gradient-flow-R0}
	\end{minipage}\hfill
	\begin{minipage}{.45\textwidth}
		\centering
\scalebox{0.6}{
\includegraphics{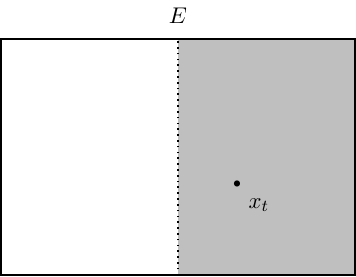}
}
\captionof{figure}{A simple gradient flow trap. The gray region is the new smaller rectangle $R_{t+1}$.}\label{fig:simple-gradient-flow-trap}
	\end{minipage}
\end{figure}

% \begin{tikzpicture}

% \clip (-1,-0.3) rectangle (6,5.3);

% \draw[thick] (0,0) rectangle ++(5,5);

% \draw[thick,dotted] (2.5,2.5) -- (5,4);

% \draw[thick] plot [smooth, tension=1] coordinates {(2.5,2.5) (3.3,2.8) (3.8,3.5) (5,4)};

% \node[label={[label distance=-0.2cm]225:$x_0$}] at (2.5,2.47) {\textbullet};

% \node[label={right:$y$}] at (5,3.96) {\textbullet};

% \end{tikzpicture}

% \begin{tikzpicture}

% \fill[lightgray] (3,0) rectangle ++(3,4);

% \draw[thick] (0,0) rectangle ++(6,4);

% \draw[thick, dotted] (3,0) -- ++(0,4);

% \node[label=above:{$E$}] at (3,4) {};

% \node[label={[label distance=-0.2cm]-45:$x_t$}] at (4,1.5) {\textbullet};

% \end{tikzpicture}

\paragraph{Invariant.}
The initialization step ensures that $R_0$ contains an $\eps$-stationary point. However, $R_0$ is too large for us to locate such an $\eps$-stationary point by brute force. Thus, the algorithm will seek to decrease the size of the rectangle at each iteration, while maintaining the invariant that it must contain a solution. After sufficiently many iterations, the rectangle will be small enough, so that the $\eps$-stationary point can easily be found.

More formally, at each iteration $t$, the algorithm will maintain a rectangle $R_t$ and a point $x_t \in R_t$, such that $R_{t+1}$ is smaller than $R_t$ by a constant fraction. The goal of the algorithm will be to maintain the following invariant: all points on the boundary of $R_t$ are $\eps$-unreachable from $x_t$. We say that a point $y$ is $\eps$-unreachable from a point $x$ if
$$f(y) > f(x) - \eps \|x-y\|_2.$$
The intuition for this definition is the following: if $y$ is $\eps$-unreachable from $x$, then the gradient flow starting at $x$ cannot reach $y$, unless it encounters an $\eps$-stationary point on the way. Thus, if $(R_t,x_t)$ satisfies the invariant, then $R_t$ must necessarily contain an $\eps$-stationary point. As explained above, $(R_0,x_0)$ satisfies the invariant, so the algorithm now has to find a way to decrease the size of $R_t$ at each step while maintaining the invariant.

\paragraph{Simple gradient flow trap.}
Let us first consider a simple attempt at decreasing the size of $R_t$ which unfortunately fails. Let $E$ be the segment that cuts $R_t$ in half (along its longest side). Furthermore, assume that we can somehow determine that all points $y \in E$ are $\eps$-unreachable from $x_t$. In that case, we could simply set $x_{t+1} := x_t$ and let $R_{t+1}$ be the half of $R_t$ which contains $x$. It is easy to see that the invariant would be satisfied by $(R_{t+1},x_{t+1})$. See Figure~\ref{fig:simple-gradient-flow-trap} for an illustration.

There are multiple issues with this simple approach. One problem is that we have not said what happens when there exists some point $z \in E$ that is \emph{not} $\eps$-unreachable from $x_t$. A crucial observation, which we will also use later, is that in that case, all points $y \in \partial R_t$ are $\eps$-unreachable from $z$. Indeed, we can combine the inequalities $f(z) \leq f(x_t) - \eps \|x_t-z\|_2$ and $f(y) > f(x_t) - \eps \|x_t-y\|_2$ to obtain
$$f(y) > f(x_t) - \eps \|x_t-y\|_2 \geq f(z) + \eps \|x_t-z\|_2 - \eps \|x_t-y\|_2 \geq f(z) - \eps \|z-y\|_2$$
i.e., $y$ is $\eps$-unreachable from $z$. This means that the update $x_{t+1} := z$ and $R_{t+1} := R_t$ would maintain the invariant. Unfortunately, this update would not decrease the size of the rectangle...

A more fundamental issue is: how can we determine whether all points $y \in E$ are $\eps$-unreachable from $x_t$? This would require an infinite number of queries to $f$... Instead, we pick a sufficiently fine $\delta$-discretization $S$ of $E$ and only query $f$ on the points in $S$. What can we say now, if we find that all points in $S$ are $\eps$-unreachable from $x_t$? It turns out that as long as $x_t$ is sufficiently far away from $E$, it follows that all points in $E$ are $\eps'$-unreachable from $x_t$, for some $\eps' > \eps$. As a result, our algorithm will have to allow for $\eps$ to increase (hopefully, only by a small amount!) from one iteration to the next one. In other words, we will also maintain a value $\eps_t$ at each iteration and the invariant will now be: $(R_t,x_t,\eps_t)$ satisfies the invariant, if all points on the boundary of $R_t$ are $\eps_t$-unreachable from $x_t$. We will aim to ensure that $\eps_t$ does not increase too much, namely by at most a constant factor in \emph{total} over all iterations. For this, we will have to pick $\delta$ to be sufficiently large, but also ensure that $x_t$ is sufficiently far away from $E$. Unfortunately, this cannot be guaranteed using this simple gradient flow trap, which is why a parallel trap is needed, as already noted by \citet{BubeckM20-trap-gradient}. In fact, the parallel will also allow us to take care of the first issue identified above.

\paragraph{The parallel trap.} For concreteness, assume that the longest side of $R_t$ is along the first coordinate, and that it has length $r$. Instead of considering a single segment $E$ which cuts $R_t$ in half, we consider segments $E_1$ and $E_2$ that cut $R_t$ into three equal parts.

We pick sufficiently fine $\delta$-discretizations $S_1$ and $S_2$ of $E_1$ and $E_2$ respectively. If all points on $S_1$ and $S_2$ are $\eps_t$-unreachable from $x_t$, then it is easy to see that we can remove a third of $R_t$ while maintaining the invariant without changing $x_t$. If $x_t$ lies in the right half of $R_t$, then we remove everything on the left of $E_1$. Otherwise, we remove everything on the right of $E_2$. In both cases we let $x_{t+1} := x_t$. See Figure~\ref{fig:parallel-trap-x} for an illustration. Note that in both cases $x_t$ cannot be arbitrarily close to the new boundary of $R_{t+1}$: the distance is always at least $r/6$. This is important to ensure that we can control how much larger $\eps_{t+1}$ is than $\eps_t$.

\begin{figure}
\centering
\scalebox{0.6}{
\includegraphics{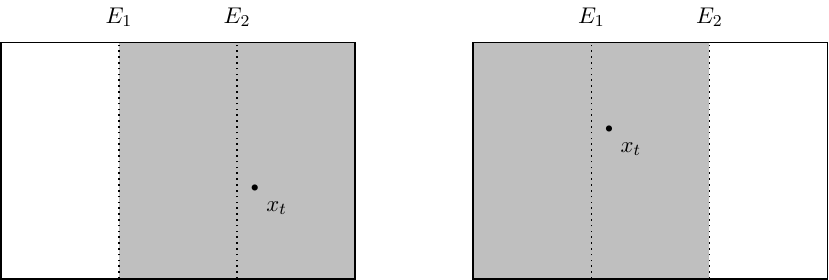}
}
% \begin{tikzpicture}
% \fill[lightgray] (2,0) rectangle ++(4,4);

% \fill[lightgray] (8,0) rectangle ++(4,4);

% \draw[thick] (0,0) rectangle ++(6,4);

% \draw[thick, dotted] (2,0) -- ++(0,4);
% \draw[thick, dotted] (4,0) -- ++(0,4);

% \draw[thick] (8,0) rectangle ++(6,4);

% \draw[thick, dotted] (10,0) -- ++(0,4);
% \draw[thick, dotted] (12,0) -- ++(0,4);

% \node[label=above:{$E_1$}] at (2,4) {};
% \node[label=above:{$E_2$}] at (4,4) {};

% \node[label=above:{$E_1$}] at (10,4) {};
% \node[label=above:{$E_2$}] at (12,4) {};

% \node[label={[label distance=-0.2cm]-45:$x_t$}] at (4.3,1.5) {\textbullet};

% \node[label={[label distance=-0.2cm]-45:$x_t$}] at (10.3,2.5) {\textbullet};
% \end{tikzpicture}
\caption{The two possible cases when all points on $S_1$ and $S_2$ are $\eps$-unreachable from $x_t$. The gray region represents the new smaller rectangle $R_{t+1}$ in both cases.}\label{fig:parallel-trap-x}
\end{figure}

If, on the other hand, there exist points $z$ on $S_1$ or $S_2$ that are not $\eps_t$-unreachable from $x$, then let $z^*$ denote such a point with minimal value $f(z^*)$. For concreteness, assume that $z^*$ lies on $S_1$. By the crucial observation presented in the simple gradient flow trap attempt above, we know that all points on $\partial R_t$ are $\eps_t$-unreachable from $z^*$. Thus, we will set $x_{t+1} := z^*$, but we will also let $R_{t+1}$ be $R_t$ with everything on the right of $E_2$ removed. See Figure~\ref{fig:parallel-trap-z} for an illustration. This will ensure that the size of the rectangle decreases, but we still need to argue that this maintains the invariant. We will show that all points on $S_2$ are $\eps_t$-unreachable from $z^*$. Then, since $z^*$ is sufficiently far away from $E_2$, we will again be able to argue that $\eps_{t+1}$ is not much larger than $\eps_t$. First, consider points in $S_2$ that are $\eps_t$-unreachable from $x_t$. By the crucial observation, these points are also $\eps_t$-unreachable from $z^*$. It remains to consider points $z \in S_2$ that are not $\eps_t$-unreachable from $x_t$. But by construction of $z^*$ we have $f(z^*) \leq f(z)$ for all these points $z$. Together with the fact that $\|z^*-z\|_2 > 0$ (since $z^* \in E_1$ and $z \in E_2$) this implies
$$f(z) \geq f(z^*) > f(z^*) - \eps_t \|z^*-z\|_2$$
i.e., $z$ is $\eps_t$-unreachable from $z^*$.

\begin{figure}
\centering
\scalebox{0.6}{
\includegraphics{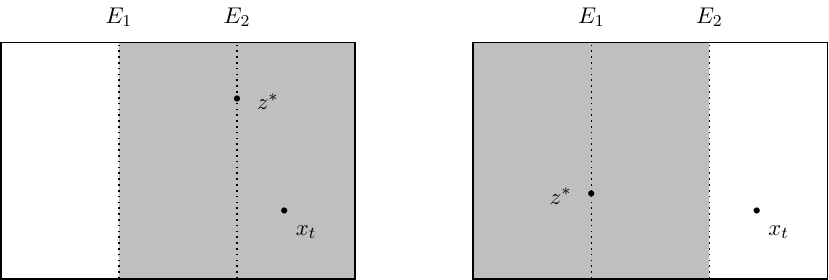}
}
% \begin{tikzpicture}
% \fill[lightgray] (2,0) rectangle ++(4,4);

% \fill[lightgray] (8,0) rectangle ++(4,4);

% \draw[thick] (0,0) rectangle ++(6,4);

% \draw[thick, dotted] (2,0) -- ++(0,4);
% \draw[thick, dotted] (4,0) -- ++(0,4);

% \draw[thick] (8,0) rectangle ++(6,4);

% \draw[thick, dotted] (10,0) -- ++(0,4);
% \draw[thick, dotted] (12,0) -- ++(0,4);

% \node[label=above:{$E_1$}] at (2,4) {};
% \node[label=above:{$E_2$}] at (4,4) {};

% \node[label=above:{$E_1$}] at (10,4) {};
% \node[label=above:{$E_2$}] at (12,4) {};

% \node[label={[label distance=-0.2cm]-45:$x_t$}] at (4.8,1.1) {\textbullet};

% \node[label={[label distance=-0.2cm]-45:$x_t$}] at (12.8,1.1) {\textbullet};

% \node[label={right:$z^*$}] at (4,3) {\textbullet};

% \node[label={left:$z^*$}] at (10,1.4) {\textbullet};
% \end{tikzpicture}
\caption{The two possible cases when there exist points on $S_1$ and $S_2$ that are not $\eps$-unreachable from $x_t$. The gray region represents the new smaller rectangle $R_{t+1}$ in both cases.}\label{fig:parallel-trap-z}
\end{figure}

\paragraph{Parameters.} In the formal analysis of the algorithm we show that it is possible to pick the size of the $\delta$-nets in each step so as to balance out the following two objectives: (i) $\delta$ is small enough so that, in total over all iterations, $\eps$ only increases by a constant factor, and (ii) $\delta$ is large enough, so that the total number of queries remains low, namely $O(\sqrt{L f(x_0)}/\eps)$ in the two-dimensional case.

\subsection{Formal presentation of GFPT}\label{sec:algo-unconstrained-proof}

Let $d \geq 2$ and let $f: \mathbb{R}^d \to [0,\infty)$ be such that $\nabla f$ is $L$-Lipschitz-continuous. We begin with some definitions and technical lemmas.

\begin{definition}
For $x,y \in \mathbb{R}^d$ and $\eps > 0$ we say that \emph{$y$ is $\eps$-unreachable from $x$} if the following holds:
\[ f(y) > f(x) - \eps \|x-y\|_2. \]
\end{definition}

Intuitively, $y$ is $\eps$-unreachable from $x$ if the value $f(y)$ is too high with respect to $f(x)$ so that $y$ cannot be reached by following the ``gradient flow'' starting at $x$ without encountering an $\eps$-stationary point along the way. We now formalize this intuition.

\begin{definition}
A $k$-dimensional hyperrectangle $R$ in $\mathbb{R}^d$ is a set of the form $R = [a_1,b_1] \times \dots \times [a_d,b_d]$, where $a_i \leq b_i$ for all $i \in [d]$, and $|\{i \in [d] : a_i < b_i\}| = k$. When $k = d$, we also say that $R$ is full-dimensional.
\end{definition}

The following lemma is essentially a more refined version of a corresponding result in \citep{BubeckM20-trap-gradient} and uses the same proof technique.

\begin{lemma}\label{lem:eps-unreachable}
Let $R$ be a full-dimensional hyperrectangle in $\mathbb{R}^d$ and $x$ a point in $R$. For any $\eps > 0$, if all $y \in \partial R$ are $\eps$-unreachable from $x$, then $R$ contains an $\eps$-stationary point of $f$.
\end{lemma}

\begin{proof}
The gradient flow for $f$ starting at $x$ is defined as the solution $\gamma(t)$ to the differential equation
$$\gamma'(t) = - \nabla f(\gamma(t))$$
with the initial condition $\gamma(0) = x$. The gradient flow is a curve that starts at $x$ and follows the direction of steepest descent as given by the negative of the gradient $\nabla f$. Since $\nabla f$ is Lipschitz-continuous, the existence and uniqueness of a solution $\gamma: [0,\infty) \to \mathbb{R}^d$ is guaranteed by standard tools from the theory of differential equations.

Define $\psi: t \mapsto f(\gamma(t))$ and note that $\psi$ is continuously differentiable with derivative $\psi'(t) = \langle \nabla f(\gamma(t)), \gamma'(t) \rangle$. For any $T \geq 0$ we can write
\begin{equation*}
\begin{split}
f(\gamma(T)) - f(\gamma(0)) = \psi(T) - \psi(0) = \int_0^T \psi'(t) dt &= \int_0^T \langle \nabla f(\gamma(t)), \gamma'(t) \rangle dt\\
&= \int_0^T \langle \nabla f(\gamma(t)), - \nabla f(\gamma(t)) \rangle dt\\
&= - \int_0^T \| \nabla f(\gamma(t)) \|_2^2 dt.
\end{split}
\end{equation*}
Assume towards a contradiction that the hyperrectangle $R$ does not contain any $\eps$-stationary points of $f$. In other words, $\| \nabla f(z) \|_2 > \eps$ for all $z \in R$. Now consider any $T \geq 0$ such that $\gamma(t) \in R$ for all $t \in [0,T]$. Then we have that
$$f(\gamma(T)) - f(\gamma(0)) = - \int_0^T \| \nabla f(\gamma(t)) \|_2^2 dt \leq - T \eps^2.$$
Since $f$ is continuous and $R$ is compact, $f$ is bounded on $R$ and thus there exists a maximal $T$ such that $\gamma(t) \in R$ for all $t \in [0,T]$. In particular, $y := \gamma(T)$ lies on the boundary $\partial R$. Furthermore, since the length of the curve can be computed as $\int_0^T \| \gamma'(t) \|_2 dt$, we deduce that $\int_0^T \| \gamma'(t) \|_2 dt \geq \|\gamma(0)-\gamma(T)\|_2 = \|x-y\|_2$. Finally, we obtain
\begin{equation*}
\begin{split}
f(y) - f(x) = f(\gamma(T)) - f(\gamma(0)) = - \int_0^T \| \nabla f(\gamma(t)) \|_2^2 dt &\leq - \eps \int_0^T \| \nabla f(\gamma(t)) \|_2 dt\\
&= - \eps \int_0^T \| \gamma'(t) \|_2 dt\\
&\leq - \eps \|x-y\|_2.
\end{split}
\end{equation*}
But this means that $y \in \partial R$ is not $\eps$-unreachable from $x$, a contradiction. As a result, $R$ must necessarily contain some $\eps$-stationary point.
\end{proof}

\begin{corollary}\label{cor:solution-unconstrained}
Let $R = [a_1,b_1] \times \dots \times [a_d,b_d]$ be a full-dimensional hyperrectangle in $\mathbb{R}^d$ and $x$ a point in $R$. For any $\eps > 0$, if all $y \in \partial R$ are $(\eps/2)$-unreachable from $x$, and $\max_i (b_i-a_i) \leq \frac{\eps}{2\sqrt{d}L}$, then $x$ is an $\eps$-stationary point of $f$.
\end{corollary}

\begin{proof}
By Lemma~\ref{lem:eps-unreachable}, $R$ contains an $\eps/2$-stationary point of $f$, i.e., there exists $y \in R$ with $\|\nabla f(y)\|_2 \leq \eps/2$. Since the gradient of $f$ is $L$-Lipschitz-continuous, we have that
$$\|\nabla f(x) - \nabla f(y)\|_2 \leq L \|x - y\|_2 \leq L \sqrt{d} \|x-y\|_\infty \leq L \sqrt{d} \max_i (b_i-a_i) \leq \eps/2$$
and thus $\|\nabla f(x)\|_2 \leq \eps$.
\end{proof}

The algorithm will make use of $\delta$-nets with some nice properties.

\begin{definition}[nice $\delta$-net]
Let $\delta > 0$ and let $R$ be a $k$-dimensional hyperrectangle in $\mathbb{R}^d$. A set of points $S \subseteq R$ is a \emph{nice $\delta$-net} of $R$ if for any face $F$ of $R$ it holds that $S \cap F$ is a $\delta$-net of $F$, i.e., for all $y \in F$ there exists $z \in S \cap F$ with $\|z-y\|_2 \leq \delta$.
\end{definition}

It is not hard to construct nice $\delta$-nets of reasonable size. We include the following construction for completeness.

\begin{lemma}\label{lem:nice-net}
Let $R = [a_1,b_1] \times \dots \times [a_d,b_d]$ be a $k$-dimensional hyperrectangle in $\mathbb{R}^d$. Then, for any $\delta > 0$, there exists a nice $\delta$-net $S$ of $R$ with $|S| = \prod_{i=1}^d (\lceil \sqrt{k}(b_i-a_i)/2\delta \rceil + 1)$. In particular, if $R$ is $(d-1)$-dimensional, then we have $|S| \leq (\sqrt{d}r/2\delta)^{d-1}$ for any $r$ satisfying $r \geq \max_i (b_i-a_i)$ and $r \geq 8 \sqrt{d} \delta$.
\end{lemma}

\begin{proof}
For each $i \in [d]$, we construct a set $S_i \subseteq [a_i,b_i]$ as follows. If $a_i = b_i$, we let $S_i := \{a_i\}$. Otherwise, namely when $a_i < b_i$, we partition the interval $[a_i,b_i]$ into $\lceil \sqrt{k}(b_i-a_i)/2\delta \rceil$ intervals of equal length, and we let $S_i \subset [a_i,b_i]$ denote the set of endpoints of these intervals. Note that we always have $|S_i| = \lceil \sqrt{k}(b_i-a_i)/2\delta \rceil + 1$ and $\{a_i,b_i\} \subseteq S_i$. Furthermore, by construction, the distance between subsequent points in $S_i$ is at most $2\delta/\sqrt{k}$. Thus, for any $p \in [a_i,b_i]$ there exists $q \in S_i$ such that $|p-q| \leq \delta/\sqrt{k}$.

We now prove that $S = S_1 \times \dots \times S_d$ is a nice $\delta$-net of $R$. Consider any point $y \in R$. We construct $z \in S$ as follows: for each $i \in [d]$, let $z_i = \argmin_{p \in S_i} |p-y_i|$ (pick an arbitrary minimizer, if it is not unique). First of all, note that we indeed have $z \in S$ by construction of $S$. Furthermore, by construction of $S_i$, we have that $|z_i-y_i| \leq \delta/\sqrt{k}$ for all $i \in [d]$. This implies that $\|z-y\|_2 \leq \delta$, since $R$ is $k$-dimensional and thus $z$ and $y$ disagree on at most $k$ coordinates. Finally, let $F$ be any face of $R$ that contains $y$. Since $\{a_i,b_i\} \subseteq S_i$, the construction ensures that $z_i = y_i$ whenever $y_i \in \{a_i,b_i\}$. Thus, we also have $z \in F$. In other words, $S \cap F$ is a $\delta$-net of $F$.

Now consider the case where $R$ is $(d-1)$-dimensional and where $r$ is some value satisfying $r \geq \max_i (b_i-a_i)$ and $r \geq 8 \sqrt{d} \delta$. Using the fact that $b_i-a_i \leq r$ for all i, and that $b_j - a_j = 0$ for some $j$, we obtain
\begin{equation*}
\begin{split}
|S| = \prod_{i=1}^d \left(\left\lceil \frac{\sqrt{d-1}(b_i-a_i)}{2\delta} \right\rceil + 1\right) \leq \left(\left\lceil \frac{\sqrt{d-1}r}{2\delta} \right\rceil + 1\right)^{d-1} &\leq \left(\frac{\sqrt{d-1}r}{2\delta} + 2\right)^{d-1}\\
&\leq \left(\frac{\sqrt{d}r}{2\delta}\right)^{d-1}
\end{split}
\end{equation*}
where the last inequality follows from
\begin{equation*}
\begin{split}
\frac{\sqrt{d}r}{2\delta} - \frac{\sqrt{d-1}r}{2\delta} = \frac{(\sqrt{d}-\sqrt{d-1})r}{2\delta} &\geq 4\sqrt{d}(\sqrt{d}-\sqrt{d-1})\\
&\geq 2 (\sqrt{d} + \sqrt{d-1})(\sqrt{d}-\sqrt{d-1})\\
&\geq 2 (d-(d-1)) = 2
\end{split}
\end{equation*}
where we used $r \geq 8 \sqrt{d} \delta$.
\end{proof}

Before giving the algorithm we prove the following technical lemma, which will be important for the proof of correctness.

\begin{lemma}\label{lem:increase-eps}
Let $E$ be a $(d-1)$-dimensional hyperrectangle of $\mathbb{R}^d$ and let $S$ be a nice $\delta$-net of $E$. Let $\eps > 0$ and let $x \in \mathbb{R}^d$ with $\ell := \dist(x,E) > 0$. Then, if all $z \in S$ are $\eps$-unreachable from $x$, it follows that all $y \in E$ are $\eps'$-unreachable from $x$, where
\[\eps' = \eps + \frac{\delta^2}{2 \ell} (L + 2\eps/\ell).\]
\end{lemma}

Here $\dist(x,E) := \min_{y \in E} \|x-y\|_2$.

\begin{proof}
Consider the function $\phi: E \to \mathbb{R}$ defined by
$$\phi(y) = f(y) - f(x) + \eps \|x-y\|_2.$$
For all $z \in S$ we have that $\phi(z) > 0$, since $z$ is $\eps$-unreachable from $x$. We will show that for all $y \in E$ we have $\phi(y) > - \frac{\delta^2}{2}(L + 2\eps/\ell)$. Note that this suffices to prove that all $y \in E$ are $\eps'$-unreachable from $x$, since
\begin{equation*}
\begin{split}
f(y) = f(x) - \eps \|x-y\|_2 + \phi(y) &> f(x) - \eps \|x-y\|_2 - \frac{\delta^2}{2}(L + 2\eps/\ell)\\
&\geq f(x) - \left(\eps + \frac{\delta^2}{2\ell}(L + 2\eps/\ell)\right) \|x-y\|_2\\
&= f(x) - \eps' \|x-y\|_2
\end{split}
\end{equation*}
where we used $\|x-y\|_2 \geq \dist(x,E) = \ell$.

Fix some $y^* \in \argmin_{y \in E} \phi(y)$. In the remainder of this proof, we show that $\phi(y^*) > - \frac{\delta^2}{2}(L + 2\eps/\ell)$. Let $F$ denote the smallest face of $E$ that contains $y^*$. If $F$ is $0$-dimensional, i.e., a corner of the hyperrectangle $E$, then by the definition of a nice $\delta$-net it follows that $S \cap F \neq \emptyset$ and thus $y^* \in S$. In particular, $\phi(y^*) > 0 \geq - \frac{\delta^2}{2}(L + 2\eps/\ell)$. Now consider the case where $F$ is $k$-dimensional for some $k \in \{1, 2, \dots, d-1\}$. Note that $y^*$ cannot lie on the boundary of $F$; otherwise, $F$ would not be the smallest face of $E$ containing $y^*$. Furthermore, $\phi$ is continuously differentiable on $E$, and thus also on the face $F$, since $f$ is continuously differentiable on $E$ and since $\|x-y\|_2 \geq \dist(x,E) > 0$ for all $y \in E$. Hence, given that $y^*$ is a minimum of $\phi$ on $F$, and that it lies in the interior of $F$, we must have that $[\nabla \phi(y^*)]_F = 0$. Here $[\cdot]_F \in \mathbb{R}^k$ denotes the restriction to the $k$ coordinates that are not fixed on the face $F$.

Since $S$ is a nice $\delta$-net of $E$, it follows that $S \cap F$ is a $\delta$-net of $F$. As a result, there exists $z \in S \cap F$ with $\|z-y^*\|_2 \leq \delta$. Let $\psi$ denote the function $\phi$ on the segment $[y^*,z]$ in $F$, parameterized as $\psi: [0,1] \to \mathbb{R}, t \mapsto \phi(y^* + t(z-y^*))$. Note that $\psi$ is continuously differentiable and $\psi'(t) = \langle z-y^*, \nabla \phi(y^* + t(z-y^*)) \rangle$. Using the fact that (i) $\nabla \phi(y) = \nabla f(y) + \eps (y-x)/\|y-x\|_2$, (ii) $\nabla f$ is $L$-Lipschitz-continuous, and (iii) $\|x-y\|_2 \geq \dist(x,E) = \ell$ for all $y \in E \supseteq F$, it can be shown that $\psi'(t)$ is $L'$-Lipschitz-continuous for $L' = \|z-y^*\|_2^2(L + 2\eps/\ell) \leq \delta^2(L + 2\eps/\ell)$.

Furthermore, $\psi'(0) = \langle z-y^*, \nabla \phi(y^*) \rangle = \langle [z-y^*]_F, \nabla^F \phi(y^*) \rangle = 0$, where we used the fact that $z$ and $y^*$ both lie in $F$ (and thus have the same value in all coordinates that are not fixed in $F$) and $\nabla^F \phi(y^*) = 0$. As a result, for any $t \in [0,1]$, we have
$$|\psi'(t)| = |\psi'(t) - \psi'(0)| \leq L' |t-0| \leq \delta^2(L + 2\eps/\ell)t.$$

Now, we can write
$$\phi(z) - \phi(y^*) = \psi(1) - \psi(0) = \int_0^1 \psi'(t) dt \leq \int_0^1 \delta^2(L + 2\eps/\ell)t dt = \frac{\delta^2}{2}(L + 2\eps/\ell)$$
which yields $\phi(y^*) \geq \phi(z) - \frac{\delta^2}{2}(L + 2\eps/\ell) > - \frac{\delta^2}{2}(L + 2\eps/\ell)$ as desired.
\end{proof}

We are now ready to present the full algorithm.

\begin{algorithm}[H]
\DontPrintSemicolon
\LinesNumbered
\caption{The Gradient Flow Parallel Trap (GFPT) Algorithm}
\SetKwInOut{Input}{input}\SetKwInOut{Output}{output}
\Input{ $\eps > 0$, $L > 0$, $d \geq 2$, $x_0 \in \mathbb{R}^d$, query access to $f: \mathbb{R}^d \to [0, + \infty)$ with $L$-Lipschitz $\nabla f$}
\Output{ an $\eps$-stationary point of $f$}
Set $t := 0$ and $\eps_0 := \eps/4$. Let $C_1 := 75 \sqrt{d}$ and $C_2 := 16 d$.\;

Initialize hyperrectangle $R_0 = [a_1,b_1] \times \dots \times [a_d,b_d]$ by setting
    \[R_0 := \{x \in \mathbb{R}^d : \|x-x_0\|_\infty \leq 2 f(x_0)/\eps_0\}.\]

\While{$r_t := \max_i (b_i-a_i) > \frac{\eps}{2\sqrt{d}L}$}{
Pick $j \in \argmax_i (b_i-a_i)$ (arbitrarily).\;

Set $\delta_t := \sqrt{\frac{\eps}{C_1 C_2 L} r_t (3/4)^{\lfloor t/d \rfloor}}$ and query $f$ on a nice $\delta_t$-net $S_1$ of $E_1$ and $S_2$ of $E_2$, where
        \[E_1 = [a_1,b_1] \times \dots \times \{a_j + r_t/3\} \times \dots \times [a_d,b_d],\]
        \[E_2 = [a_1,b_1] \times \dots \times \{b_j - r_t/3\} \times \dots \times [a_d,b_d].\]
        
Let $S^* := \{z \in S_1 \cup S_2: f(z) \leq f(x_t) - \eps_t \|x_t-z\|_2\}$ be the set of all points on the nets that are \emph{not} $\eps_t$-unreachable from $x_t$.\;

  \lIf{$S^* = \emptyset$}{$x_{t+1} := x_t$ \textbf{else} $x_{t+1} := z^* \in \argmin_{z \in S^*} f(z)$.}

  \lIf{$[x_{t+1}]_j \geq a_j + r_t/2$}{$a_j := a_j + r_t/3$ \textbf{else} $b_j := b_j - r_t/3$.}
  
  Set $\eps_{t+1} := \eps_t + \frac{\eps (3/4)^{\lfloor t/d \rfloor}}{C_2}$ and update $t := t+1$.
}

Output $x_t$.
\end{algorithm}

Note that at each iteration of the algorithm, the length of the rectangle $R$ along some dimension $j$ decreases by a factor $2/3$. Since we always pick a dimension $j$ along which the rectangle $R$ has maximal length, and since we stop as soon as the length in each dimension is at most $\eps/2\sqrt{d}L$, the algorithm terminates after $T = d \lceil \log_{3/2}(r_0/(\eps/2\sqrt{d}L)) \rceil = d \lceil \log_{3/2}(16\sqrt{d}f(x_0)L/\eps^2) \rceil$ iterations.

We first argue about the correctness of the algorithm.

\begin{lemma}
The output of the algorithm $x_T$ is an $\eps$-stationary point of $f$.
\end{lemma}

\begin{proof}
First of all, note that if $f(x_0) = 0$, then the algorithm outputs $x_0$ and $x_0$ is a global minimum of $f$ and thus a stationary point. Thus, for the rest of this proof we assume that $f(x_0) > 0$.

Let $R_t$ denote the hyperrectangle at the beginning of iteration $t$. In particular, $R_0 = R$.
We will show that the algorithm satisfies the following invariant at every iteration: all points $y \in \partial R_t$ are $\eps_t$-unreachable from $x_t$.

Let us first see why this invariant implies the correctness of the algorithm. At the last iteration $T$ of the algorithm, we can bound
\begin{equation*}
\begin{split}
\eps_T = \eps_{T-1} + \frac{\eps (3/4)^{\lfloor (T-1)/d \rfloor}}{C_2} = \dots = \eps_0 + \sum_{t=0}^{T-1} \frac{\eps (3/4)^{\lfloor t/d \rfloor}}{C_2} &\leq \eps_0 + \sum_{t=0}^{\infty} \frac{\eps (3/4)^{\lfloor t/d \rfloor}}{C_2}\\
&= \eps_0 + \frac{\eps}{16d} \cdot d \cdot \sum_{i=0}^\infty (3/4)^i\\
&\leq \eps/4 + \eps/4 = \eps/2.
\end{split}
\end{equation*}
Thus, by using the invariant for $t=T$, all $y \in \partial R_T$ are $(\eps/2)$-unreachable from $x_T$. Since the rectangle $R_T$ has side-length at most $\eps/2\sqrt{d}L$, by Corollary~\ref{cor:solution-unconstrained} $x_T$ is an $\eps$-stationary point of $f$.

It remains to prove the invariant. Note that by construction of $R = R_0$ the invariant holds at $t=0$. Indeed, for all $y \in \partial R_0$, we have $\|x_0-y\|_2 \geq \|x_0-y\|_\infty = 2f(x_0)/\eps_0$ and together with the fact that $f(y) \geq 0$ this yields
$$f(y) \geq 0 \geq 2f(x_0) - \eps_0 \|x_0-y\|_2 > f(x_0) - \eps_0 \|x_0-y\|_2$$
where we used the fact that $f(x_0) > 0$.

Now consider any $t$ such that the invariant holds for iteration $t$. We will show that it also holds for iteration $t+1$.

There are two cases to consider, depending on whether $S^* = \emptyset$ or not. First consider the case where $S^* = \emptyset$. In that case, the algorithm sets $x_{t+1} := x_t$. Note that all points $y \in \partial R_{t+1} \cap \partial R_t$ are $\eps_t$-unreachable from $x_t$ by the invariant, and so also $\eps_{t+1}$-unreachable from $x_{t+1}$, since $\eps_{t+1} \geq \eps_t$. It remains to show that all points in $\partial R_{t+1} \setminus \partial R_t$ are also $\eps_{t+1}$-unreachable from $x_t$. By construction of the algorithm it holds that $\partial R_{t+1} \setminus \partial R_t \subseteq E_i$, where $i \in \{1,2\}$ is such that $\ell := \dist(x_t, E_i) \geq r_t/6$. Furthermore, since $S^* = \emptyset$, it follows that all points on the nice $\delta_t$-net $S_i$ of $E_i$ are $\eps_t$-unreachable from $x_t$. As a result, we can apply Lemma~\ref{lem:increase-eps} to deduce that all $y \in E_i$ are $\eps'$-unreachable from $x$, where
\begin{equation*}
\begin{split}
\eps' = \eps_t + \frac{\delta_t^2}{2\ell}(L + 2\eps_t/\ell) \leq \eps_t + \frac{\delta_t^2}{2\ell}(L + 24\sqrt{d}L) \leq \eps_t + \frac{3\delta_t^2}{r_t}(L + 24\sqrt{d}L) &\leq \eps_t + \frac{75\sqrt{d}L\delta_t^2}{r_t}\\
&= \eps_t + \frac{C_1 L\delta_t^2}{r_t}\\
&= \eps_t + \frac{\eps (3/4)^{\lfloor t/d \rfloor}}{C_2}\\
&= \eps_{t+1}
\end{split}
\end{equation*}
where we used $\ell \geq r_t/6 \geq \eps/12\sqrt{d}L$, $\eps_t \leq \eps$ and $\delta_t := \sqrt{\frac{\eps}{C_1 C_2 L} r_t (3/4)^{\lfloor t/d \rfloor}}$. Thus the invariant holds in the first case.

Now consider the case where $S^* \neq \emptyset$. In that case, the algorithm sets $x_{t+1} := z^*$, where $z^* \in \argmin_{z \in S^*} f(z)$. As before, by the invariant we know that all points in $\partial R_{t+1} \cap \partial R_t$ are $\eps_t$-unreachable from $x_t$. Furthermore, since $z^* \in S^*$, we have that $f(z^*) \leq f(x_t) - \eps_t \|x_t - z^*\|_2$. As a result, using the triangle inequality, we obtain that for any $y \in \partial R_{t+1} \cap \partial R_t$
$$f(y) > f(x_t) - \eps_t \|x_t-y\|_2 \geq f(z^*) + \eps_t \|x_t-z^*\|_2 - \eps_t \|x_t-y\|_2 \geq f(z^*) - \eps_t \|z^*-y\|_2$$
i.e., $y$ is $\eps_t$-unreachable from $z^*$. Since $\eps_{t+1} \geq \eps_t$, it follows that all points in $\partial R_{t+1} \cap \partial R_t$ are $\eps_{t+1}$-unreachable from $z^*$. It remains to show that all points in $\partial R_{t+1} \setminus \partial R_t$ are also $\eps_{t+1}$-unreachable from $z^* = x_{t+1}$. By construction of the algorithm it holds that $\partial R_{t+1} \setminus \partial R_t \subseteq E_i$, where $i \in \{1,2\}$ is such that $\dist(z^*, E_i) = r_t/3$. We first show that all points in $S_i$ are $\eps_t$-unreachable from $z^*$. For all points $z \in S_i \setminus S^*$, note that they are $\eps_t$-unreachable from $x_t$, and thus also $\eps_t$-unreachable from $z^*$ by using the triangle inequality as we did earlier for the points $y \in \partial R_{t+1} \cap \partial R_t$. For all points $z \in S_i \cap S^*$, we must have $f(z) \geq f(z^*)$ since we picked $z^* \in \argmin_{z \in S^*} f(z)$. But since $z \in S_i \subset E_i$ and $\dist(z^*, E_i) = r_t/3$, we must have $\|z^*-z\|_2 \geq r_t/3 > 0$, and thus $f(z) > f(z^*) - \eps_t \|z^* - z\|_2$, i.e., $z$ is $\eps_t$-unreachable from $z^*$. Thus, all points in $S_i$ are $\eps_t$-unreachable from $z^*$. Finally, just as in the previous case, we can use Lemma~\ref{lem:increase-eps} to deduce from this that all points in $E_i$ are $\eps_{t+1}$-unreachable from $z^*$. In particular, note that we have $\dist(z^*, E_i) = r_t/3 \geq r_t/6$ so the same arguments can indeed be applied. As a result, the invariant holds in the second case as well.
\end{proof}

\begin{lemma}
For any $d \geq 2$, the number of queries to $f$ performed by the algorithm is
$$O(d)^{\frac{5d-1}{4}} \left(\frac{\sqrt{L f(x_0)}}{\eps}\right)^{d-1}$$
In particular, for $d=2$ the number of queries is $O(\sqrt{L f(x_0)}/\eps)$.
\end{lemma}

\begin{proof}
Let $q_t$ denote the number of queries performed by the algorithm in iteration $t$. Then, the total number of queries performed by the algorithm is $Q = \sum_{t=0}^{T-1} q_t$. In iteration $t$ the algorithm queries the value of $f$ on two nice $\delta_t$-nets $S_1$ and $S_2$ of $E_1$ and $E_2$ respectively. Since $E_1$ and $E_2$ are $(d-1)$-dimensional hyperrectangles with side-length bounded by $r_t$, using Lemma~\ref{lem:nice-net} we obtain that
$$q_t \leq 2 \left(\frac{\sqrt{d}r_t}{2\delta_t}\right)^{d-1}.$$
Note that Lemma~\ref{lem:nice-net} also requires $r_t \geq 8\sqrt{d} \delta_t$. This indeed holds for all $t \in \{0, 1, \dots, T-1\}$, since
$$\frac{r_t}{\delta_t} = \sqrt{\frac{C_1 C_2 L r_t}{\eps (3/4)^{\lfloor t/d \rfloor}}} \geq \sqrt{\frac{C_1 C_2 L r_t}{\eps}} \geq \sqrt{\frac{C_1 C_2}{2\sqrt{d}}} \geq 8 \sqrt{d}$$
where we used $r_t \geq \eps/2\sqrt{d}L$. As a result, the total number of queries $Q$ can be bounded as follows
\begin{equation*}
\begin{split}
Q = \sum_{t=0}^{T-1} q_t \leq \sum_{t=0}^{T-1} 2 \left(\frac{\sqrt{d}r_t}{2\delta_t}\right)^{d-1} &= 2 \sum_{t=0}^{T-1} \left( \sqrt{\frac{d C_1 C_2 L r_t}{4 \eps (3/4)^{\lfloor t/d \rfloor}}} \right)^{d-1}\\
&\leq 2 \left(\frac{d C_1 C_2}{4}\right)^{\frac{d-1}{2}} \left(\frac{L r_0}{\eps}\right)^{\frac{d-1}{2}} \sum_{t=0}^{T-1} \left( \frac{(2/3)^{\lfloor t/d \rfloor}}{(3/4)^{\lfloor t/d \rfloor}} \right)^{\frac{d-1}{2}}
\end{split}
\end{equation*}
where we used $r_t = r_0 (2/3)^{\lfloor t/d \rfloor}$. Furthermore, we can bound
$$\sum_{t=0}^{T-1} \left( \frac{(2/3)^{\lfloor t/d \rfloor}}{(3/4)^{\lfloor t/d \rfloor}} \right)^{\frac{d-1}{2}} \leq d \sum_{i=0}^{\infty} \left( \left(\frac{8}{9}\right)^{i} \right)^{\frac{d-1}{2}} = \frac{d}{1-(8/9)^{(d-1)/2}} \leq 18 d$$
since $d \geq 2$. Thus, we obtain
$$Q \leq O(d)^{\frac{5d-1}{4}} \left(\frac{L r_0}{\eps}\right)^{\frac{d-1}{2}}.$$
Finally, using the fact that $r_0 = 2f(x_0)/\eps_0 = 8f(x_0)/\eps$ we have
$$Q \leq O(d)^{\frac{5d-1}{4}} \left(\frac{\sqrt{L f(x_0)}}{\eps}\right)^{d-1}.$$
\end{proof}

\begin{remark}
The analysis shows that we could have replaced the constant $3/4$ by any other constant $\alpha \in (2/3,1)$. In that case one would also have to modify the values of $C_1$ and $C_2$ accordingly.
\end{remark}

\subsection{Adapting GFPT to the constrained setting}\label{sec:algo-constrained}

In this section we briefly discuss how the algorithm can be adapted to the problem of finding stationary points in a compact domain.

To be more specific, let us consider the minimization problem for a function $f: [0,1]^d \to \mathbb{R}$ with an $L$-Lipschitz-continuous gradient on $[0,1]^d$. Note that unlike in the unconstrained case, we no longer assume (i) that the function is lower bounded, and (ii) that we have a starting point $x_0$. The goal is to find an $\eps$-stationary point for the constrained setting, also known as an $\eps$-KKT point, i.e., a point $x \in [0,1]^d$ such that $\|g(x)\|_2 \leq \eps$, where $g$ is the \emph{projected gradient} of $f$ on $[0,1]^d$. For the minimization setting, the projected gradient $g: [0,1]^d \to \mathbb{R}^d$ of $f$ on $[0,1]^d$ is defined as
\begin{equation*}
g_i(x) = \left\{\begin{tabular}{ll}
$\min\{0,[\nabla f(x)]_i\}$     &  if $x_i = 0$\\
$[\nabla f(x)]_i$     &  if $x_i \in (0,1)$\\
$\max\{0,[\nabla f(x)]_i\}$     &  if $x_i = 1$
\end{tabular}\right.
\end{equation*}
The modifications to the algorithm are very mild, namely:
\begin{itemize}
    \item \textbf{Initialization:} We initialize $x_0$ to be an arbitrary point in $[0,1]^d$, e.g., $x = (1/2, \dots, 1/2)$. Furthermore, instead of initializing
    \[R_0 := \{x \in \mathbb{R}^d : \|x-x_0\|_\infty \leq 2 f(x_0)/\eps_0\}\]
    we initialize
    \[R_0 := [0,1]^d.\]
    \item \textbf{Extraction of a solution:} Instead of outputting $x_T$, we check the $2^d$ corners of $R_T$ until we find one with $\|g(y)\|_2 \leq \eps$. We can use $O(d)$ queries to $f$ to compute a sufficiently good approximation of $\nabla f$ and thus $g$ at any point.
\end{itemize}
The rest of the algorithm is unchanged. The analysis of the algorithm is mostly the same. Here we highlight the main differences.
\begin{itemize}
    \item \textbf{Correctness:}
    The invariant is modified to say that at every iteration $t$: for all points $y$ that lie on a facet $E$ of $R_t$ with $E \nsubseteq \partial [0,1]^d$, we have that $y$ is $\eps_t$-unreachable from $x_t$. In other words, the old invariant still holds, but only on facets of $R_t$ that are not part of the boundary of the domain $[0,1]^d$. The proof of this new invariant is essentially identical to the proof for the old invariant.
    
    This new invariant, together with a modified version of the proof of Lemma~\ref{lem:eps-unreachable} (where we consider the piecewise differentiable \emph{projected} gradient flow defined by $\gamma'(t) = -g(\gamma(t))$) yields: there exists a point $x \in R_T$ with $\|g(x)\|_2 \leq \eps/2$. A simple argument then shows that any corner $y$ of the smallest face of $R_T$ containing $x$ must satisfy $\|g(y)\|_2 \leq \eps$. Thus, the algorithm indeed returns an $\eps$-stationary point.
    \item \textbf{Running time:} The analysis of the number of queries used by the algorithm is identical to the unconstrained version, up to the point where the bound
    $$O(d)^{\frac{5d-1}{4}} \left(\frac{L r_0}{\eps}\right)^{\frac{d-1}{2}}$$
    is derived. Then, using the fact that we now have $r_0 = 1$ (instead of $r_0 = 2f(x_0)/\eps_0$), we obtain the bound
    $$O(d)^{\frac{5d-1}{4}} \left(\sqrt{\frac{L}{\eps}}\right)^{d-1}.$$
\end{itemize}

\section*{Acknowledgments}

We thank the reviewers for helpful suggestions, and Takashi Ishizuka for providing feedback on an earlier version. AH was supported by the Swiss State Secretariat for Education, Research and Innovation (SERI) under contract number MB22.00026. MZ was supported by the Army Research Office (ARO) under contract W911NF-17-1-0304 as part of the collaboration between US DOD, UK MOD and UK Engineering and Physical Research Council (EPSRC) under the Multidisciplinary University Research Initiative (MURI).

\clearpage

\appendix

\section{Definition of Search Problems and Reductions} \label{app:tfnp}

\begin{definition}[Search Problems - $\FNP$] \label{def:FNP}
    A binary relation $\calQ \subseteq \set{0, 1}^* \times \set{0, 1}^*$ is in
  the class $\FNP$ if (i) for every $\vecx, \vecy \in \set{0, 1}^*$ such that
  $(\vecx, \vecy) \in \calQ$, it holds that 
  $\abs{\vecy} \le \poly(\abs{\vecx})$; and (ii) there exists an algorithm that
  verifies whether $(\vecx, \vecy) \in \calQ$ in time 
  $\poly(\abs{\vecx}, \abs{\vecy})$. The \textit{search problem} associated with
  a binary relation $\calQ$ takes some $\vecx$ as input and requests as output
  some $\vecy$ such that $(\vecx, \vecy) \in \calQ$ or $\bot$ if no
  such $\vecy$ exists. 
  The \textit{decision problem} associated with $\calQ$
  takes some $\vecx$ as input and requests as output the bit $1$, if there
  exists some $\vecy$ such that $(\vecx, \vecy) \in \calQ$, and the bit $0$,
  otherwise. The class $\NP$ is defined as the set of decision problems
  associated with relations $\calQ \in \FNP$. 
  The class $\TFNP$ is defined as the set of all $\FNP$ problems $\calQ$ that are \emph{total}, i.e., for all $\vecx \in \set{0, 1}^*$ there exists $\vecy \in \set{0, 1}^*$ with $(\vecx, \vecy) \in \calQ$.
\end{definition}

\begin{definition}[Polynomial-Time Reductions]
    A search problem $P_1$ is \textit{polynomial-time reducible} to $P_2$ if there exist polynomial-time computable functions 
  $f : \set{0, 1}^* \to \set{0, 1}^*$ and 
  $g : \set{0, 1}^* \times \set{0, 1}^* \to \set{0, 1}^*$
  with the following properties: (i) if $\vecx$ is an input to $P_1$, then
  $f(\vecx)$ is an input to $P_2$; and (ii) if $\vecy$ is a solution to $P_2$ on
  input $f(\vecx)$, then $g(\vecx, \vecy)$ is a solution to $P_1$ on
  input $\vecx$.
\end{definition}

We also refer to \cite{BeameCEIP98-NP-search} for a formal definition of black-box reductions.

\subsection{Complete Definition of the Problem \stationary/} \label{app:complete:stationary}

\begin{tcolorbox}[breakable,enhanced]
	\begin{definition}\label{def:kkt-general}
		\stationary/:
		
		\noindent\textbf{Input}:
		\begin{itemize}
			\item precision parameter $\varepsilon > 0$,
			\item Turing machines $\calC_f$ and $\calC_{\nabla f}$ representing $f: \mathbb{R}^d \to \mathbb{R}$ and $\nabla f: \mathbb{R}^d \to \mathbb{R}^d$,
			\item a boundedness constant $B > 0$, and a smoothness constant $L > 0$.
		\end{itemize}
		
		\noindent\textbf{Goal}: Find $x^{\star} \in \R^d$ such that $\norm{\nabla f(x^{\star})}_2 \le \eps$.
        \smallskip
		
		\noindent Alternatively, we also accept one of the following violations as a solution:
		\begin{itemize}
			\item ($f$ or $\nabla f$ is not $L$-Lipschitz) $x,y \in \R^d$ such that
			$$|f(x) - f(y)| > L \|x-y\|_2 \qquad \text{or} \qquad \|\nabla f(x) - \nabla f(y)\|_2 > L \|x-y\|_2,$$
            \item ($f$ is not $B$-bounded) $x\in \R^d$ such that
			$$|f(x)| > B,$$
			\item ($\nabla f$ is not the gradient of $f$) $x,y \in \R^d$ that contradict Taylor's theorem, i.e.,
			$$\bigl|f(y) - f(x) - \langle \nabla f(x), y-x \rangle \bigr| > \frac{L}{2} \|y-x\|_2^2.$$
		\end{itemize}
	\end{definition}
\end{tcolorbox}

\section{Membership in $\PLS$ -- Proof of Lemma \ref{lem:inclusion}} \label{app:inclusion}

    We formalize the argument that we presented in the beginning of Section \ref{sec:inclusion}. Let $\gamma = \eps/(\sqrt{d} L)$, $m = \lceil 8\sqrt{d}B/(\eps \gamma) \rceil$, and $R = m \gamma$. Define the following grid of $[-R,R]^d$
    \[ G_{\gamma} = \set{\gamma \cdot a \mid a \in [[-m, m]]^d}, \]
    In other words, $G_{\gamma}$ is the ortho-canonical grid of $[-R, R]^d$ with step $\gamma$ between two neighboring vertices of the grid. 
    
    Our next goal is to define a neighbor function $N : G_{\gamma} \to G_{\gamma}$ and a potential function $P : G_{\gamma} \to [-B - 1, B + 1]$ such that every point $x \in G_{\gamma}$ with $P(N(x)) \geq P(x)$ is an $\eps$-stationary point of $f$. 
    \medskip
    
    \noindent \textbf{Valid and invalid points.} From the definition of \stationary/ we are given a Turing machine $\calC_f$ that on input $x$ runs in time $\poly(\len(x))$ and outputs the value of $f(x)$. First, we define the set of \textit{valid points} $\mathcal{V}_{\gamma}$ which is a subset of the grid points $G_{\gamma}$. A point $v \in G_{\gamma}$ is valid, i.e., $v \in \mathcal{V}_{\gamma}$, if and only if, $f(v) \le f(0) - (\eps/(2\sqrt{d})) \cdot \norm{v}_2 $. Intuitively, a point $v$ is valid if and only if it is possible for the gradient descent flow starting at 0 to reach $v$ without passing through any $\eps/(2\sqrt{d})$-stationary point.\footnote{Note that this does not mean that this is the case, just that it is possible. Conversely, if a point $v$ is invalid, then it is impossible for the gradient flow starting at $0$ to visit $v$ without encountering an $\eps/(2\sqrt{d})$-stationary point.} We define the set of \textit{invalid points} $\calI_{\gamma}$ to be $\calI_{\gamma} = G_{\gamma} \setminus \calV_{\gamma}$. Finally, for every $v \in G_{\gamma}$ we define $\Gamma(v)$ to be the set of \textit{immediate neighbors} of $v$ as follows
    \[ \Gamma(v) = \set{u \in G_{\gamma} \mid u \neq v, ~ \norm{u - v}_2 \le \gamma}. \]
    It is easy to see that $\abs{\Gamma(v)} \le 2 d$, where we have exact equality for all the points of the grid in the interior of the box $[-R, R]^d$ and inequality for the points on the boundary. From the definition of valid points we have the following claim.

    \begin{claim} \label{clm:validNeighbor}
      For all $v \in \calV_{\gamma}$ it holds that $\abs{\Gamma(v)} = 2 \cdot d$. (Or $v$ yields a counter-example to the boundedness of $f$.)
    \end{claim}
    \begin{proof}
      To show this we only need to argue that all the points of $G_{\gamma}$ that lie on the boundary of $[-R, R]^d$ are invalid. To see this observe that any point $w$ on the boundary of $[-R, R]^d$ satisfies that $\norm{w}_2 \ge R$. Therefore, for $w$ to be valid we need to have that $f(w) \le f(0) - (\eps/(2\sqrt{d})) R$ but $f(0) \le B$. This means that we need $f(w) \le B - (\eps/(2\sqrt{d})) R$ which, since $R \geq 8\sqrt{d}B/\eps$ by definition, implies $f(w) \le B - 4B = -3B$. But if $f(w) \notin [-B, B]$, then $w$ can be used as a violation solution, since it violates the boundedness of $f$.
    \end{proof}
    \medskip
    
    \noindent \textbf{The potential function $\boldsymbol{P}$.} We start by defining the function $P$. For the valid points we define $P(v)$ to be the output of the Turing machine $\calC_f$ on $v$, i.e., $P(v) = f(v)$. For any invalid point $v \in \calI_{\gamma}$, we define $P(v) = B + 1$.
    \medskip
    
    \noindent \textbf{The neighbor function $\boldsymbol{N}$.} For every $v \in \calI_{\gamma}$ we define $N(v) = 0$. Since $P(v) = B + 1$ and $0$ is valid by definition, this means that $P(0) < P(v)$ and hence none of the invalid points can satisfy $P(N(v)) \ge P(v)$. Next we define the function $N$ on the valid points. We define $N(v)$ to be the immediate neighbor of $v$ that has minimum value of $P(v)$, i.e.,
    \[ N(v) = \argmin_{w \in \Gamma(v)} P(w). \]
    
    \medskip
    
    It remains to show that any solution of the \textsc{LocalOpt} instance with inputs $P$, $N$ as defined above, yields a solution to the \stationary/ problem. Let $v \in G_\gamma$ be such that $\|\nabla f(v)\|_2 > \eps$. We will show that $v$ cannot be a solution of the \textsc{LocalOpt} instance. In other words, we will show that $P(N(v)) < P(v)$. We have already noted above that if $v$ is invalid, it cannot be a solution to \textsc{LocalOpt}. Thus, we may also assume that $v$ is valid. Since $\|\nabla f(v)\|_2 > \eps$, there exist $i \in [d]$ and $s \in \{0,1\}$ such that $\langle \nabla f(v), (-1)^s e_i \rangle \leq -\eps/\sqrt{d}$. Consider the point $w = v + \gamma \cdot (-1)^s e_i$. Note that by Claim~\ref{clm:validNeighbor}, $w \in \Gamma(v)$, i.e., $w$ is an immediate neighbor of $v$. If Taylor's theorem does not hold for $v, w$, then we have found a violation solution. If, on the other hand, it holds, then we can write
    \begin{equation*}
    \begin{split}
    f(w) \leq f(v) + \langle \nabla f(v), w-v \rangle + \frac{L}{2}\|w-v\|_2^2 &= f(v) + \gamma \langle \nabla f(v), (-1)^s e_i \rangle + \frac{L \gamma^2}{2}\|(-1)^s e_i\|_2^2\\
    &\leq f(v) -\frac{\eps \gamma}{\sqrt{d}} + \frac{L \gamma^2}{2}\\
    &= f(v) -\frac{\eps \gamma}{2\sqrt{d}}\\
    &= f(v) -\frac{\eps}{2\sqrt{d}} \|w-v\|_2
    \end{split}
    \end{equation*}
    where we used $\gamma = \eps/(\sqrt{d}L)$. From this inequality we can already deduce that $f(w) < f(v)$. Furthermore, using the fact that $v$ is valid, and thus $f(v) \leq f(0) - \eps/(2\sqrt{d}) \cdot \|v\|_2$, we also deduce that $f(w) \leq f(0) - \eps/(2\sqrt{d}) \cdot \|w\|_2$. In other words, $w$ is also valid and thus $P(w) = f(w) < f(v) = P(v)$. It remains to show that $P(N(v)) \leq P(w)$. But this follows from the definition of $N$ and the fact that $w \in \Gamma(v)$.

    To conclude the reduction we need to transform $N$, $P$ from Turing machines to boolean circuits and their domain and range to $[2^n]$ for some $n \in \N$. First, observe that the binary representation of the points in $v \in G_{\gamma}$ is $\len(v) \le d \cdot \log(R/\gamma)$. Hence, the output $f(v)$ for any $v \in G_{\gamma}$ has $\len(f(v)) = q(\len(v))$ where $q$ is some polynomial that is specified together with the description of $\calC_f$. We can then pick $n = q(\len(v))$ and we can map the set $[2^n]$ to describe both $G_{\gamma}$ and the possible outputs of the Turing machine $P$ that we described above. This way we can make $N$, $P$ to be mappings from $[2^n]$ to itself and have running time $r(n)$ for some polynomial $r$. We can now use classical transformations of Turing machines with running time $r(n)$ to boolean circuits with size $r(n) \log (r(n))$. This completes the reduction of \stationary/ to \textsc{LocalOpt}. Based on the above, it is also easy to see that the reduction is polynomial-time in the input, i.e., the size of the Turing machines $\calC_f$, $\calC_{\nabla f}$, and polylogarithmic in the parameters $B, L$, and $1/\eps$.

\section{$\PLS$-Hardness Proof -- Proof of Lemma \ref{lem:hardness}} \label{app:hardness}

\noindent We prove the hardness result by constructing a function $f : \R^2 \to \R$ where finding approximate stationary points is $\PLS$-complete. For $d > 2$ we can use the function $f$ below and define a function $h : \R^d \to \R$ as $h(x) = f(x_1, x_2)$ and it is easy to see that finding approximate stationary points for $h$ is as hard as finding approximate stationary points for $f$. Hence, our hardness holds for any $d \ge 2$.

\noindent We describe the hardness construction in a top-down way. We start with how we use some \textit{tiles} to cover the whole plane $\R^2$, then we describe the structure of these tiles, and how they contain what we call a \textbf{PLS Box} and then we describe the structure of the PLS Box. To complete our proof we need to make sure that we do not create any stationary points in places that do not correspond to a solution of the $\PLS$ instance that we want to solve. This last step is not too complicated, but rather tedious because it involves checking many small parts of our construction.

\subsection{Periodic Structure of $f$} \label{sec:periodic}

We need to define the function $f$ over the whole plane $\R^2$ and we want to make sure that the value of $f$ remains bounded. In order to satisfy this property, we define $f$ as a periodic function. More formally we will design some function $g : [0, M]^2 \to \R$ for some \textbf{odd}\footnote{We need to pick $M$ odd to make sure that there is an even number of natural numbers in the interval $[0, M]$.} natural number $M$ and then we define $f$ as:
\begin{align}
  f(x, y) = g\p{x - M \cdot \floor{\frac{x}{M}}, y - M \cdot \floor{\frac{y}{M}}} \label{eq:definitionPeriodic}
\end{align}
where $\floor{x}$ denotes the integer part of $x$ and hence $x - M \floor{\frac{x}{M}}$ represents the ``modulo $M$'' operator over the reals. Of course, in order for this construction to be correct we need to make sure that on the boundary of $[0, M]^2$ the different copies of the function $g$ touch in a consistent way. In particular, we need that 
\begin{align}
  g(x, 0) = g(x, M) \quad \text{ and } \quad g(0, y) = g(M, y) \quad \text{ for all } x, y \in [0, M]. \label{eq:boudaryConditions}
\end{align}

\begin{align}
  \nabla g(x, 0) = \nabla g(x, M) \quad \text{ and } \quad \nabla g(0, y) = \nabla g(M, y) \quad \text{ for all } x, y \in [0, M]. \label{eq:boudaryConditions:Gradient}
\end{align}

\subsection{Defining $g$ on a Grid} \label{sec:grid}

The next step is to design the function $g$. Because we start from a combinatorial problem, i.e., \iter/, and our goal is to construct a continuous function it is very helpful to design $g$ in two steps: first we define the values and the gradients of $g$ on a discrete grid that is a subset of the square $[0, M]^2$ and then we use an interpolation, described in Section \ref{sec:interpolation}, to define $g$ in the rest of $[0, M]^2$. The grid that we use is the following:
\begin{align}
  G_{M} = \set{(a, b) \mid a, b \in \N, a, b \in [0, M]} \triangleq [[0, M]]^2 \label{eq:definitionGrid}.
\end{align}  
Our construction starts with defining function values $g(a, b)$ and gradient values $\nabla g(a, b)$ for the points $(a, b) \in G_M$ of the grid $G_M$. Outside the grid, the values of $g$ are defined via interpolation as we describe in the next section.

\subsection{Bi-cubic Interpolation} \label{sec:interpolation}

Given the function and gradient values of $g$ on $G_M$ we apply \textit{bi-cubic interpolation} in every \textit{small box} of $G_M$ to define $g$ everywhere in $[0, M]^2$. Consider a small box $[a, a + 1] \times [b, b + 1]$, with $(a, b) \in G_M$. If we have the function and the gradient values of $g$ for all the four corners $(a, b)$, $(a, b + 1)$, $(a + 1, b)$, $(a + 1, b + 1)$ then we can define $g$ in every point of the small box $[a, a + 1] \times [b, b + 1]$ using a polynomial of the form:
  \begin{equation} 
    g(x, y) = \sum_{i=0}^3 \sum_{j=0}^3 a_{ij} \cdot (x - a)^i \cdot (y - b)^j \label{eq:bicubic}
  \end{equation}
where the coefficients $a_{ij}$ are computed as follows
\begin{align}\label{eq:bicubic-coefficients}
  &\begin{bmatrix}
    a_{00} & a_{01} & a_{02} & a_{03}\\
    a_{10} & a_{11} & a_{12} & a_{13}\\
    a_{20} & a_{21} & a_{22} & a_{23}\\
    a_{30} & a_{31} & a_{32} & a_{33}\\
   \end{bmatrix}\\
  &=
  \begin{bmatrix}
    1 & 0 & 0 & 0\\
    0 & 0 & 1 & 0\\
    -3 & 3 & -2 & -1\\
    2 & -2 & 1 & 1\\
  \end{bmatrix} \cdot 
  \begin{bmatrix}
    g(a, b) & g(a, b') & g_y(a, b) & g_y(a, b') \\
    g(a', b) & g(a', b') & g_y(a', b) & g_y(a', b') \\
    g_x(a, b) & g_x(a, b') & 0 & 0 \\
    g_x(a', b) & g_x(a', b') & 0 & 0 \\
  \end{bmatrix} \cdot
  \begin{bmatrix}
    1 & 0 & -3 & 2\\
    0 & 0 & 3 & -2\\
    0 & 1 & -2 & 1\\
    0 & 0 & -1 & 1\\
  \end{bmatrix}\nonumber
\end{align}
where $a' := a+1$, $b' := b+1$, and where $g_x$ and $g_y$ denote the partial derivatives of $g$ with respect to $x$ and $y$ respectively. It is well-known that the above interpolation yields a function that is bounded, Lipschitz, and smooth, e.g., see \cite{russell1995polynomial}. The bound on the function value and the Lipschitz constant are determined from the specified function and gradient values on the corner of the corresponding cell as we will see in Section \ref{sec:Lipschitzness}.

From now on we will focus on the definition of $g$ on the grid points $G_M$ and when we finish the description of this we will come back to the properties of bi-cubic interpolation to complete our construction.

\subsection{Structure of the function $g$ on $G_M$} \label{sec:structureOfG}

As we explained we pick $M$ to be an odd number so that in the interval $[0, M]$ there is an even number of natural numbers. We define $K = M/2$ (observe that $K$ is not an integer!) and $M = 3N + 4$ or in other words $N = (M - 4)/3$. For now we will keep $N$ as a parameter and we will pick a specific value of $N$ when we construct the $\PLS$ instance inside $g$. 

The definition of $g$ on $G_M$ contains \textbf{eight} different regions. We represent $5$ of them with different colors, one of them corresponds to the top 2-rows of $G_M$ and is there to satisfy the boundary conditions \ref{eq:boudaryConditions}, and we defer the definition of the last two for Section \ref{sec:PLSBoxes}. For each colored region we specify the following thing: (a) its location, (b) how the function values are defined in the discrete points of $G_M$ in this region, and (c) how the gradient values are defined in the discrete points of $G_M$ in this region.  The values of the gradient that we use are either $(-\delta, 0)$ or $(0, -\delta)$ with $\delta = 1/2$. We use the vector multiplier $\delta$ here because it makes some arguments easier.
  
To be more precise, our construction involves the following regions: \textcolor{blueDark}{\textbf{dark blue region}}, \textcolor{redDark}{\textbf{dark red region}}, \textcolor{blueContr}{\textbf{light blue region}}, \textcolor{redContr}{\textbf{light red region}}, \textcolor{black}{\textbf{background}}, \textcolor{black}{\textbf{top region}}, \textcolor{black}{\textbf{PLS Box (A)}}, and \textcolor{black}{\textbf{PLS Box (B)}} as shown in Figure \ref{fig:regionsLocations}. Our goal is to first define all the regions except the PLS Boxes and show that these regions do not contain any stationary point. Then, in Section \ref{sec:PLSBoxes}, we define the structure of the PLS boxes and how we can encode a PLS instance in them. e finally then show that solutions can occur only in places that correspond to solutions of the initial PLS instance.

\begin{figure}
    \centering
    \subfigure[The locations of the eight regions that we use to define $g$ on $G_M$.]{\label{fig:regionsLocations} \includegraphics[scale=0.9]{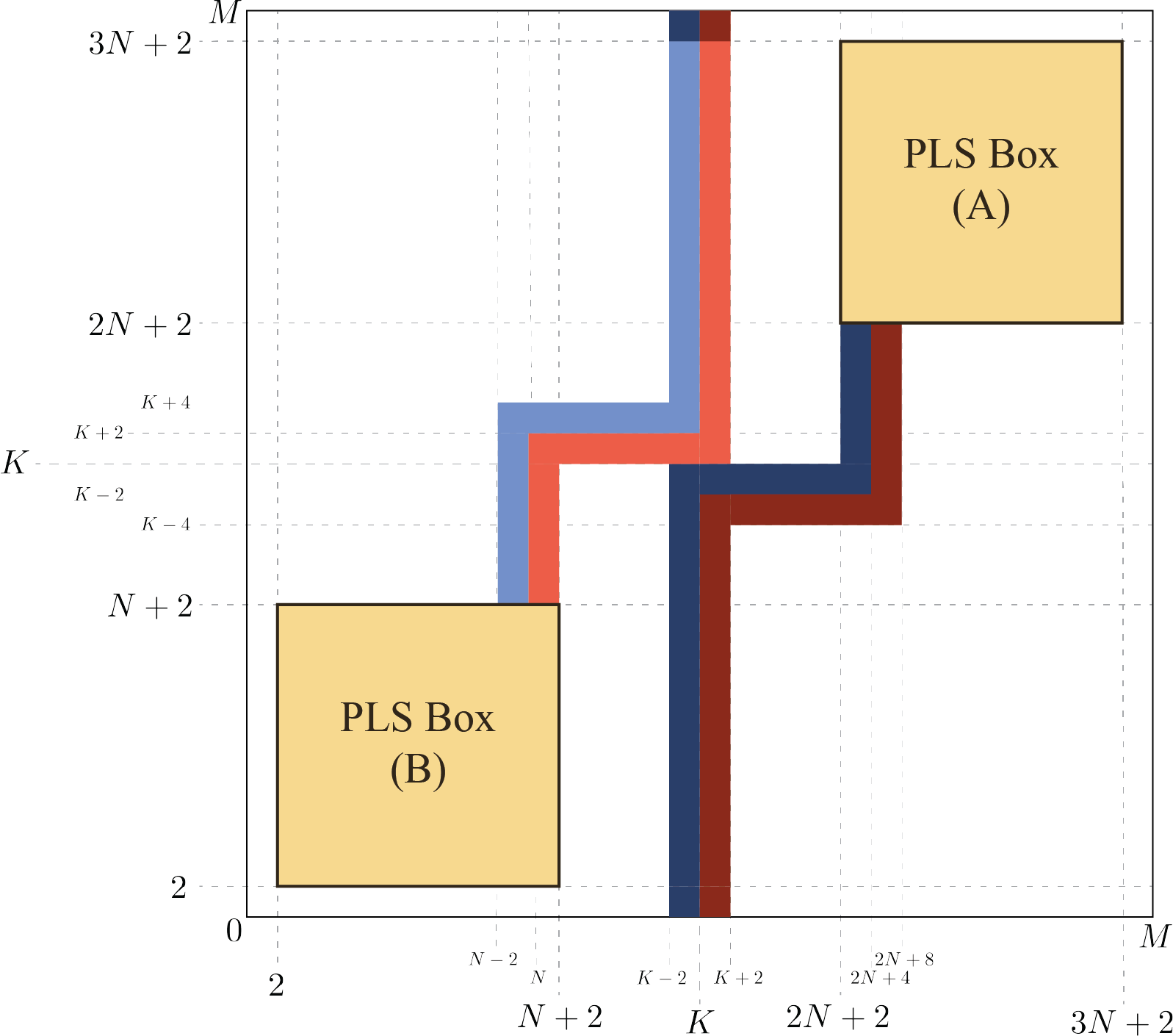}} \\[6pt]
    \subfigure[The arrows indicate the direction in which the function values decrease in every region except the PLS boxes.]{\label{fig:functionValues} \includegraphics[width=0.48\textwidth]{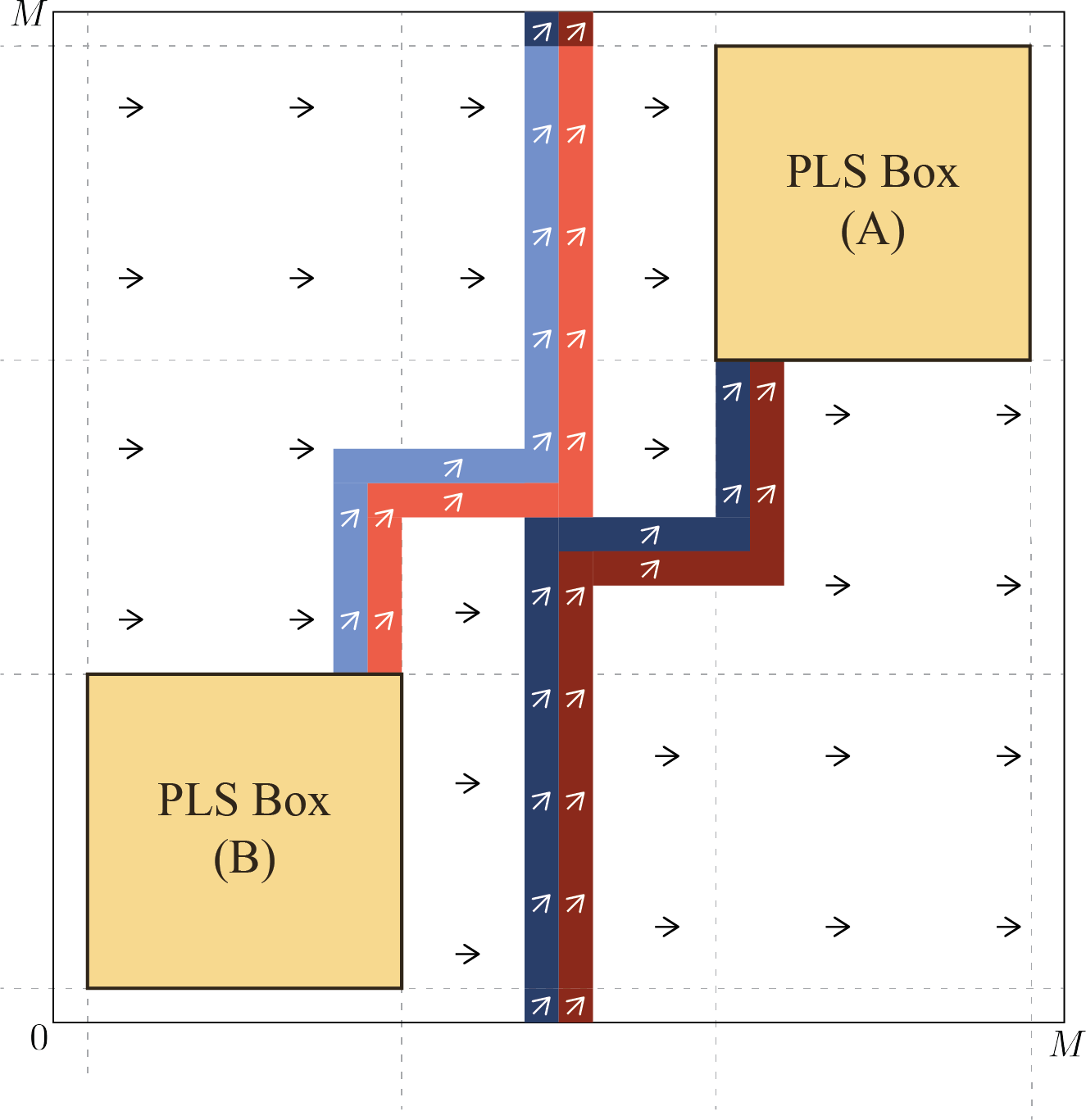}} ~
    \subfigure[The arrows indicate the direction of the negative gradient in every region except the PLS boxes.]{\label{fig:gradientValues} \includegraphics[width=0.48\textwidth]{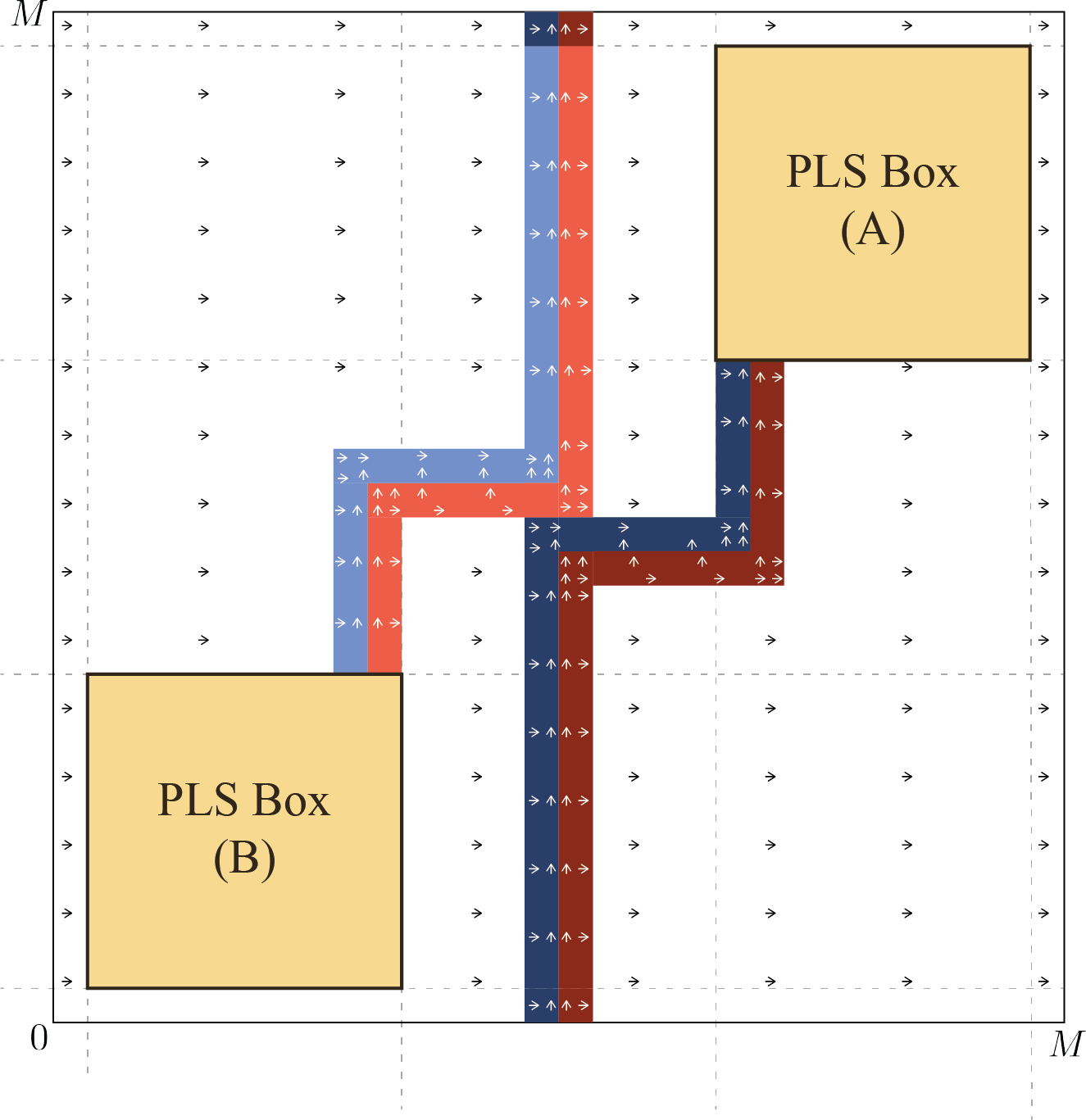}}
    \caption{The location, direction that the function values decrease and directions of the negative gradient for the regions that define $g$ on $G_M$.}
  \end{figure}

  \begin{figure}
    \centering
    \includegraphics[width=0.5\textwidth]{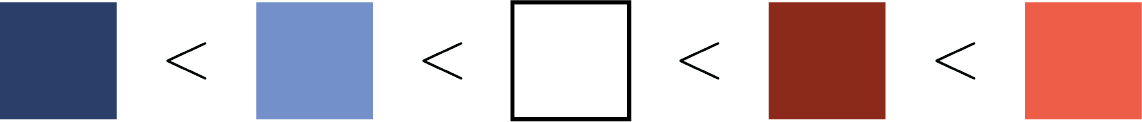}
    \caption{The relative order of function values in different regions. The middle white color with black boundary corresponds to the function values of the background region.}
    \label{fig:relativeOrder}
  \end{figure}
  
  \medskip
  
  \noindent \textcolor{blueDark}{\textbf{Dark Blue Region.}} We denote this region by $\calD_B$ and has the following specifications:
  \begin{enumerate}
    \item[\textbf{(a)}] \textbf{Location.} As shown in Figure \ref{fig:regionsLocations}, the dark blue region contains the following three segments of points, each of which as we can see has width $2$, two of these are vertical segments and one is horizontal. In particular, $\calD_B$ contains the following points:
    \begin{itemize}
      \item[$\triangleright$] The points $(a, b) \in G_M$ such that $K - 2 \le a \le K$ and $0 \le b \le K$.
      \item[$\triangleright$] The points $(a, b) \in G_M$ such that $K - 2 \le a \le 2 N + 4$ and $K - 2 \le b \le K$.
      \item[$\triangleright$] The points $(a, b) \in G_M$ such that $2 N + 2 \le a \le 2 N + 4$ and $K \le b \le 2 N + 2$.
    \end{itemize}
    
    \item[\textbf{(b)}] \textbf{Function Values.} We define the following function
    \begin{align}
      h_{DB}(x, y) = -x - y - 6 \cdot M \label{eq:darkBlueFunctionValueDefinition}
    \end{align}
    and for every $(a, b) \in \calD_B$ we set $g(a, b) = h_{DB}(a, b)$. As shown in Figure \ref{fig:functionValues} the function values inside the dark blue region decrease as both of the coordinates $x$ and $y$ increase. Another property of dark blue region, that will be clear once we define all the regions, is that it has the smallest values among all the regions of $g$ as shown in Figure \ref{fig:relativeOrder}.

    \item[\textbf{(c)}] \textbf{Gradient Values.} The gradient values in the dark blue region is either $(-\delta, 0)$ or $(0, -\delta)$. As we explained in (a) the dark blue region contains two vertical segments and one horizontal segment each of which has length $2$. For the vertical segments the left part has gradient $(-\delta, 0)$ and the right $(0, -\delta)$. For the horizontal segment the top part has gradient $(-\delta, 0)$ and the bottom $(0, -\delta)$. More precisely we have that:
    \begin{itemize}
      \item[$\blacktriangleright$] For the points $(a, b) \in \calD_B$ it holds that  $\nabla g(a, b) = (-\delta, 0)$ in the following cases:
      \begin{itemize}
        \item[$\triangleright$] $a = K - 3/2$ and $0 \le b \le K$,
        \item[$\triangleright$] $K - 2 \le a \le 2 N + 3$ and $b = K - 1/2$,
        \item[$\triangleright$] $a = 2 N + 5/2$ and $K \le b \le 2N + 2$.
      \end{itemize}

      \item[$\blacktriangleright$] For the points $(a, b) \in \calD_B$ it holds that  $\nabla g(a, b) = (0, -\delta)$ in the following cases:
      \begin{itemize}
        \item[$\triangleright$] $a = K - 1/2$ and $0 \le b \le K - 1$,
        \item[$\triangleright$] $K - 1 \le a \le 2 N + 4$ and $b = K - 3/2$,
        \item[$\triangleright$] $a = 2 N + 7/2$ and $K - 1 \le b \le 2N + 2$.
      \end{itemize}
    \end{itemize}
    In Figure \ref{fig:gradientValues} we show the direction of the negative gradient for all the regions.
  \end{enumerate}
\medskip

\noindent \textcolor{redDark}{\textbf{Dark Red Region.}} We denote this region by $\calD_R$ and has the following specifications:
  \begin{enumerate}
    \item[\textbf{(a)}] \textbf{Location.} As shown in Figure \ref{fig:regionsLocations}, the dark red region contains the following three segments of points, each of which as we can see as width $2$, two of these are vertical segments and one is horizontal. In particular, $\calD_R$ contains the following points:
    \begin{itemize}
      \item[$\triangleright$] The points $(a, b) \in G_M$ such that $K \le a \le K + 2$ and $0 \le b \le K - 2$.
      \item[$\triangleright$] The points $(a, b) \in G_M$ such that $K \le a \le 2 N + 4$ and $K \le b \le K + 2$.
      \item[$\triangleright$] The points $(a, b) \in G_M$ such that $2 N + 4 \le a \le 2 N + 6$ and $K - 4 \le b \le 2 N + 2$.
    \end{itemize}
    
    \item[\textbf{(b)}] \textbf{Function Values.} We define the following function
    \begin{align}
      h_{DR}(x, y) = -x - y + 3 \cdot M \label{eq:darkRedFunctionValueDefinition}
    \end{align}
    and for every $(a, b) \in \calD_R$ we set $g(a, b) = h_{DR}(a, b)$. As shown in Figure \ref{fig:functionValues} the function values inside the dark red region decrease as both of the coordinates $x$ and $y$ increase. Another property of dark red region, that will be clear once we define all the regions, is that it has the value greater than both of the blue regions and the background and has value smaller that the light red region as shown in Figure \ref{fig:relativeOrder}.

    \item[\textbf{(c)}] \textbf{Gradient Values.} The gradient values in the dark red region is either $(-\delta, 0)$ or $(0, -\delta)$. As we explained in (a) the dark red region contains two vertical segments and one horizontal segment each of which has length $2$. For the vertical segments the left part has gradient $(0, -\delta)$ and the right $(-\delta, 0)$. For the horizontal segment the top part has gradient $(0, -\delta)$ and the bottom $(-\delta, 0)$. More precisely we have that:
    \begin{itemize}
      \item[$\blacktriangleright$] For the points $(a, b) \in \calD_R$ it holds that  $\nabla g(a, b) = (-\delta, 0)$ in the following cases:
      \begin{itemize}
        \item[$\triangleright$] $a = K + 3/2$ and $0 \le b \le K - 3$,
        \item[$\triangleright$] $K + 1 \le a \le 2 N + 5$ and $b = K - 7/2$,
        \item[$\triangleright$] $a = 2 N + 11/2$ and $K - 4 \le b \le 2N + 2$.
      \end{itemize}

      \item[$\blacktriangleright$] For the points $(a, b) \in \calD_R$ it holds that  $\nabla g(a, b) = (0, -\delta)$ in the following cases:
      \begin{itemize}
        \item[$\triangleright$] $a = K + 1/2$ and $0 \le b \le K - 2$,
        \item[$\triangleright$] $K \le a \le 2 N + 5$ and $b = K - 5/2$,
        \item[$\triangleright$] $a = 2 N + 9/2$ and $K - 3 \le b \le 2N + 2$.
      \end{itemize}
    \end{itemize}
    In Figure \ref{fig:gradientValues} we show the direction of the negative gradient for all the regions.
  \end{enumerate}
\medskip

\noindent \textcolor{blueContr}{\textbf{Light Blue Region.}} We denote this region by $\calL_B$ and has the following specifications:
  \begin{enumerate}
    \item[\textbf{(a)}] \textbf{Location.} As shown in Figure \ref{fig:regionsLocations}, the light blue region contains the following three segments of points, each of which as we can see as width $2$, two of these are vertical segments and one is horizontal. In particular, $\calL_B$ contains the following points:
    \begin{itemize}
      \item[$\triangleright$] The points $(a, b) \in G_M$ such that $K - 2 \le a \le K$ and $K + 2 \le b \le M$.
      \item[$\triangleright$] The points $(a, b) \in G_M$ such that $N - 2 \le a \le K$ and $K + 2 \le b \le K + 4$.
      \item[$\triangleright$] The points $(a, b) \in G_M$ such that $N - 2 \le a \le N$ and $N + 2 \le b \le K + 4$.
    \end{itemize}
    
    \item[\textbf{(b)}] \textbf{Function Values.} We define the following function
    \begin{align}
      h_{LB}(x, y) = - x - y - 3 \cdot  M \label{eq:lightBlueFunctionValueDefinition}
    \end{align}
    and for every $(a, b) \in \calL_B$ we set $g(a, b) = h_{LB}(a, b)$. As shown in Figure \ref{fig:functionValues} the function values inside the light blue region decrease as both of the coordinates $x$ and $y$ increase. Another property of light blue region, that will be clear once we define all the regions, is that it has the value greater than the dark blue region but smaller than both of the red regions and the background as shown in Figure \ref{fig:relativeOrder}.

    \item[\textbf{(c)}] \textbf{Gradient Values.} The gradient values in the light blue region is either $(-\delta, 0)$ or $(0, -\delta)$. As we explained in (a) the light blue region contains two vertical segments and one horizontal segment each of which has length $2$. For the vertical segments the left part has gradient $(-\delta, 0)$ and the right $(0, -\delta)$. For the horizontal segment the top part has gradient $(-\delta, 0)$ and the bottom $(0, -\delta)$. More precisely we have that:
    \begin{itemize}
      \item[$\blacktriangleright$] For the points $(a, b) \in \calL_B$ it holds that  $\nabla g(a, b) = (-\delta, 0)$ in the following cases:
      \begin{itemize}
        \item[$\triangleright$] $a = K - 3/2$ and $K + 3 \le v \le M - 2$,
        \item[$\triangleright$] $N - 2 \le a \le K - 1$ and $b = K + 7/2$,
        \item[$\triangleright$] $a = N - 3/2$ and $N + 2 \le b \le K + 4$.
      \end{itemize}

      \item[$\blacktriangleright$] For the points $(a, b) \in \calL_B$ it holds that  $\nabla g(a, b) = (0, -\delta)$ in the following cases:
      \begin{itemize}
        \item[$\triangleright$] $a = K - 1/2$ and $K + 2 \le b \le M$,
        \item[$\triangleright$] $N - 1 \le a \le K$ and $b = K + 5/2$,
        \item[$\triangleright$] $a = N - 1/2$ and $N + 2 \le b \le K + 1$.
      \end{itemize}
    \end{itemize}
    In Figure \ref{fig:gradientValues} we show the direction of the negative gradient for all the regions.
  \end{enumerate}
\medskip

\noindent \textcolor{redContr}{\textbf{Light Red Region.}} We denote this region by $\calL_R$ and has the following specifications:
  \begin{enumerate}
    \item[\textbf{(a)}] \textbf{Location.} As shown in Figure \ref{fig:regionsLocations}, the light red region contains the following three segments of points, each of which as we can see as width $2$, two of these are vertical segments and one is horizontal. In particular, $\calL_R$ contains the following points:
    \begin{itemize}
      \item[$\triangleright$] The points $(a, b) \in G_M$ such that $K \le a \le K + 2$ and $K \le b \le M$.
      \item[$\triangleright$] The points $(a, b) \in G_M$ such that $N \le a \le K + 2$ and $K \le b \le K + 2$.
      \item[$\triangleright$] The points $(a, b) \in G_M$ such that $N \le a \le N + 2$ and $N + 2 \le b \le K + 2$.
    \end{itemize}
    
    \item[\textbf{(b)}] \textbf{Function Values.} We define the following function
    \begin{align}
      h_{LR}(x, y) = - x - y + 6 \cdot M \label{eq:lightRedFunctionValueDefinition}
    \end{align}
    and for every $(a, b) \in \calL_R$ we set $g(a, b) = h_{LR}(a, b)$. As shown in Figure \ref{fig:functionValues} the function values inside the light red region decrease as both of the coordinates $x$ and $y$ increase. Another property of light red region, that will be clear once we define all the regions, is that it has the value greater than all the other regions as shown in Figure \ref{fig:relativeOrder}.

    \item[\textbf{(c)}] \textbf{Gradient Values.} The gradient values in the light red region is either $(-\delta, 0)$ or $(0, -\delta)$. As we explained in (a) the light red region contains two vertical segments and one horizontal segment each of which has length $2$. For the vertical segments the left part has gradient $(0, -\delta)$ and the right $(-\delta, 0)$. For the horizontal segment the top part has gradient $(0, -\delta)$ and the bottom $(-\delta, 0)$. More precisely we have that:
    \begin{itemize}
      \item[$\blacktriangleright$] For the points $(a, b) \in \calL_R$ it holds that  $\nabla g(a, b) = (-\delta, 0)$ in the following cases:
      \begin{itemize}
        \item[$\triangleright$] $a = K + 3/2$ and $K \le b \le M$,
        \item[$\triangleright$] $N + 1 \le a \le K + 2$ and $b = K + 1/2$,
        \item[$\triangleright$] $a = N + 3/2$ and $N + 2 \le b \le K + 1$.
      \end{itemize}

      \item[$\blacktriangleright$] For the points $(a, b) \in \calL_B$ it holds that  $\nabla g(a, b) = (0, -\delta)$ in the following cases:
      \begin{itemize}
        \item[$\triangleright$] $a = K + 1/2$ and $K + 1 \le b \le M$,
        \item[$\triangleright$] $N \le a \le K + 1$ and $b = K + 3/2$,
        \item[$\triangleright$] $a = N + 1/2$ and $N + 2 \le b \le K + 2$.
      \end{itemize}
    \end{itemize}
    In Figure \ref{fig:gradientValues} we show the direction of the negative gradient for all the regions.
  \end{enumerate}
\medskip

\noindent \textcolor{black}{\textbf{Background.}} We denote this region by $\calB$ and has the following specifications:
  \begin{enumerate}
    \item[\textbf{(a)}] \textbf{Location.} Background is shown in Figure \ref{fig:regionsLocations} with white and as we can see it contains all the points that are not in any other region.
    
    \item[\textbf{(b)}] \textbf{Function Values.} We define the following function
    \begin{align}
      h_{B}(x, y) = - x + \mathbf{1}\set{x \ge K} + M \label{eq:backgroundFunctionValueDefinition}
    \end{align}
    and for every $(a, b) \in \calB$ we set $g(a, b) = h_{B}(a, b)$. As shown in Figure \ref{fig:functionValues} the function values inside the background decrease as the $x$ coordinate increases and it is independent from the coordinate $y$. Another property of the background is that it has the value greater than the blue regions and smaller than the red regions as shown in Figure \ref{fig:relativeOrder}.

    \item[\textbf{(c)}] \textbf{Gradient Values.} The gradient values in the background is always $(-\delta, 0)$. In Figure \ref{fig:gradientValues} we show the direction of the negative gradient for all the regions.
  \end{enumerate}
\medskip

\noindent \textcolor{black}{\textbf{Top blue region.}} We denote this region by $\calT_B$ and has the following specifications:
  \begin{enumerate}
    \item[\textbf{(a)}] \textbf{Location.} The top region consists of the two top rows of $G_M$ and two of the middle columns, i.e., $(a, b) \in \calT_B$ if and only if $K - 2 \le a \le K$ and $M - 1 \le b \le M$.
    
    \item[\textbf{(b)}] \textbf{Function Values.} We define the following function
    \begin{align}
      h_{TB}(x, y) = - x - (y - M - 1) - 6 \cdot M \label{eq:topBlueFunctionValueDefinition}
    \end{align}
    and for every $(a, b) \in \calB$ we set $g(a, b) = h_{TB}(a, b)$. As shown in Figure \ref{fig:functionValues} the function values inside the background decrease as the $x$ coordinate increases and it is independent from the coordinate $y$. We treat this region as a dark blue region because they have a pretty minor difference in the values and for the purpose of all the proofs the behave exactly the same.

    \item[\textbf{(c)}] \textbf{Gradient Values.} The same as in the dark blue region: for $(a, b) \in \calT_B$, with $a = K - 1/2$ we have $\nabla g(a, b) = (0, -\delta)$ and for $(a, b) \in \calT_B$, with $a = K - 3/2$ we have $\nabla g(a, b) = (-\delta, 0)$.
  \end{enumerate}
\medskip

\noindent \textcolor{black}{\textbf{Top red region.}} We denote this region by $\calT_R$ and has the following specifications:
  \begin{enumerate}
    \item[\textbf{(a)}] \textbf{Location.} The top region consists of the two top rows of $G_M$ and two of the middle columns, i.e., $(a, b) \in \calT_R$ if and only if $K  \le a \le K + 2$ and $M - 1 \le b \le M$.
    
    \item[\textbf{(b)}] \textbf{Function Values.} We define the following function
    \begin{align}
      h_{TR}(x, y) = - x - (y - M - 1) + 3 \cdot M \label{eq:topRedFunctionValueDefinition}
    \end{align}
    and for every $(a, b) \in \calB$ we set $g(a, b) = h_{TR}(a, b)$. As shown in Figure \ref{fig:functionValues} the function values inside the background decrease as the $x$ coordinate increases and it is independent from the coordinate $y$. We treat this region as a dark red region because they have a pretty minor difference in the values and for the purpose of all the proofs the behave exactly the same.

    \item[\textbf{(c)}] \textbf{Gradient Values.} The same as in the dark blue region: for $(a, b) \in \calT_R$, with $a = K + 1/2$ we have $\nabla g(a, b) = (0, -\delta)$ and for $(a, b) \in \calT_R$, with $a = K + 3/2$ we have $\nabla g(a, b) = (-\delta, 0)$.
  \end{enumerate}
\medskip

\noindent \textcolor{black}{\textbf{PLS Boxes.}} We defer the discussion about PLS boxes for Section \ref{sec:PLSBoxes}.
\medskip

Now that we defined all the regions except from the PLS boxes we are ready to prove the following lemma.

\begin{lemma} \label{lem:noSolutionsLemma}
  After interpolating using the techniques discussed in Section \ref{sec:interpolation}, there is no $0.01$-stationary point outside the region of PLS Box (A) or PLS Box (B).
\end{lemma}

  \begin{proof}

  In order to prove Lemma \ref{lem:noSolutionsLemma}, we will show that any small box that does not lie in a solution region, does not contain any $\varepsilon$-stationary point. The behaviour of the function in a given small box depends on the information we have about the four corners, namely the colours and arrows at the four corners, but also on the position of the box in our instance, since the value defined by a colour depends on the position. As in the proof of Lemma 4.3 of \cite{FearnleyGHS22-gradient}, for our proof, it is convenient to consider a box with the (colour and arrow) information about its four corners, but without any information about its position. Indeed, if we can show that a box does not contain any $\varepsilon$-stationary point using only this information, then this will always hold, wherever the box is positioned. As a result, we obtain a finite number of boxes (with colour and arrow information) that we need to check. Conceptually, this is a straightforward task: for each small box we get a set of cubic polynomials that could be generated by bicubic interpolation for that box, and we must prove that no polynomial in that set has an $\eps$-stationary point within the box. Unfortunately there is a big number of boxes that we need to check and for this reason it is convenient to cluster the possible instances that arise in four groups each of which is guaranteed not to have a solution. Every small box can be classified in one of the groups if after a combination of the following transformations we can produce exactly the directions of the arrows at the four corners of the small box that appear in the characteristic image of the group. The possible transformations are:
  \begin{itemize}
    \item[$\triangleright$] \emph{Reflection with respect to the $y$-axis.} Applying this transformation to a box has the following effect: the two corners at the top of the box now find themselves at the bottom of the box (and vice-versa) and the sign of the $y$-coordinate of each arrow is flipped. Using \cref{eq:bicubic,eq:bicubic-coefficients} one can check  that taking the bicubic interpolation of this reflected box yields the same result as taking the interpolation of the original box and then applying the reflection to the interpolated function. We can summarize this by saying that bicubic interpolation \emph{commutes} with this transformation.
    \item[$\triangleright$] \emph{Reflection with respect to the $x$-axis.} Similarly with the above reflection of the $y$-axis.
    \item[$\triangleright$] \emph{Reflection with respect to the axis $y=x$}, i.e., the diagonal through the box. This corresponds to swapping the corners $(0,1)$ and $(1,0)$ of the box, and additionally also swapping the $x$- and $y$-coordinate of the arrows at all four corners. Again, using \cref{eq:bicubic,eq:bicubic-coefficients} one can directly verify that this transformation also commutes with bicubic interpolation (where applying the transformation to the interpolated function $f$ corresponds to considering $(x,y) \mapsto f(y,x)$).
    \item[$\triangleright$] \emph{Reflection with respect to the axis $y=-x$.} Similarly with the above reflection with respect to the $y=-x$ axis.
    \item[$\triangleright$] \emph{Negation.} This corresponds to negating the values and arrows at the four corners, where ``negating an arrow'' just means replacing it by an arrow in the opposite direction. Using \cref{eq:bicubic,eq:bicubic-coefficients}, it is immediately clear that negation commutes with bicubic interpolation.
  \end{itemize}

  Since all five transformations commute with bicubic interpolation, this continues to hold for any more involved transformation that is constructed from these basic five. Furthermore, it is easy to see that the basic transformations do not introduce $\varepsilon$-stationary points. Indeed, if a function does not have any $\varepsilon$-stationary points, then applying any reflection or taking the negation cannot change that property. As a result, if two boxes are ``symmetric'' (i.e., one can be obtained from the other by applying a combination of the three basic transformations), then it is enough to verify that just one of these two boxes does not contain any $\varepsilon$-stationary points when we take the bicubic interpolation.

  Using this notion of symmetry, the boxes that need to be verified can be grouped into just four different groups, namely Groups 1 to 4 that we define below. These are the same four groups used in \citep{FearnleyGHS22-gradient}, where it was already shown that they do not contain any $0.01$-stationary points.

	\paragraph{\bf Group 1}
	
	This group contains all the boxes that are symmetric to a box of the following form:
	
	\begin{minipage}{0.6\textwidth}
		\begin{center}
			\resizebox{85pt}{!}{
				\begin{tikzpicture}[scale=5]
				\path (-0.07, -0.07) -- (0.27, 0.27);
				\draw[lightgray,dashed] (0.0,0.0) rectangle (0.2,0.2);
				\path[->,line width=1,black] (0.0, 0.0) edge (0.08, 0.0);
				\path[->,line width=1,black] (0.0, 0.2) edge (0.08, 0.2);
				\path[->,line width=1,black] (0.2, 0.0) edge (0.28, 0.0);
				\path[->,line width=1,black] (0.2, 0.2) edge (0.28, 0.2);
				\node [draw=black,fill=white,circle,inner sep=0,minimum size=9] at (0.0, 0.0)  {\scriptsize $c$};
				\node [draw=black,fill=white,circle,inner sep=0,minimum size=9] at (0.0, 0.2)  {\scriptsize $a$};
				\node [draw=black,fill=white,circle,inner sep=0,minimum size=9] at (0.2, 0.0)  {\scriptsize $d$};
				\node [draw=black,fill=white,circle,inner sep=0,minimum size=9] at (0.2, 0.2)  {\scriptsize $b$};
				\end{tikzpicture}
			}
		\end{center}
	\end{minipage} \hfill
	\begin{minipage}{0.4\textwidth}
		\begin{tabular}{ll}
			Conditions:\\
			~$a \geq b+1$\\
			~$c \geq d+1$\\
		\end{tabular}
	\end{minipage}
	where $a,b,c,d$ are the values at the four corners of the box as shown.
	
	\paragraph{\bf Group 2}
	
	This group contains all the boxes that are symmetric to a box of the following form:
	
	\begin{minipage}{0.6\textwidth}
		\begin{center}
			\resizebox{85pt}{!}{
				\begin{tikzpicture}[scale=5]
				\path (-0.07, -0.07) -- (0.27, 0.27);
				\draw[darkgray,dotted] (0.0,0.134) -- (0.2,0.134);
				\draw[lightgray,dashed] (0.0,0.0) rectangle (0.2,0.2);
				\path[->,line width=1,black] (0.0, 0.0) edge (0.0, 0.08);
				\path[->,line width=1,black] (0.0, 0.2) edge (0.08, 0.2);
				\path[->,line width=1,black] (0.2, 0.0) edge (0.2, 0.08);
				\path[->,line width=1,black] (0.2, 0.2) edge (0.28, 0.2);
				\node [draw=black,fill=white,circle,inner sep=0,minimum size=9] at (0.0, 0.0)  {\scriptsize $c$};
				\node [draw=black,fill=white,circle,inner sep=0,minimum size=9] at (0.0, 0.2)  {\scriptsize $a$};
				\node [draw=black,fill=white,circle,inner sep=0,minimum size=9] at (0.2, 0.0)  {\scriptsize $d$};
				\node [draw=black,fill=white,circle,inner sep=0,minimum size=9] at (0.2, 0.2)  {\scriptsize $b$};
				\end{tikzpicture}
			}
		\end{center}
	\end{minipage} \hfill
	\begin{minipage}{0.4\textwidth}
		\begin{tabular}{ll}
			\\
			Conditions:\\
			~$a \geq b+1$\\
			~$c \geq a+1$\\
			~$d \geq b+1$\\
			~$c \geq d-1$\\
			\\
		\end{tabular}
	\end{minipage}

	\paragraph{\bf Group 3}
	
	This group contains all the boxes that are symmetric to a box of the following form:
	
	\begin{minipage}{0.6\textwidth}
		\begin{center}
			\resizebox{85pt}{!}{
				\begin{tikzpicture}[scale=5]
				\path (-0.07, -0.07) -- (0.27, 0.27);
				\draw[darkgray,dotted] (0.0,0.134) -- (0.1,0.134) -- (0.1,0.2);
				\draw[lightgray,dashed] (0.0,0.0) rectangle (0.2,0.2);
				\path[->,line width=1,black] (0.0, 0.0) edge (0.0, 0.08);
				\path[->,line width=1,black] (0.0, 0.2) edge (0.08, 0.2);
				\path[->,line width=1,black] (0.2, 0.0) edge (0.2, 0.08);
				\path[->,line width=1,black] (0.2, 0.2) edge (0.2, 0.28);
				\node [draw=black,fill=white,circle,inner sep=0,minimum size=9] at (0.0, 0.0)  {\scriptsize $c$};
				\node [draw=black,fill=white,circle,inner sep=0,minimum size=9] at (0.0, 0.2)  {\scriptsize $a$};
				\node [draw=black,fill=white,circle,inner sep=0,minimum size=9] at (0.2, 0.0)  {\scriptsize $d$};
				\node [draw=black,fill=white,circle,inner sep=0,minimum size=9] at (0.2, 0.2)  {\scriptsize $b$};
				\end{tikzpicture}
			}
		\end{center}
	\end{minipage} \hfill
	\begin{minipage}{0.4\textwidth}
		\begin{tabular}{ll}
			\\
			Conditions:\\
			~$a \geq b+1$\\
			~$c \geq a+1$\\
			~$d \geq b+1$\\
			~$c \geq d-1$\\
			\\
		\end{tabular}
	\end{minipage}

	\paragraph{\bf Group 4}
	
	This group contains all the boxes that are symmetric to a box of the following form:
	
	\begin{minipage}{0.6\textwidth}
		\begin{center}
			\resizebox{85pt}{!}{
				\begin{tikzpicture}[scale=5]
				\path (-0.07, -0.07) -- (0.27, 0.27);
				\draw[darkgray,dotted] (0.0,0.067) -- (0.1,0.067) -- (0.1,0);
				\draw[lightgray,dashed] (0.0,0.0) rectangle (0.2,0.2);
				\path[->,line width=1,black] (0.0, 0.0) edge (0.08, 0.0);
				\path[->,line width=1,black] (0.0, 0.2) edge (0.0, 0.28);
				\path[->,line width=1,black] (0.2, 0.0) edge (0.2, 0.08);
				\path[->,line width=1,black] (0.2, 0.2) edge (0.2, 0.28);
				\node [draw=black,fill=white,circle,inner sep=0,minimum size=9] at (0.0, 0.0)  {\scriptsize $c$};
				\node [draw=black,fill=white,circle,inner sep=0,minimum size=9] at (0.0, 0.2)  {\scriptsize $a$};
				\node [draw=black,fill=white,circle,inner sep=0,minimum size=9] at (0.2, 0.0)  {\scriptsize $d$};
				\node [draw=black,fill=white,circle,inner sep=0,minimum size=9] at (0.2, 0.2)  {\scriptsize $b$};
				\end{tikzpicture}
			}
		\end{center}
	\end{minipage} \hfill
	\begin{minipage}{0.4\textwidth}
		\begin{tabular}{ll}
			\\
			Conditions:\\
			~$c \geq a+1$\\
			~$d \geq b+1$\\
			~$c \geq d+1$\\
			~$a \geq b-1$\\
			\\
		\end{tabular}
	\end{minipage}

Since only the directions of the gradients and the colors of the corners of the boxes are important there is only a finite number of small cells that we need to check. In particular, as we show in Figure \ref{fig:regions2Analyze}, in the interior of $[0, 1]^2$ there are only four places where different combinations of colors and directions of gradient arise:
\begin{enumerate}
  \item[A.] Space that only invoke the background region.
  \item[B.] The left turn of the regions with light color.
  \item[C.] The left turn of the regions with dark color.
  \item[D.] The right turns of both of the regions that also touch in this space.
   \item[E.] The place where the light and dark colors meet after we periodically repeat the construction of $[0, M]^2$.
\end{enumerate}

\begin{figure}[t]
    \centering
    \includegraphics[scale=0.7]{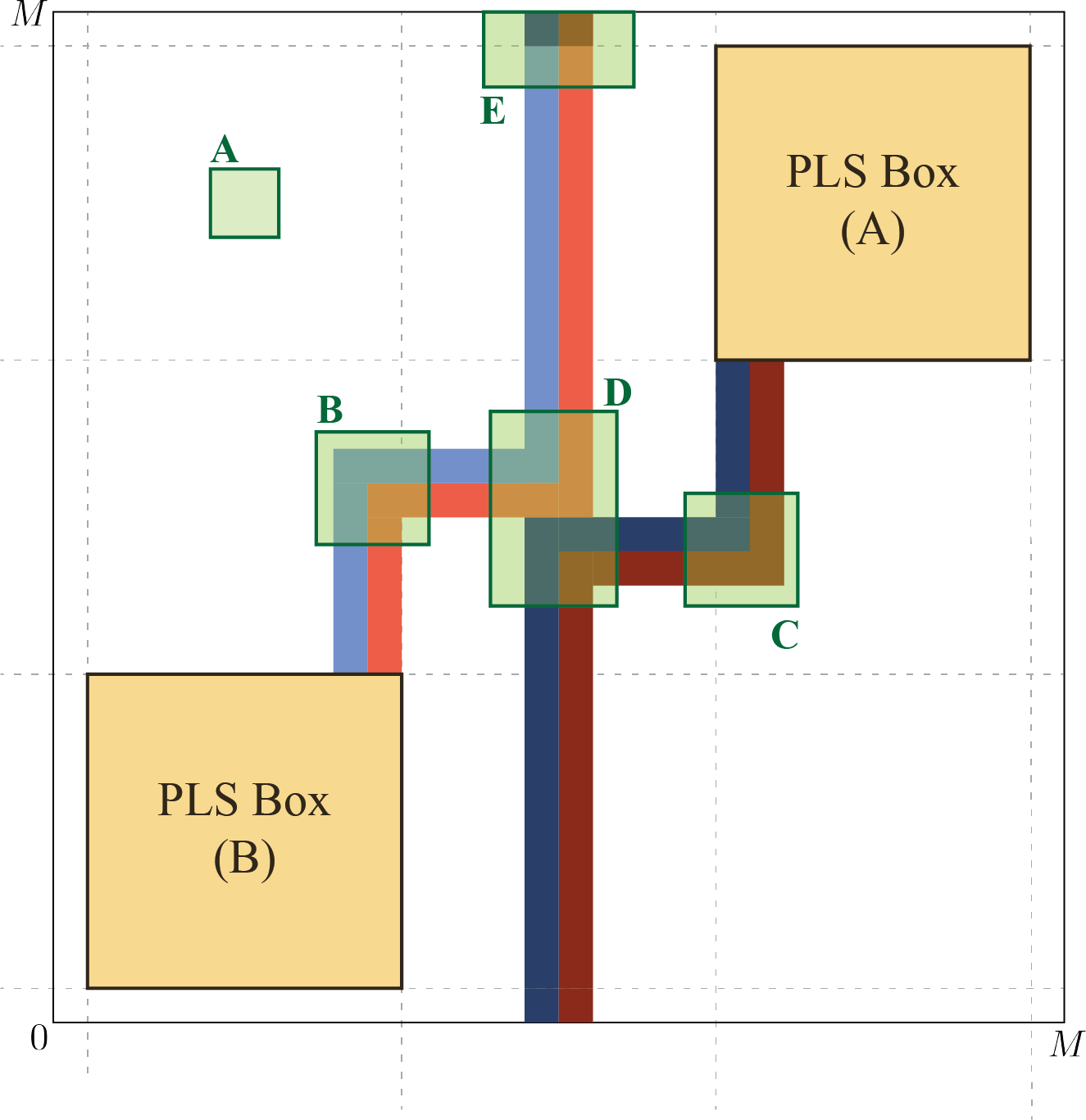}
    \caption{In this figure we indicate with green boxes the regions that we need to check to make sure that no stationary points are created. If we make sure that there are no stationary points in these regions then we can conclude that there are no stationary points anywhere outside the PLS boxes since all the possible patterns appear in these green boxes.}
    \label{fig:regions2Analyze}
\end{figure}

Apart from the interior we also need to make sure that no solutions arise in the boundary after we periodically repeat the $[0, M]^2$. For this observe that both the line $x = 0$ and the line $x = M$ belong to the background and for the background function $h_B$ it is easy to see that $h_B(0, y) = 0 = h_B(M, y)$. For the other direction we observe that the top and bottom rows have either the background, which does not depend on $y$ and hence obviously $h_B(x, 0) = h_B(x, M)$, or it consist of the dark blue and dark red regions and the top blue and top red regions. From the definition of the top blue and the top red regions we then get that $g(x, 0) = g(x, M)$ for all $x \in [0, M]$. Now it is easy to see that using the same argument we can show that $\nabla g(a, 0) = \nabla g(a, M)$ and $\nabla g(0, b) = \nabla g(M, b)$ for all $(a, b) \in G_M$. To extend this property to all the real numbers $x, y \in [0, M]$ we observe that when we use bi-cubic interpolation the values of the gradient on the edge of a small box, depends only on the specified gradients of the corners of that corresponding edge. Using this we can conclude that $\nabla g(x, 0) = \nabla g(x, M)$ and $\nabla g(0, y) = \nabla g(M, y)$ for all $x, y \in [0, 1]$. Putting all these together we conclude that the boundary conditions \eqref{eq:boudaryConditions} and \eqref{eq:boudaryConditions:Gradient} are satisfied by our construction. We state this in the following lemma.

\begin{lemma} \label{lem:boundaryConditions}
  The boundary conditions \eqref{eq:boudaryConditions} and \eqref{eq:boudaryConditions:Gradient} are satisfied by our construction.
\end{lemma} 

Next we verify using Groups 1. - 4. that there is no $0.01$-stationary point in any of the places A, B, C, D, E and this implies that there is no $0.01$-stationary point in any region outside the PLS Boxes which implies our Lemma \ref{lem:noSolutionsLemma}.

\paragraph{A.} We start with a figure of the region A in Figure \ref{fig:A}.

\begin{figure}
    \centering
    \includegraphics[scale=0.2]{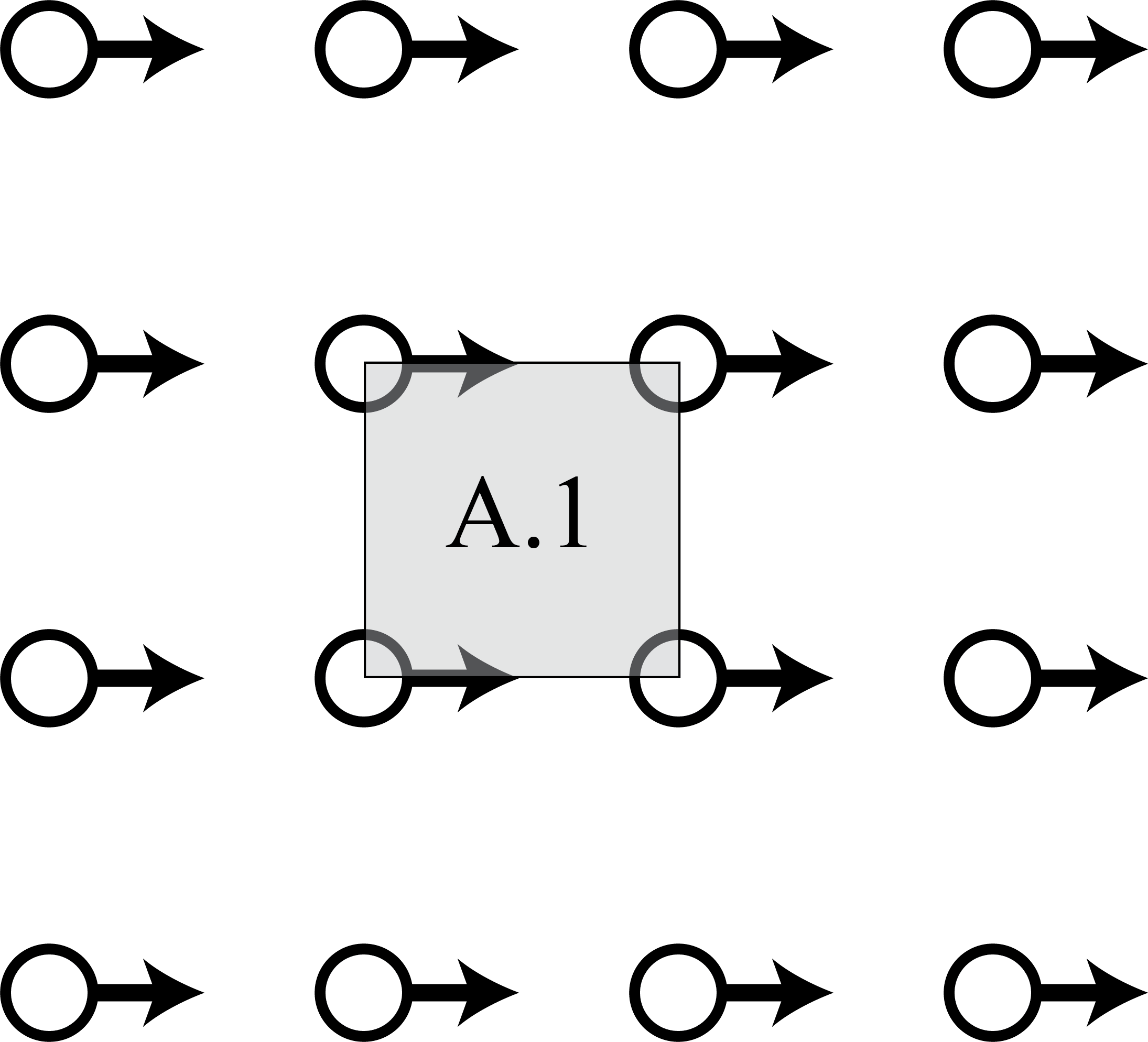}
    \caption{Small boxes of region A of Figure \ref{fig:regions2Analyze}.}
    \label{fig:A}
\end{figure}

As expected in this region A there is only one type of box that appears and it is of Group 1. It is easy to check that since the function value in the background decreases linearly with $x$, the conditions of Group 1 are satisfied and hence there is no solution in this region.

\paragraph{B.} We start with a figure of the region B in Figure \ref{fig:B}, where we indicate all the small boxes with colors and gradients that have not appeared in A.

\begin{figure}
    \centering
    \includegraphics[scale=0.2]{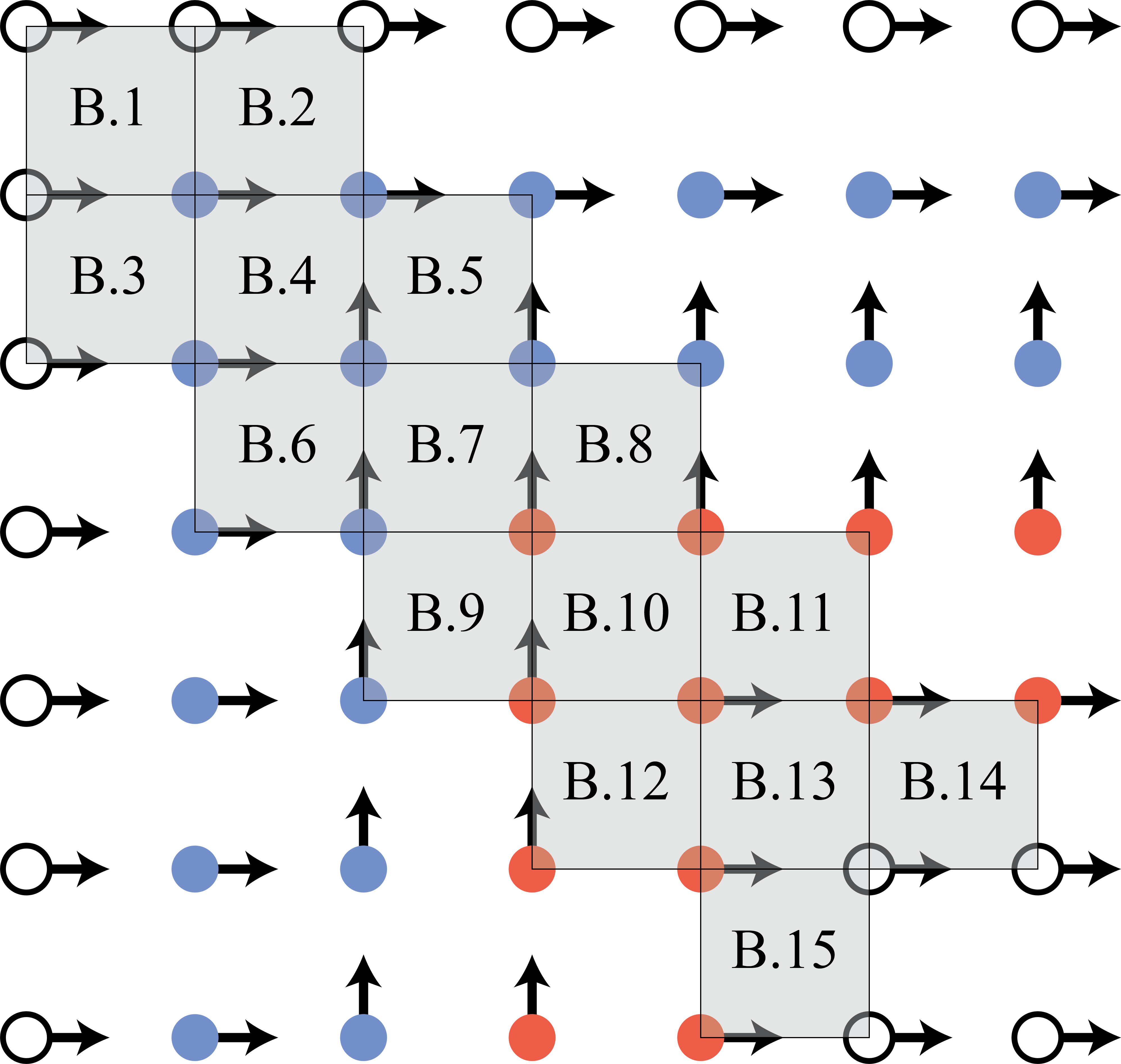}
    \caption{Small boxes of region B of Figure \ref{fig:regions2Analyze}.}
    \label{fig:B}
\end{figure}

We have the following categories:
\begin{itemize}
  \item[$\blacktriangleright$] The small boxes: B.1, B.2, B.3, B.13, B.14, B.15 follow from plain application of Group 1 using that: (1) all the colors decrease linearly as the $x$ coordinate increases, (2) light blue is everywhere at least $\frac{1}{M}$ smaller than the background, and (3) light red is everywhere at least $\frac{1}{M}$ larger than the background.
  \item[$\blacktriangleright$] The small boxes: B.7, B.8, B.9 follow from application of Group 1 after taking a reflection with respect to the $y = x$ axis using that: (1) light blue and light red decrease with linear rate as the $y$ coordinate increases, and (2) light red is everywhere at least $\frac{1}{M}$ larger than light blue.
  \item[$\blacktriangleright$] The small box B.4 follows from Group 3 after applying a $y = x$ reflection.
  \item[$\blacktriangleright$] The small box B.5 follows from a plain application of Group 2.
  \item[$\blacktriangleright$] The small box B.6 follows from Group 2 after applying a $y = x$ reflection.
  \item[$\blacktriangleright$] The small box B.10 follows from Group 3 after applying the following transformations: (i) reflection with respect to $y$, (ii) reflection with respect to $x$, and (iii) negation.
  \item[$\blacktriangleright$] The small box B.11 follows from Group 2 after applying the following transformations: (i) reflection with respect to $y$, (ii) reflection with respect to $x$, and (iii) negation.
  \item[$\blacktriangleright$] The small box B.12 follows from Group 2 after applying a $y = -x$ reflection and negation.
\end{itemize}

So no solution appears in B as well.

\paragraph{C.} We start with a figure of the region C in Figure \ref{fig:C}, where we indicate all the small boxes with colors and gradients that have not appeared in A or B.

\begin{figure}
    \centering
    \includegraphics[scale=0.2]{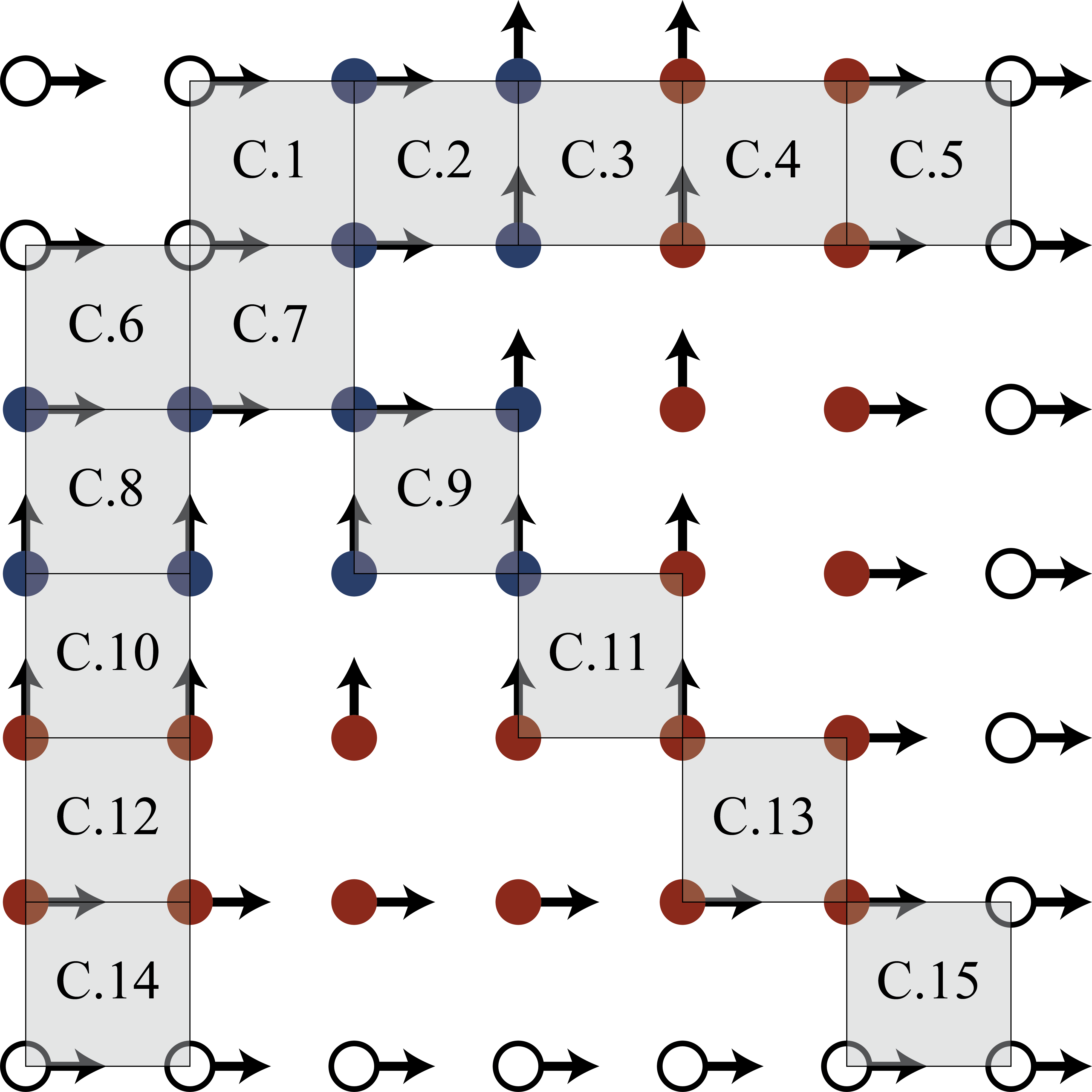}
    \caption{Small boxes of region C of Figure \ref{fig:regions2Analyze}.}
    \label{fig:C}
\end{figure}

We have the following categories:
\begin{itemize}
  \item[$\blacktriangleright$] The small boxes: C.1, C.5, C.6, C.7, C.14, C.15 follow from plain application of Group 1 using that: (1) all the colors decrease linearly as the $x$ coordinate increases, (2) dark blue is everywhere at least $\frac{1}{M}$ smaller than the background, and (3) dark red is everywhere at least $\frac{1}{M}$ larger than the background.
  \item[$\blacktriangleright$] The small boxes: C.3, C.10, C.11 follow from application of Group 1 after taking a reflection with respect to the $y = x$ axis using that: (1) dark blue and dark red decrease with linear rate as the $y$ coordinate increases, and (2) dark red is everywhere at least $\frac{1}{M}$ larger than dark blue.
  \item[$\blacktriangleright$] The small box C.2 follows from Group 2 after applying a $y = x$ reflection.
  \item[$\blacktriangleright$] The small box C.4 follows from Group 2 after applying a $y = -x$ reflection and negation.
  \item[$\blacktriangleright$] The small box C.8 follows from a plain application of Group 2.
  \item[$\blacktriangleright$] The small box C.9 follows from a plain application of Group 3.
  \item[$\blacktriangleright$] The small box C.12 follows from Group 3 after applying the following transformations: (i) reflection with respect to $y$, (ii) reflection with respect to $x$, and (iii) negation.
  \item[$\blacktriangleright$] The small box C.13 follows from Group 3 after applying a $y = -x$ reflection and negation.
\end{itemize}

So no solution appears in C as well.

\paragraph{D.} We start with a figure of the region D in Figure \ref{fig:D}, where we indicate all the small boxes with colors and gradients that have not appeared in A, B, or C.

\begin{figure}
    \centering
    \includegraphics[scale=0.2]{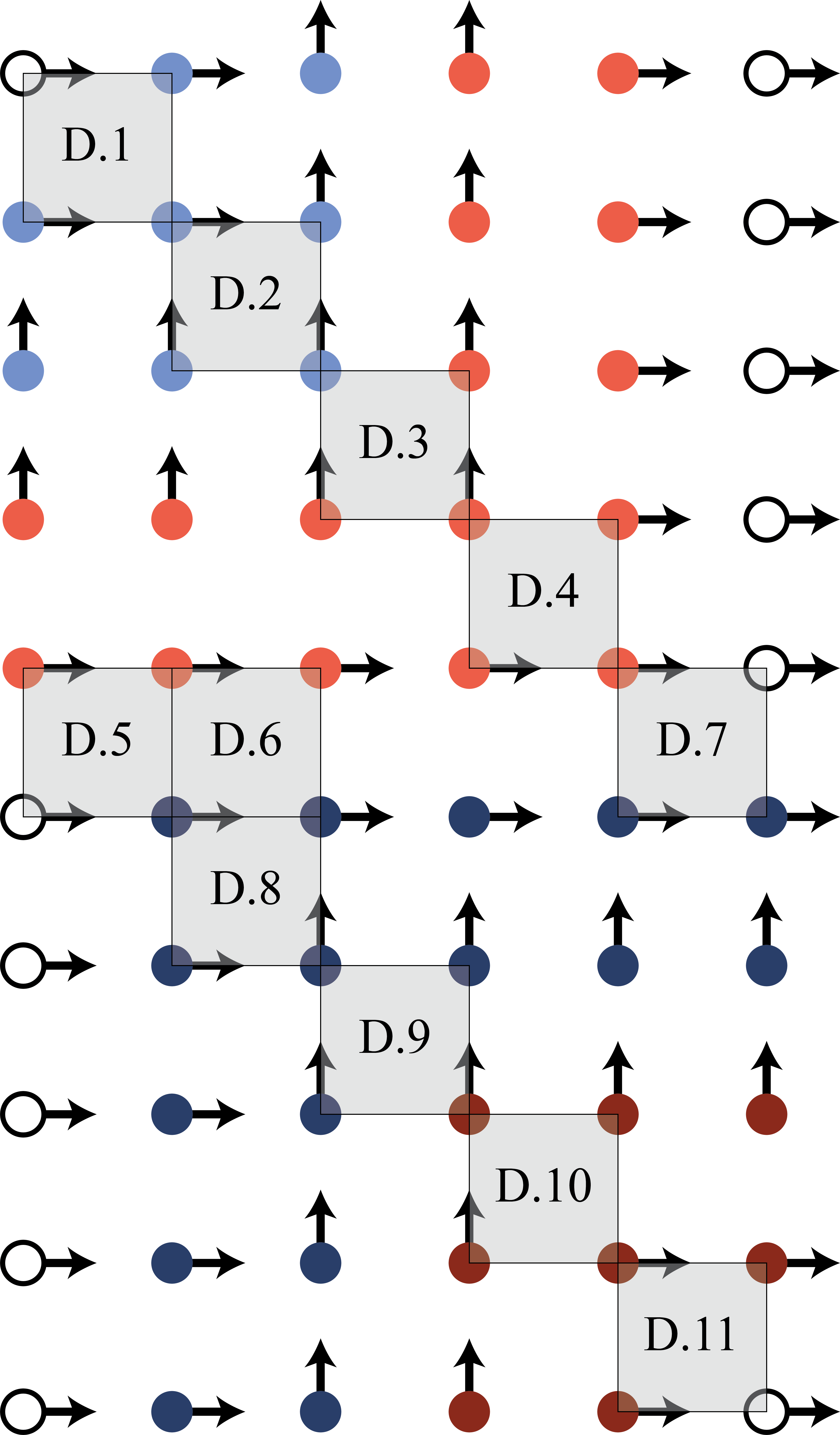}
    \caption{Small boxes of region D of Figure \ref{fig:regions2Analyze}.}
    \label{fig:D}
\end{figure}

\begin{itemize}
  \item[$\blacktriangleright$] The small boxes: D.1, D.5, D.6, D.7, D.11 follow from plain application of Group 1 using that: (1) all the colors decrease linearly as the $x$ coordinate increases, (2) light and dark blue are everywhere at least $\frac{1}{M}$ smaller than the background, (3) light and dark red are everywhere at least $\frac{1}{M}$ larger than the background, and (4) light and dark red are everywhere at least $\frac{1}{M}$ larger than light or dark blue.
  \item[$\blacktriangleright$] The small boxes: D.3, D.9 follow from application of Group 1 after taking a reflection with respect to the $y = x$ axis using that: (1) all blue and red colors decrease with linear rate as the $y$ coordinate increases, and (2) light and dark red are everywhere at least $\frac{1}{M}$ larger than light or dark blue.
  \item[$\blacktriangleright$] The small box D.2 follows from a plain application of Group 3.
  \item[$\blacktriangleright$] The small box D.4 follows from Group 3 after applying a $y = -x$ reflection and negation.
  \item[$\blacktriangleright$] The small box D.8 follows from Group 3 after applying a $y = x$ reflection.
  \item[$\blacktriangleright$] The small box D.10 follows from Group 3 after applying the following transformations: (i) reflection with respect to $y$, (ii) reflection with respect to $x$, and (iii) negation.
\end{itemize}

So no solution appears in D as well.

\paragraph{E.} We start with a figure of the region E in Figure \ref{fig:E}, where we indicate all the small boxes with colors and gradients that have not appeared in A, B, C, or D.

\begin{figure}
    \centering
    \includegraphics[scale=0.2]{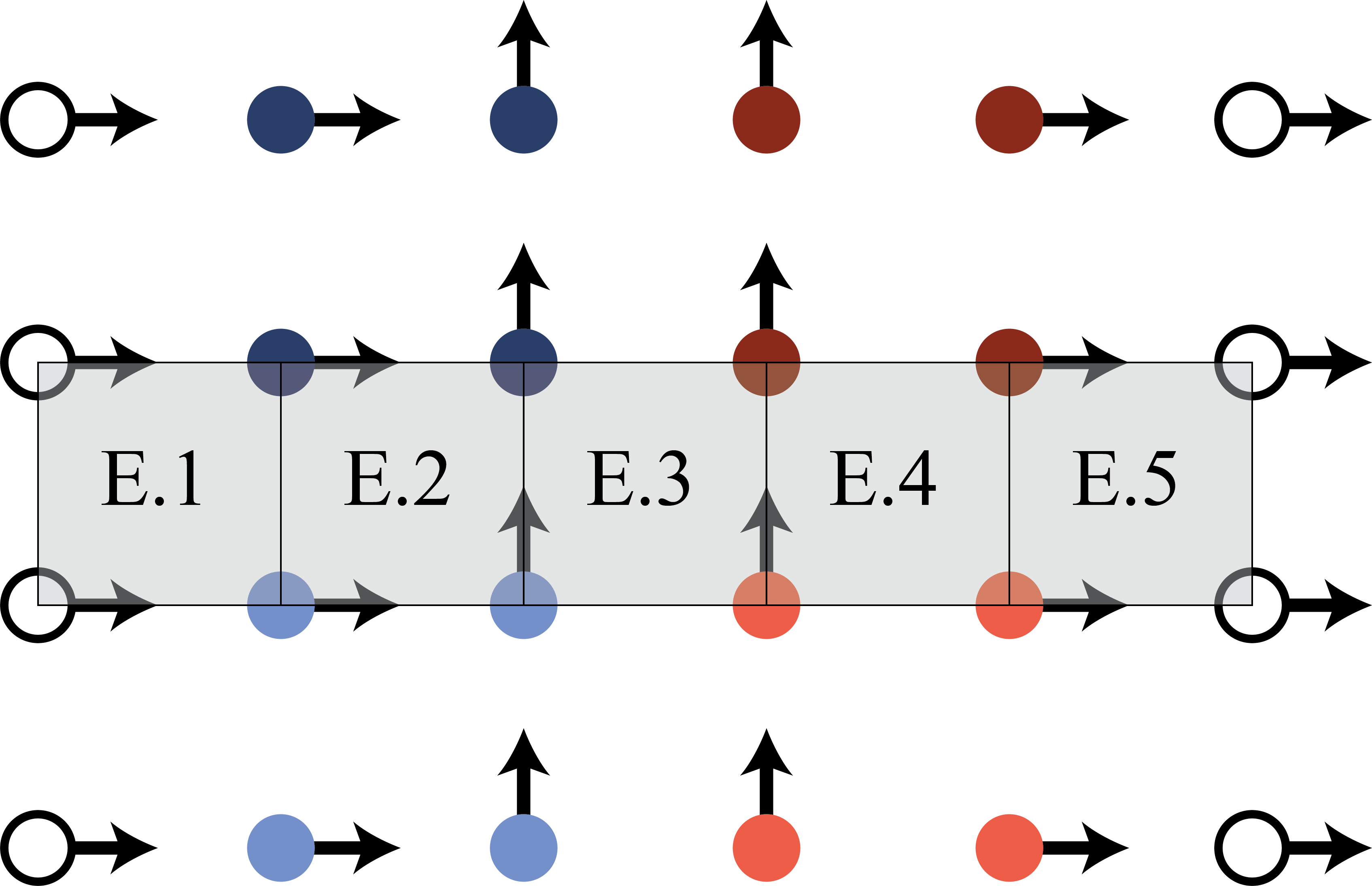}
    \caption{Small boxes of region E of Figure \ref{fig:regions2Analyze}.}
    \label{fig:E}
\end{figure}

\begin{itemize}
  \item[$\blacktriangleright$] The small boxes: E.1, E.5 follow from plain application of Group 1 using that: (1) all the colors decrease linearly as the $x$ coordinate increases, (2) light and dark blue are everywhere at least $\frac{1}{M}$ smaller than the background, and (3) light and dark red are everywhere at least $\frac{1}{M}$ larger than the background.
  \item[$\blacktriangleright$] The small box E.3 follows from application of Group 1 after taking a reflection with respect to the $y = x$ axis using that: (1) light blue is at least $\frac{1}{M}$ larger than dark blue, (2) light read is at least $\frac{1}{M}$ larger than dark red.
  \item[$\blacktriangleright$] The small box E.2 follows from Group 2 after applying a $y = x$ reflection and using the fact that light blue is larger than dark blue.
  \item[$\blacktriangleright$] The small box E.4 follows from Group 2 after applying a $y = -x$ reflection and negation and using the fact that light red is larger than dark red.
\end{itemize}

So no solution appears in E as well.
\medskip

\noindent Since no solutions have appeared in any of A, B, C, D, E and there are no other types of small boxes that appear outside the PLS boxes we conclude that Lemma \ref{lem:noSolutionsLemma} follows.
\end{proof}

\subsection{PLS Boxes} \label{sec:PLSBoxes}

  We are now ready to define the PLS boxes. The construction of the PLS boxes follows the high level idea of the PLS Labyrinth of \cite{FearnleyGHS22-gradient} but adapted to fit the periodic construction that we described above. We start with an instance of the \iter/ problem and we want to encode it inside both of the PLS Box (A) and the PLS Box (B). For the illustration of the construction we use the \iter/ instance that we show in Figure \ref{fig:ITERexample}.

  \begin{figure}
    \centering
    \includegraphics[scale=0.3]{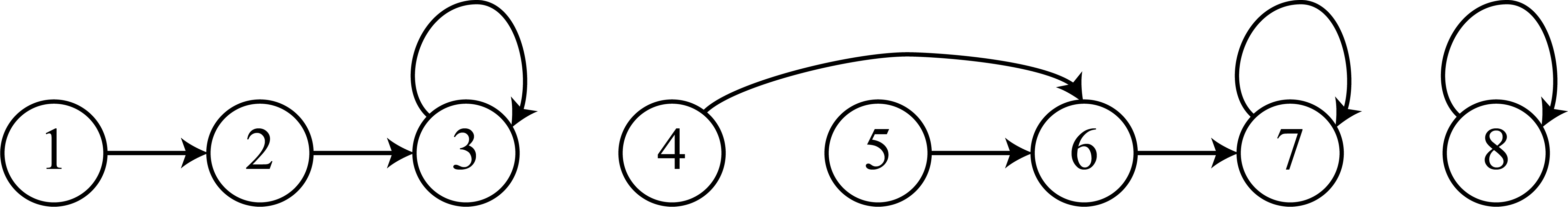}
    \caption{Instance of the \iter/ problem. The nodes correspond to the nodes of the set $[2^n]$ for $n = 3$ and the arrows correspond to the output of the circuit $C$. The solutions are the nodes $2$ and $6$ in this example.}
    \label{fig:ITERexample}
  \end{figure}

  We begin with the description of the PLS Box (A) and then we describe PLS Box (B) which is a symmetric version of (A). Before that we give some general definitions that are needed for both of the boxes. Let assume that we are given an instance of \iter/, i.e., a boolean circuit $C : [2^n] \to [2^n]$ for some $n \in \N$. The PLS Box contains a subgrid of size $N \times N$. We split this subgrid into $2^n \times 2^n$ \textit{medium boxes} of size $8 \times 8$. Hence, we pick $N = 2^{n + 3}$. For simplicity we index all the rows and columns of the PLS boxes starting from $0$ to $N$ ignoring the constant offset of the placement of the PLS boxes inside the grid $G_M$.
  
  \paragraph{PLS Box (A).} In Figure \ref{fig:PLSsubgridA} we show how we split the $N \times N$ subgrid of $G_M$ into $2^n \times 2^n$ medium boxes, and we also show that we use $Q(i, j)$ to indicate the medium box that is in the $i$th row of medium boxes and the $j$th column of medium boxes.

  \begin{figure}
    \centering
    \subfigure[PLS Box (A).]{\label{fig:PLSsubgridA} \includegraphics[scale=0.3]{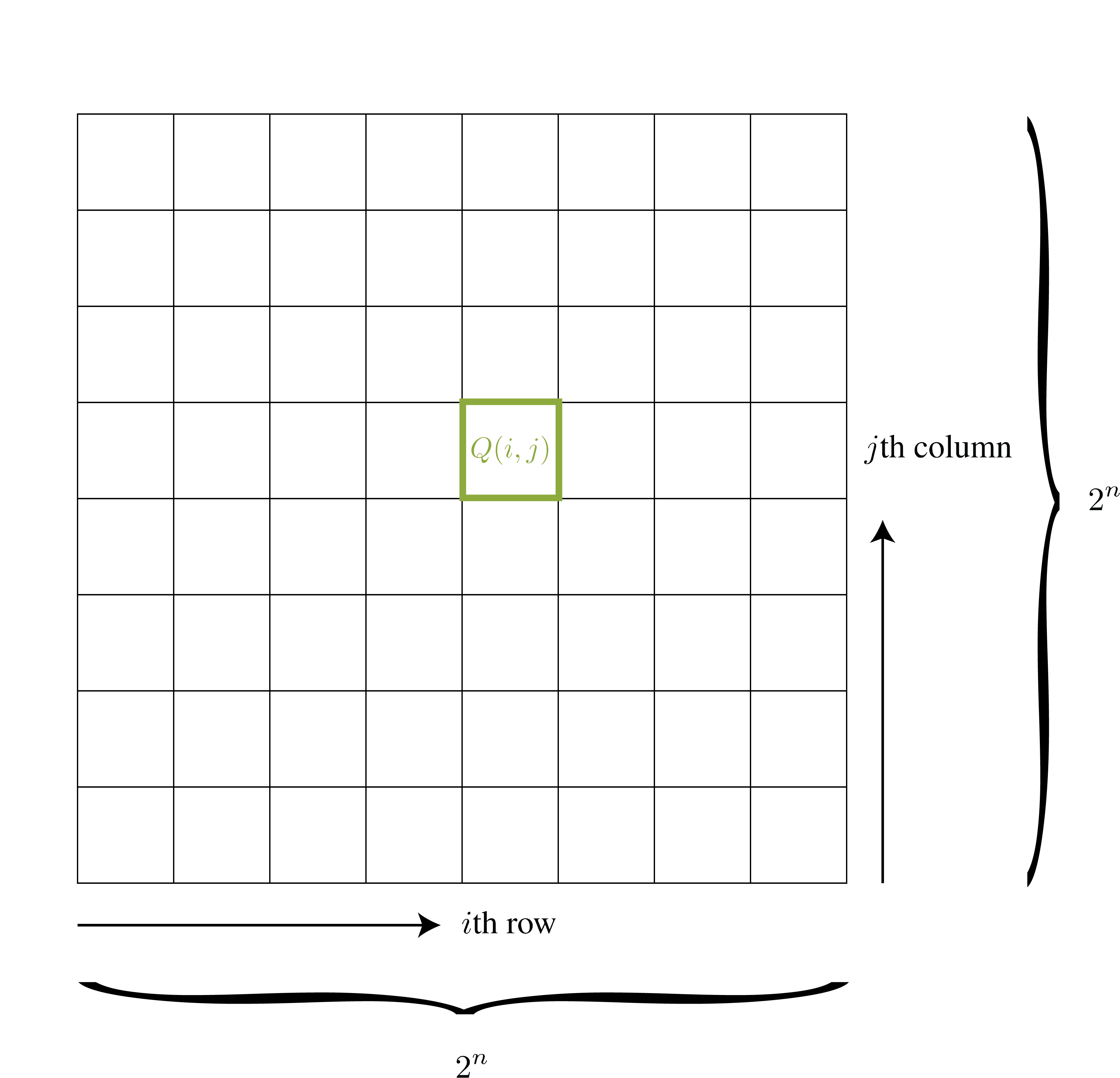}} \\[6pt]
    \subfigure[PLS Box (B).]{\label{fig:PLSsubgridB} \includegraphics[scale=0.3]{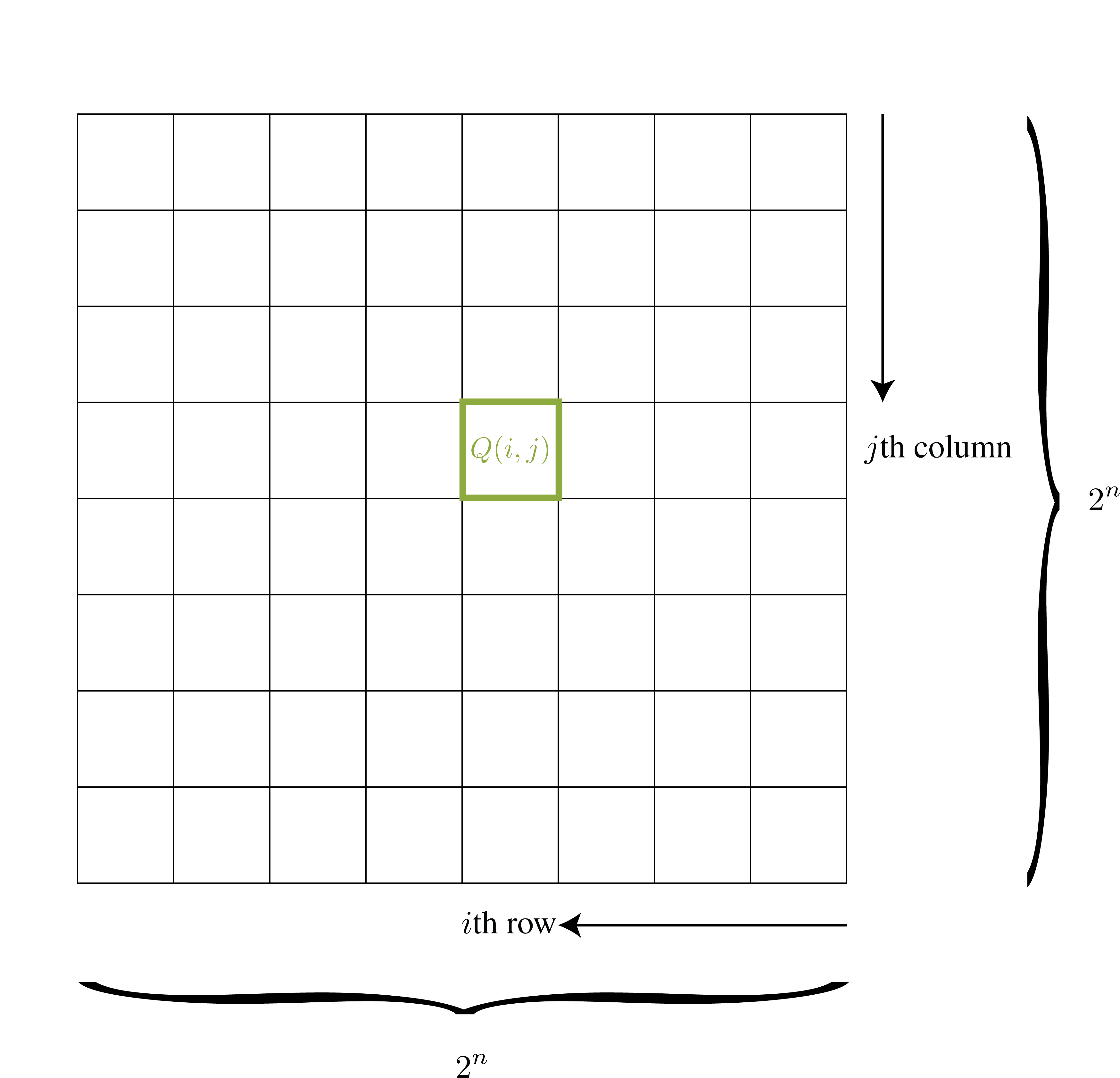}}
    \caption{The medium boxes and their indexing rules in the PLS Box (A) and the PLS Box (B). The arrows indicate the direction of increase of $i$ and $j$ in PLS Box (A) and the PLS Box (B) respectively.}
    \label{fig:PLSsubgrid}
  \end{figure}

  \medskip
  \noindent \textbf{(a) Regions.} The locations of the regions for constructing the PLS Box (A) are the following
  \begin{itemize}
    \item[$\blacktriangleright$] \textbf{Horizontal Dark Red Line.} We recall the definition of the regions outside of the PLS boxes. In particular, as shown in Figure \ref{fig:regionsLocations}, on the bottom left corner of the PLS Box (A) there is a dark blue $\&$ dark red line that goes inside the PLS box. The first thing in the construction of the PLS box is to extend a horizontal line, in the bottom two rows of the PLS Box, with dark red color starting from the bottom of the medium box $Q(1,1)$ all the way until the box $Q(2^n, 1)$ as shown in the example of Figure \ref{fig:PLSboxA}.
    \item[$\blacktriangleright$] \textbf{Initial Node.} The initial node corresponds to the node $1$ of the \iter/ instance. For example see Figure \ref{fig:ITERexample}. The implementation of this initial node is to continue vertically the dark blue $\&$ dark red line that touches the PLS Box (A) from outside for 4 rows up.
    \item[$\blacktriangleright$] \textbf{Nodes $\boldsymbol{u \in [2^n]}$ with $\boldsymbol{C(u) > u}$.} These are the nodes of \iter/ that do not have a self-loop and are potential solutions to \iter/ depending on whether $C(C(u)) > C(u)$ or not. Independently of whether $u$ is a solution, for every $u$ such that $C(u) > u$ we start from the $Q(u, 1)$ medium box until the $Q(u, u)$ box a vertical dark blue $\&$ dark red line where each of the colors have width $2$. The $4$ columns that this line uses are the $4$ columns in the middle of $Q(u, 1)$, i.e., the $4$ small columns in the middle of the $u$th column of medium boxes. Inside $Q(u, 1)$ this vertical dark blue $\&$ dark red line starts from the third row of $Q(u, 1)$, i.e., right above the horizontal dark red line that we described above. Inside $Q(u, u)$ the vertical dark blue $\&$ dark red line stops at the fourth row of $Q(u, u)$, i.e., in the middle of the medium box $Q(u, u)$.
    \item[$\blacktriangleright$] \textbf{Connection of $\boldsymbol{u \in [2^n]}$ with $\boldsymbol{C(u)}$ if $\boldsymbol{C(u) > u \textbf{ and } C(C(u)) > C(u)}$.} If $C(u) > u$ then already as we described construct a vertical dark blue $\&$ dark red line that ends in the middle of $Q(u, u)$. Now if we additionally have that $C(C(u)) > C(u)$ then in the middle of $Q(u, u)$ right above the end of the vertical dark blue $\&$ dark red line we start a horizontal dark blue line with width $2$ that continues until the medium box $Q(C(u), u)$. Since $C(C(u)) > C(u)$ and $C(u) > u$ this means that according to the rule above there will be a vertical dark blue $\&$ dark red line that goes through the box $Q(C(u), u)$. This dark blue $\&$ dark red line will have the blue to the left and the red the right. Our horizontal dark blue line will stop once it hits the vertical dark blue line in $Q(C(u), u)$.
    \item[$\blacktriangleright$] \textbf{Crossing between horizontal dark blue lines and vertical dark blue $\&$ dark red lines.} Depending on the instance, it could be that there exist two vertices $u, v \in [2^n]$ such that: $C(u) > u$, $C(C(u)) > C(u)$, $u < v < C(u)$, and $C(v) > v$. For example this is true in the example of Figure \ref{fig:ITERexample} with $u = 4$ and $v = 5$. In that case, in the box $Q(v, u)$ the horizontal dark blue line from $Q(u, u)$ to $Q(C(u), u)$ and the vertical dark blue $\&$ dark red line from $Q(v, 1)$ to $Q(v, v)$ cross. When this happens then inside $Q(v, u)$ the vertical dark blue $\&$ dark red line overwrites the horizontal dark blue line and otherwise the lines occupy the small boxes that they were supposed to.
    \item[$\blacktriangleright$] \textbf{Background.} Every point in the grid that does not belong to any of the aforementioned regions belongs to the background which is defined the same way that we defined the background in Section \ref{sec:structureOfG}.
  \end{itemize}

  We illustrate the above construction for the \iter/ example of Figure \ref{fig:ITERexample} in Figure \ref{fig:PLSboxA}.

  \begin{figure}[t]
    \centering
    \includegraphics[scale=0.4]{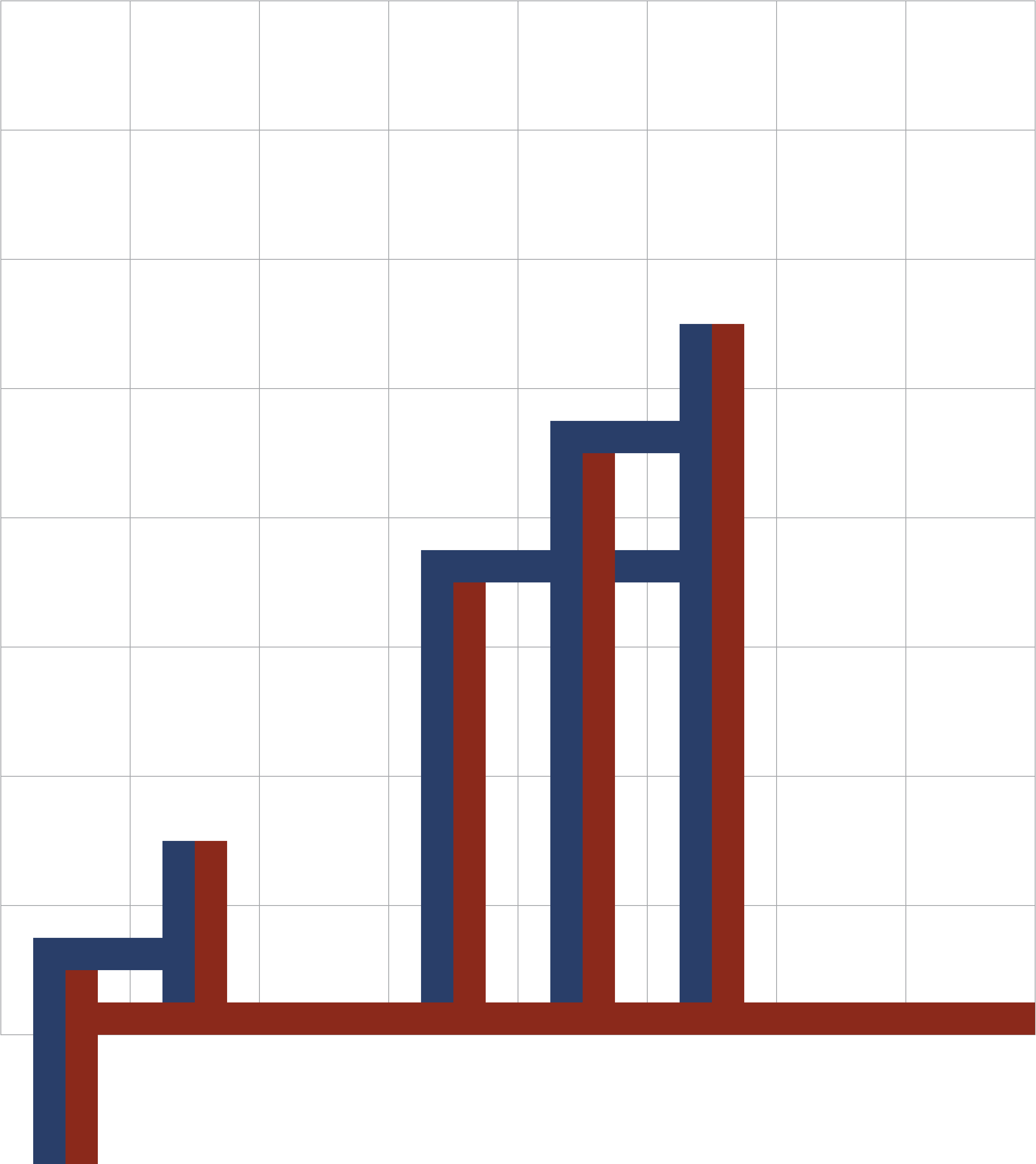}
    \caption{The regions of PLS Box (A) when implementing the \iter/ instance provided in Figure \ref{fig:ITERexample}.}
    \label{fig:PLSboxA}
  \end{figure}
  
  \medskip

  \noindent \textbf{(b) Function Values.} The function values are defined exactly the same as we described in Section \ref{sec:structureOfG} for the regions: dark blue, dark red, and background. In particular, using the functions $h_B$, $h_{DB}$, $h_{DR}$.
  \medskip

  \noindent \textbf{(c) Gradient Values.} The gradient values follow the following rules:
  \begin{itemize}
    \item[$\blacktriangleright$] \textbf{Horizontal dark red lines.} The gradients in the horizontal dark red lines are always $(-\delta, 0)$ expect from the following places: for any $u \in [2^n]$ such that $C(u) > u$, in the middle of the medium box $Q(u, 1)$ in the row before the vertical dark blue $\&$ dark red line and directly below the right dark blue line and the left dark red line in these two grid points the gradient is $(0, -\delta)$.
    \item[$\blacktriangleright$] \textbf{Vertical dark blue $\&$ dark red lines.} In this line the left dark blue lines has gradients $(-\delta, 0)$, the right dark blue line has gradients $(0, -\delta)$, the left dark red line has gradients $(0, -\delta)$ and the right dark red line has gradients $(-\delta, 0)$.
    \item[$\blacktriangleright$] \textbf{Horizontal dark blue lines.} The gradients in the horizontal dark blue lines are always $(-\delta, 0)$ expect from the following places: for any $u \in [2^n]$ such that $C(u) > u$ and $C(C(u)) > C(u)$ the, in the middle of the medium box $Q(u, u)$ in the row above the vertical dark blue $\&$ dark red line and directly above the right dark blue line and the left dark red line in these two grid points the gradient is $(0, -\delta)$.
    \item[$\blacktriangleright$] \textbf{Background.} The gradient in the grid points of the background is always $(-\delta, 0)$.
  \end{itemize}

  The rest of the function in PLS Box (A) is defined via the bicubic interpolation that we described in Section \ref{sec:interpolation}. This completes the description of the PLS Box (A).

  \paragraph{PLS Box (B).} We again split the $N \times N$ subgrid of $G_M$ into $2^n \times 2^n$ medium boxes, and we use $Q(i, j)$ to indicate the medium box that is in the $i$th row of medium boxes and the $j$th column of medium boxes but when instead of starting the counting from the bottom left corner, we start the counting from the top right corner, as we show in Figure \ref{fig:PLSsubgridB}.
  \medskip
  
  \noindent \textbf{(a) Regions.} The locations of the regions for constructing the PLS Box (B) are the following
  \begin{itemize}
    \item[$\blacktriangleright$] \textbf{Horizontal Light Blue Line.} We recall the definition of the regions outside of the PLS boxes. In particular, as shown in Figure \ref{fig:regionsLocations}, on the top right corner of the PLS Box (B) there is a light blue $\&$ light red line that gets inside the PLS box. The first thing in the construction of the PLS box is to extend a horizontal line, in the top two rows of the PLS Box, with light blue color starting from the top of the medium box $Q(1,1)$ all the way until the box $Q(2^n, 1)$ as shown in the example of Figure \ref{fig:PLSboxB}.
    \item[$\blacktriangleright$] \textbf{Initial Node.} The initial node corresponds to the node $1$ of the \iter/ instance. For example see Figure \ref{fig:ITERexample}. The implementation of this initial node is to continue vertically the light blue $\&$ light red line that touches the PLS Box (B) from outside for 4 rows down.
    \item[$\blacktriangleright$] \textbf{Nodes $\boldsymbol{u \in [2^n]}$ with $\boldsymbol{C(u) > u}$.} These are the nodes of \iter/ that do not have a self-loop and are potential solutions to \iter/ depending on whether $C(C(u)) > C(u)$ or not. Independently of whether $u$ is a solution, for every $u$ such that $C(u) > u$ we start from the $Q(u, 1)$ medium box until the $Q(u, u)$ box a downwards vertical light blue $\&$ light red line where each of the colors have width $2$. The $4$ columns that this line uses are the $4$ columns in the middle of $Q(u, 1)$, i.e., the $4$ small columns in the middle of the $u$th column of medium boxes. Inside $Q(u, 1)$ this vertical light blue $\&$ light red line starts from the third row of $Q(u, 1)$, i.e., right above the horizontal light blue line that we described above. Inside $Q(u, u)$ the vertical light blue $\&$ light red line stops at the fourth row of $Q(u, u)$, i.e., in the middle of the medium box $Q(u, u)$.
    \item[$\blacktriangleright$] \textbf{Connection of $\boldsymbol{u \in [2^n]}$ with $\boldsymbol{C(u)}$ if $\boldsymbol{C(u) > u \textbf{ and } C(C(u)) > C(u)}$.} If $C(u) > u$ then already as we described construct a downwards vertical light blue $\&$ light red line that ends in the middle of $Q(u, u)$. Now if we additionally have that $C(C(u)) > C(u)$ then in the middle of $Q(u, u)$ right below the end of the vertical light blue $\&$ light red line we start a horizontal light red line with width $2$ that continues until the medium box $Q(C(u), u)$. Since $C(C(u)) > C(u)$ and $C(u) > u$ this means that according to the rule above there will be a vertical light blue $\&$ light red line that goes downwards through the box $Q(C(u), u)$. This light blue $\&$ light red line will have the blue to the left and the red the right. Our horizontal light red line, which in this case is coming from the right, will stop once it hits the vertical light red line in $Q(C(u), u)$.
    \item[$\blacktriangleright$] \textbf{Crossing between horizontal light red lines and vertical light blue $\&$ light red lines.} Depending on the instance, it could be that there exist two vertices $u, v \in [2^n]$ such that: $C(u) > u$, $C(C(u)) > C(u)$, $u < v < C(u)$, and $C(v) > v$. For example this is true in the example of Figure \ref{fig:ITERexample} with $u = 4$ and $v = 5$. In that case, in the box $Q(v, u)$ the horizontal light red line from $Q(u, u)$ to $Q(C(u), u)$ and the vertical light blue $\&$ light red line from $Q(v, 1)$ to $Q(v, v)$ cross. When this happens then inside $Q(v, u)$ the vertical light blue $\&$ light red line overwrites the horizontal light red line and otherwise the lines occupy the small boxes that they were supposed to.
    \item[$\blacktriangleright$] \textbf{Background.} Every point in the grid that does not belong to any of the aforementioned regions belongs to the background which is defined the same way that we defined the background in Section \ref{sec:structureOfG}.
  \end{itemize}

  We illustrate the above construction for the \iter/ example of Figure \ref{fig:ITERexample} in Figure \ref{fig:PLSboxB}.

  \begin{figure}[t]
    \centering
    \includegraphics[scale=0.4]{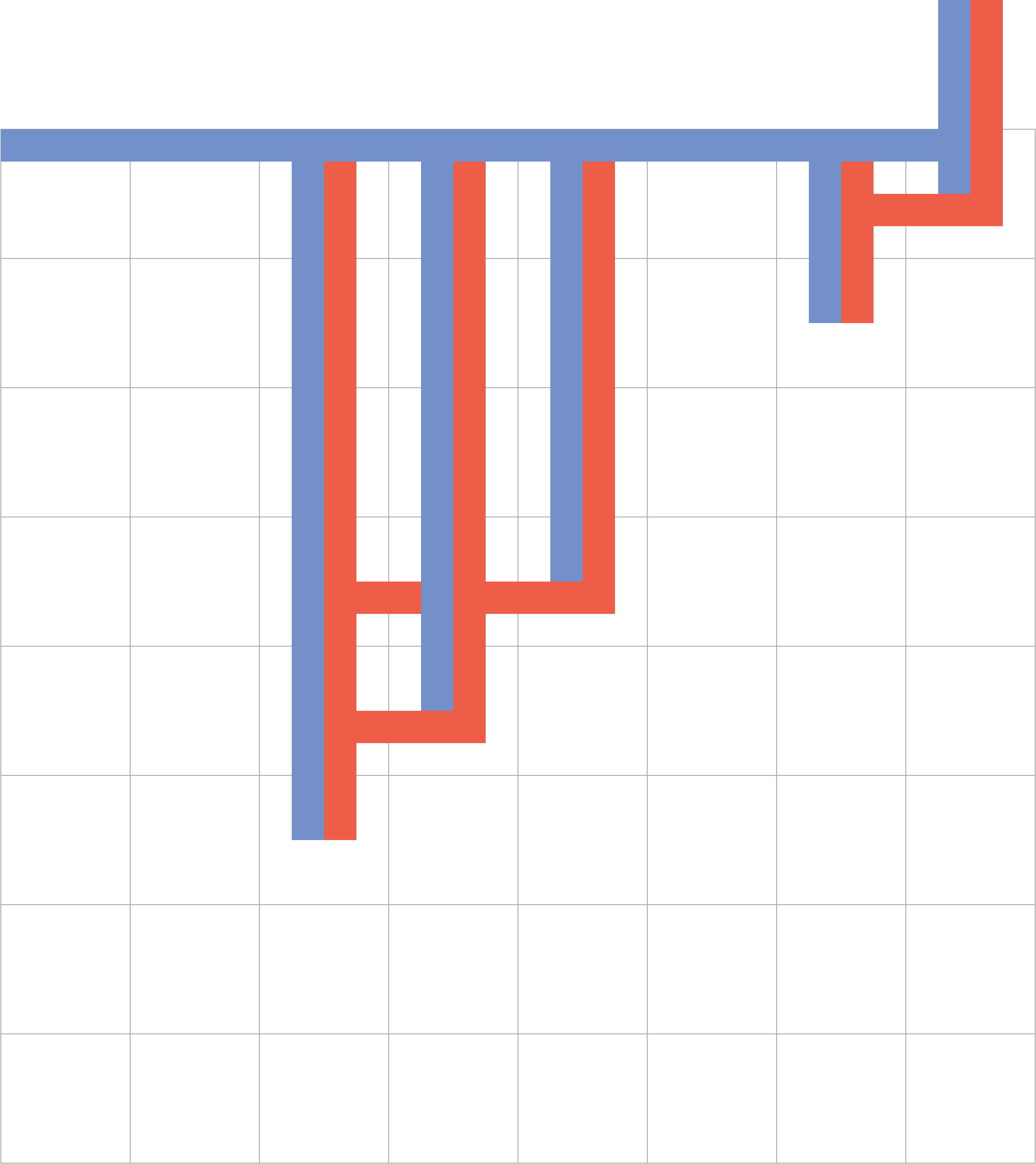}
    \caption{The regions of PLS Box (B) when implementing the \iter/ instance provided in Figure \ref{fig:ITERexample}.}
    \label{fig:PLSboxB}
  \end{figure}
  \medskip

  \noindent \textbf{(b) Function Values.} The function values are defined exactly the same as we described in Section \ref{sec:structureOfG} for the regions: light blue, light red, and background. In particular, using the functions $h_B$, $h_{LB}$, $h_{LR}$.
  \medskip

  \noindent \textbf{(c) Gradient Values.} The gradient values follow the following rules:
  \begin{itemize}
    \item[$\blacktriangleright$] \textbf{Horizontal light blue lines.} The gradients in the horizontal light blue lines are always $(-\delta, 0)$ expect from the following places: for any $u \in [2^n]$ such that $C(u) > u$, in the middle of the medium box $Q(u, 1)$ in the row after the vertical light blue $\&$ light red line and directly above the right light blue line and the left light red line, in these two grid points, the gradient is $(0, -\delta)$.
    \item[$\blacktriangleright$] \textbf{Vertical light blue $\&$ light red lines.} In this line the left light blue lines has gradients $(-\delta, 0)$, the right light blue line has gradients $(0, -\delta)$, the left light red line has gradients $(0, -\delta)$ and the light dark red line has gradients $(-\delta, 0)$.
    \item[$\blacktriangleright$] \textbf{Horizontal light red lines.} The gradients in the horizontal light red lines are always $(-\delta, 0)$ expect from the following places: for any $u \in [2^n]$ such that $C(u) > u$ and $C(C(u)) > C(u)$ the, in the middle of the medium box $Q(u, u)$ in the row below the vertical light blue $\&$ light red line and directly below the right light blue line and the left light red line, in these two grid points, the gradient is $(0, -\delta)$.
    \item[$\blacktriangleright$] \textbf{Background.} The gradient in the grid points of the background is always $(-\delta, 0)$.
  \end{itemize}

  The rest of the function in PLS Box (B) is defined via the bicubic interpolation that we described in Section \ref{sec:interpolation}. This completes the description of the PLS Box (B).

  We are now ready to prove the following lemma.

  \begin{lemma} \label{lem:noSolutionsInsidePLSbox}
    After interpolating using the techniques discussed in Section \ref{sec:interpolation}, there is no $0.01$-stationary point in any medium box of PLS Box (A) and PLS Box (B), except medium boxes $Q(u, u)$ for which $C(u) > u$ and $C(C(u)) = C(u)$.
  \end{lemma}

  \begin{proof}(of Lemma \ref{lem:noSolutionsInsidePLSbox})
    For this proof we will again use the groups of small boxes that we introduced in the proof of Lemma \ref{lem:noSolutionsLemma}.

    We start with identifying the regions in PLS Box (A) and PLS Box (B) that we need to check, shown in Figures \ref{fig:PLSboxAAnalysis} and \ref{fig:PLSboxBAnalysis} respectively. In particular we need to check the following regions in PLS Box (A):
    \begin{enumerate}
      \item[F.] The initial node in PLS Box (A) where the horizontal dark red line starts.
      \item[G.] The start of a vertical dark blue $\&$ dark red line.
      \item[H.] The end of a vertical dark blue $\&$ dark red line and start of a horizontal dark blue line.
      \item[I.] The crossing of a horizontal dark blue line and a vertical dark blue $\&$ dark red line.
      \item[J.] The end of a horizontal dark blue line.
      \item[K.] The end of the horizontal dark red line.
    \end{enumerate}
    and symmetrically the following regions in PLS Box (B):
    \begin{enumerate}
      \item[L.] The initial node in PLS Box (B) where the horizontal light blue line starts.
      \item[M.] The start of a vertical light blue $\&$ light red line.
      \item[N.] The end of a vertical light blue $\&$ light red line and start of a horizontal light red line.
      \item[O.] The crossing of a horizontal light red line and a vertical light blue $\&$ light red line.
      \item[P.] The end of a horizontal light red line.
      \item[R.] The end of the horizontal light blue line.
    \end{enumerate}
    These regions are shown in Figures \ref{fig:PLSboxAAnalysis} and \ref{fig:PLSboxBAnalysis}. We intentionally left out of the above lists the end of the vertical dark blue $\&$ dark red lines with the start of a horizontal dark blue line because these correspond exactly to medium boxes $Q(u, u)$ such that $C(u) > u$ but $C(C(u)) = C(u)$ which correspond to solutions of the \iter/ instance. Hence if there are solutions in these medium boxes then we recover a solution of the \iter/ instance which is what we want. For the same reason we left out of the above lists the end of a vertical light blue $\&$ light red path without the start of a horizontal light red path.

    \begin{figure}[t]
      \centering
      \includegraphics[scale=0.45]{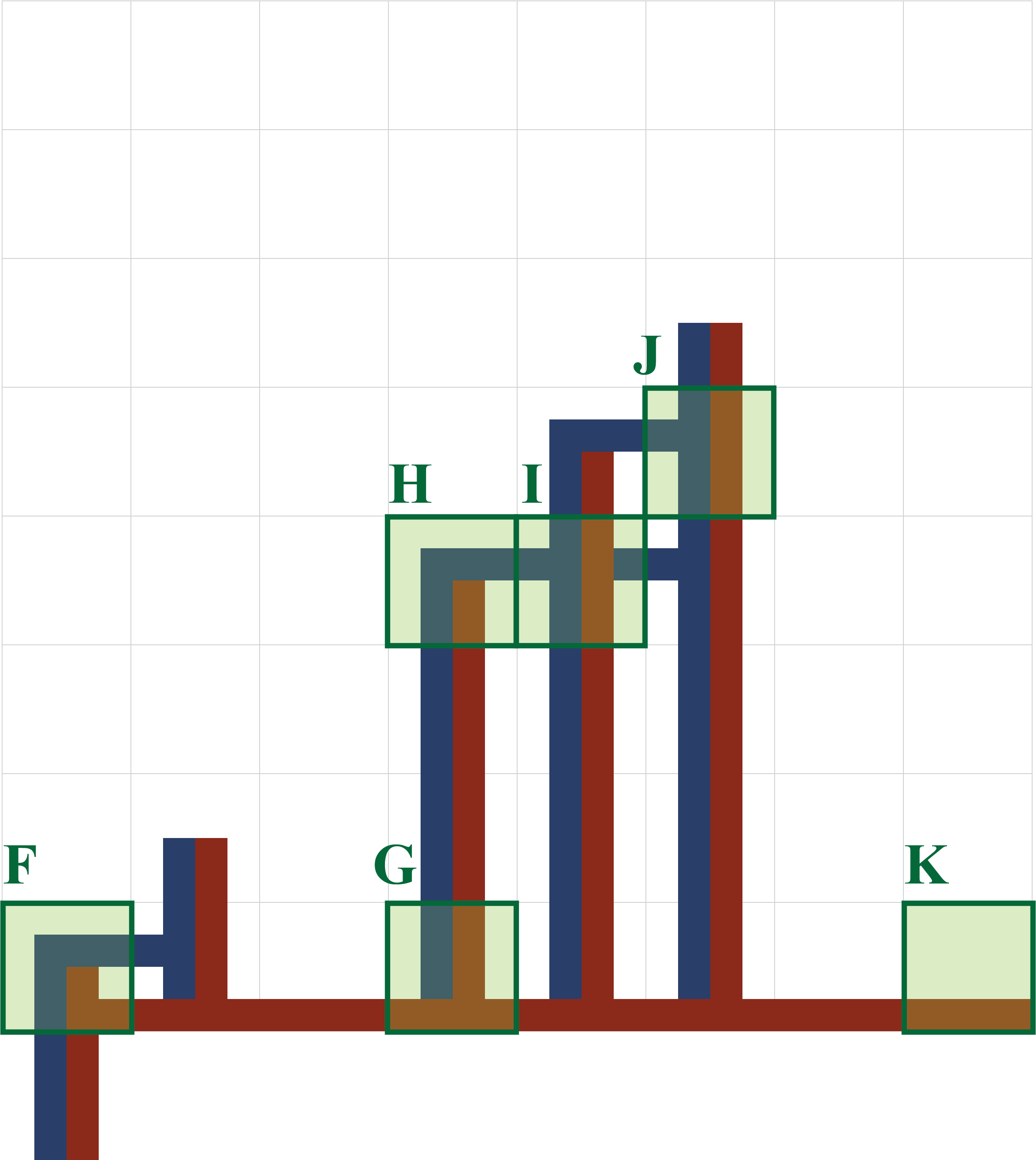}
      \caption{With the green squares we indicate the regions, with their labels, that we need to analyze in the proof of Lemma \ref{lem:noSolutionsInsidePLSbox} for the case of PLS Box (A).}
      \label{fig:PLSboxAAnalysis}
    \end{figure}
    
    \begin{figure}[t]
      \centering
      \includegraphics[scale=0.45]{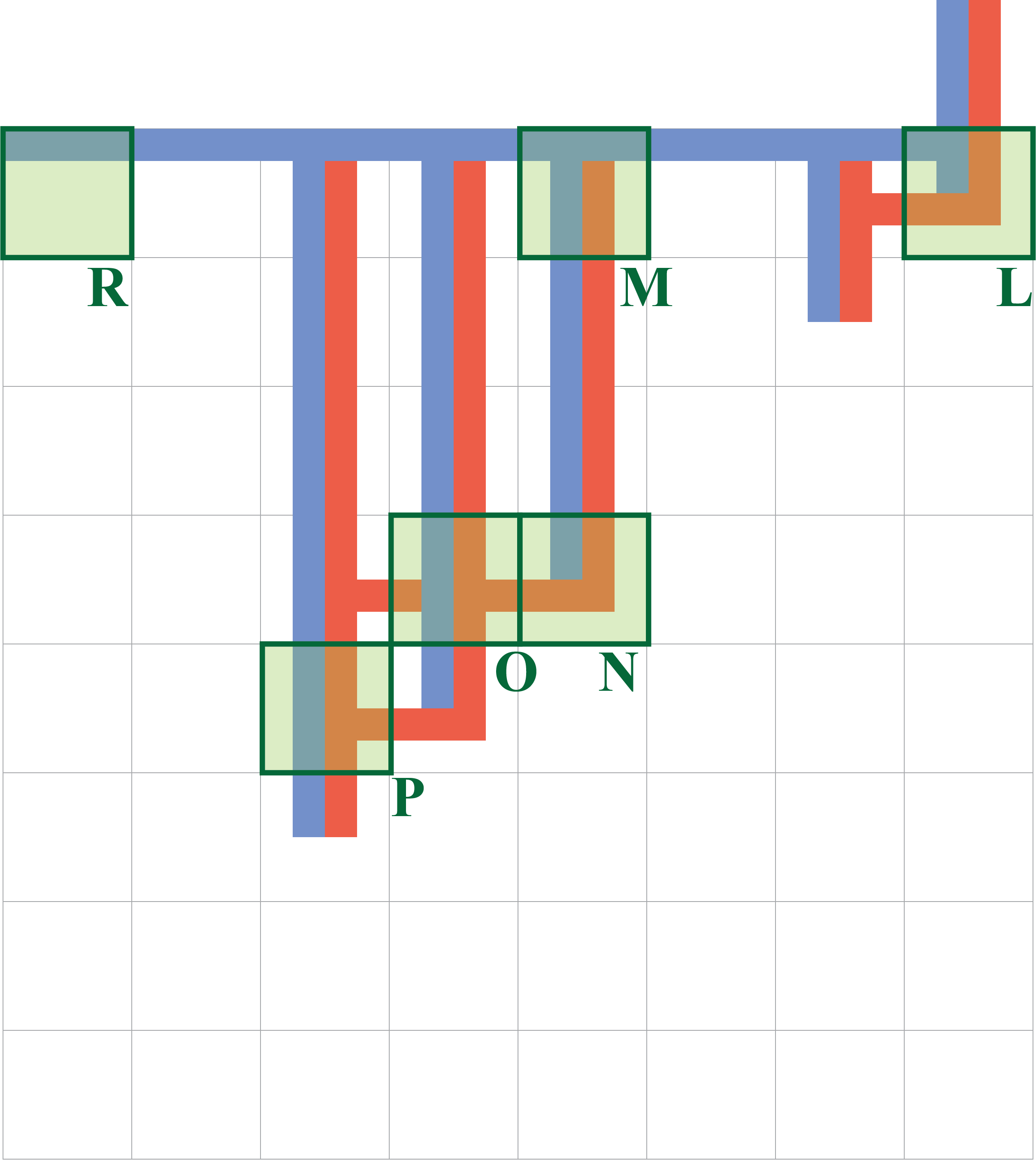}
      \caption{With the green squares we indicate the regions, with their labels, that we need to analyze in the proof of Lemma \ref{lem:noSolutionsInsidePLSbox} for the case of PLS Box (B).}
      \label{fig:PLSboxBAnalysis}
    \end{figure}

\paragraph{F.} We start with a figure of the region F in Figure \ref{fig:F}, where we indicate all the small boxes with colors and gradients that have not appeared in A - E.

\begin{figure}
    \centering
    \includegraphics[scale=0.2]{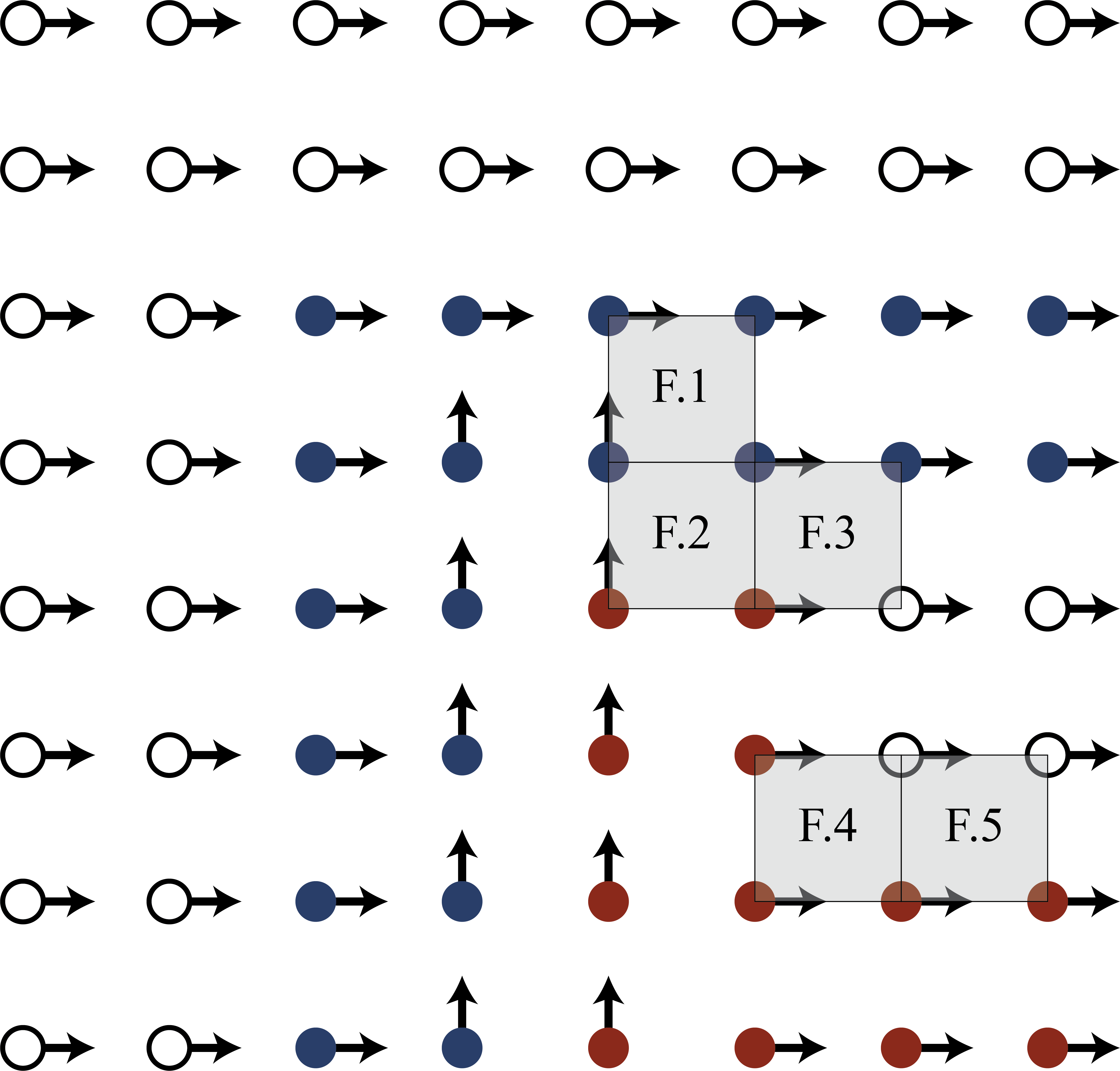}
    \caption{The small boxes of the region F shown in Figure \ref{fig:PLSboxAAnalysis}.}
    \label{fig:F}
\end{figure}

Most of the small boxes have appeared before except from the following:
\begin{itemize}
  \item[$\blacktriangleright$] The small box F.1 follows from Group 4 after applying a $y = x$ reflection.
  \item[$\blacktriangleright$] The small box F.2 follows from Group 2 after applying a $y = - x$ reflection and negation.
  \item[$\blacktriangleright$] The small boxes F.3, F.4, F.5 follow from plain application of Group 1 using that: (1) all the colors decrease linearly as the $x$ coordinate increases, (2) dark red is everywhere at least $1$ larger than the background.
\end{itemize}

So no solution appears in F.

\paragraph{G.} We start with a figure of the region G in Figure \ref{fig:G}, where we indicate all the small boxes with colors and gradients that have not appeared in A - F.

\begin{figure}
    \centering
    \includegraphics[scale=0.2]{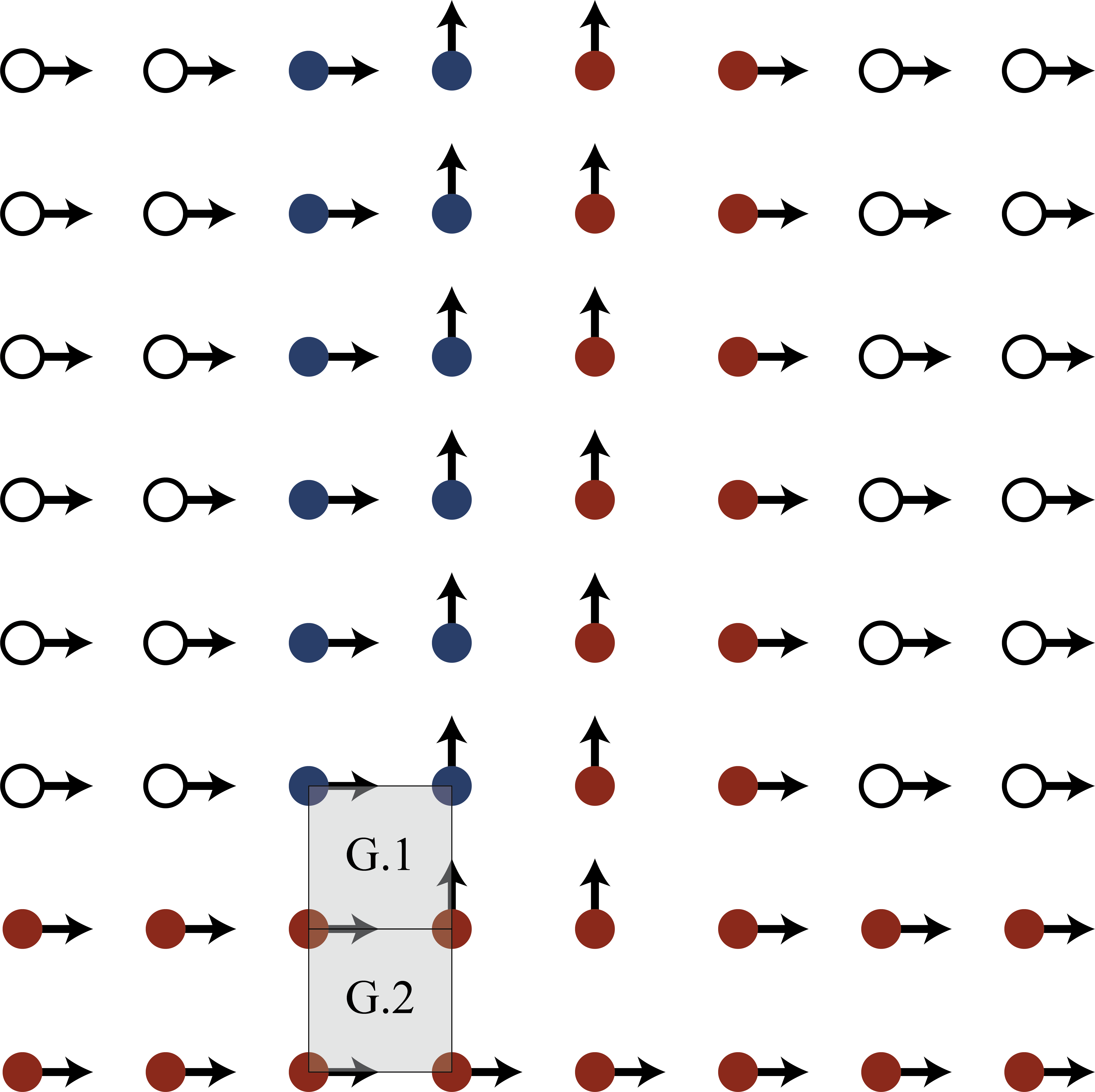}
    \caption{The small boxes of the region G shown in Figure \ref{fig:PLSboxAAnalysis}.}
    \label{fig:G}
\end{figure}

Most of the small boxes have appeared before except from the following:
\begin{itemize}
  \item[$\blacktriangleright$] The small box G.1 follows from Group 2 after applying a $y = x$ reflection.
  \item[$\blacktriangleright$] The small box G.2 follows from Group 4 after applying a $y = - x$ reflection and negation.
\end{itemize}

So no solution appears in G.

\paragraph{H.} In region H, as we can see in Figure \ref{fig:H}, all the small boxes have appeared before in regions A - G and hence we can directly conclude that there are no solutions in H.

\begin{figure}
    \centering
    \includegraphics[scale=0.2]{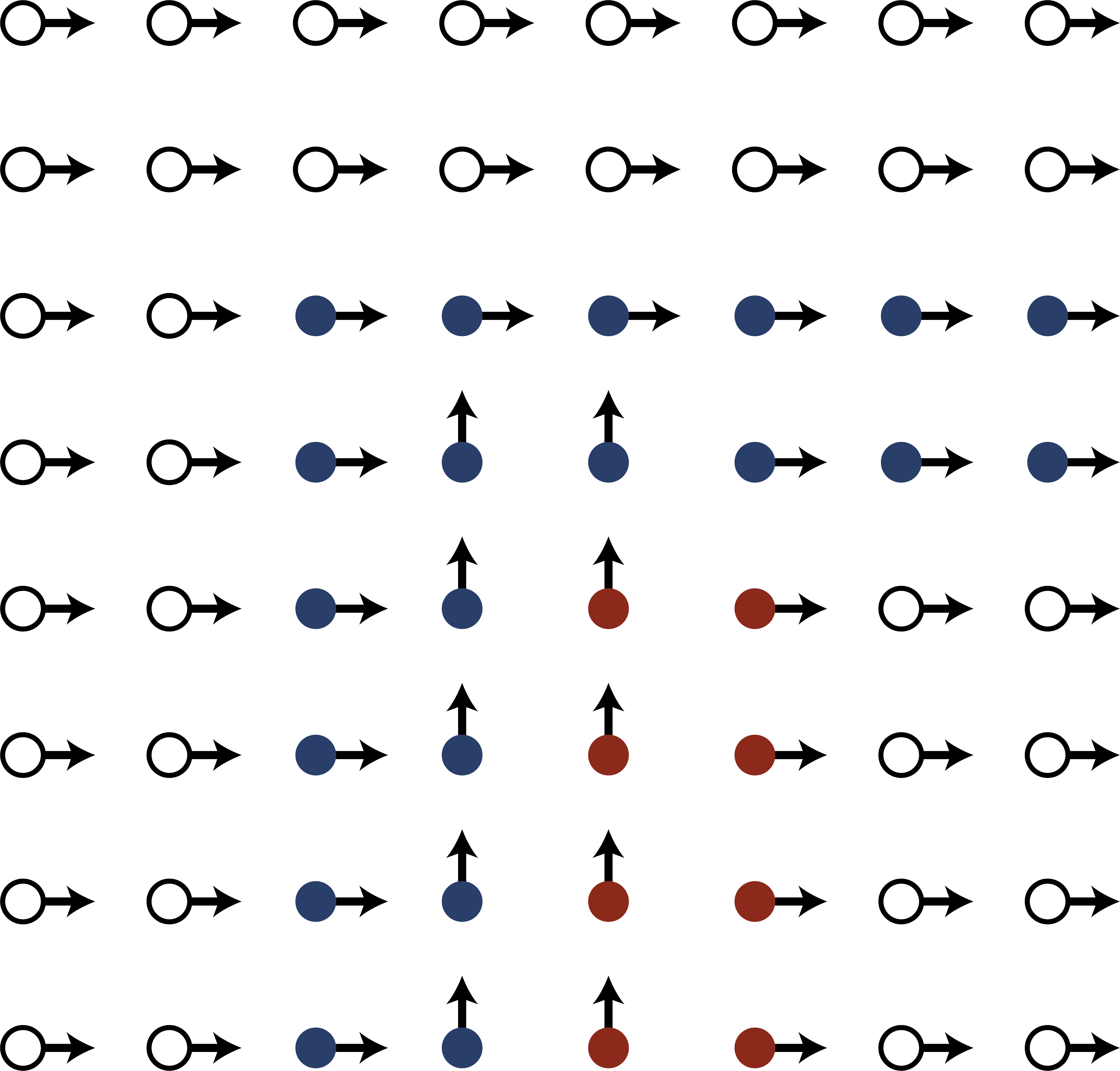}
    \caption{The small boxes of the region H shown in Figure \ref{fig:PLSboxAAnalysis}.}
    \label{fig:H}
\end{figure}

\paragraph{I.} In region I, as we can see in Figure \ref{fig:I}, all the small boxes have appeared before in regions A - G and hence we can directly conclude that there are no solutions in I.

\begin{figure}
    \centering
    \includegraphics[scale=0.2]{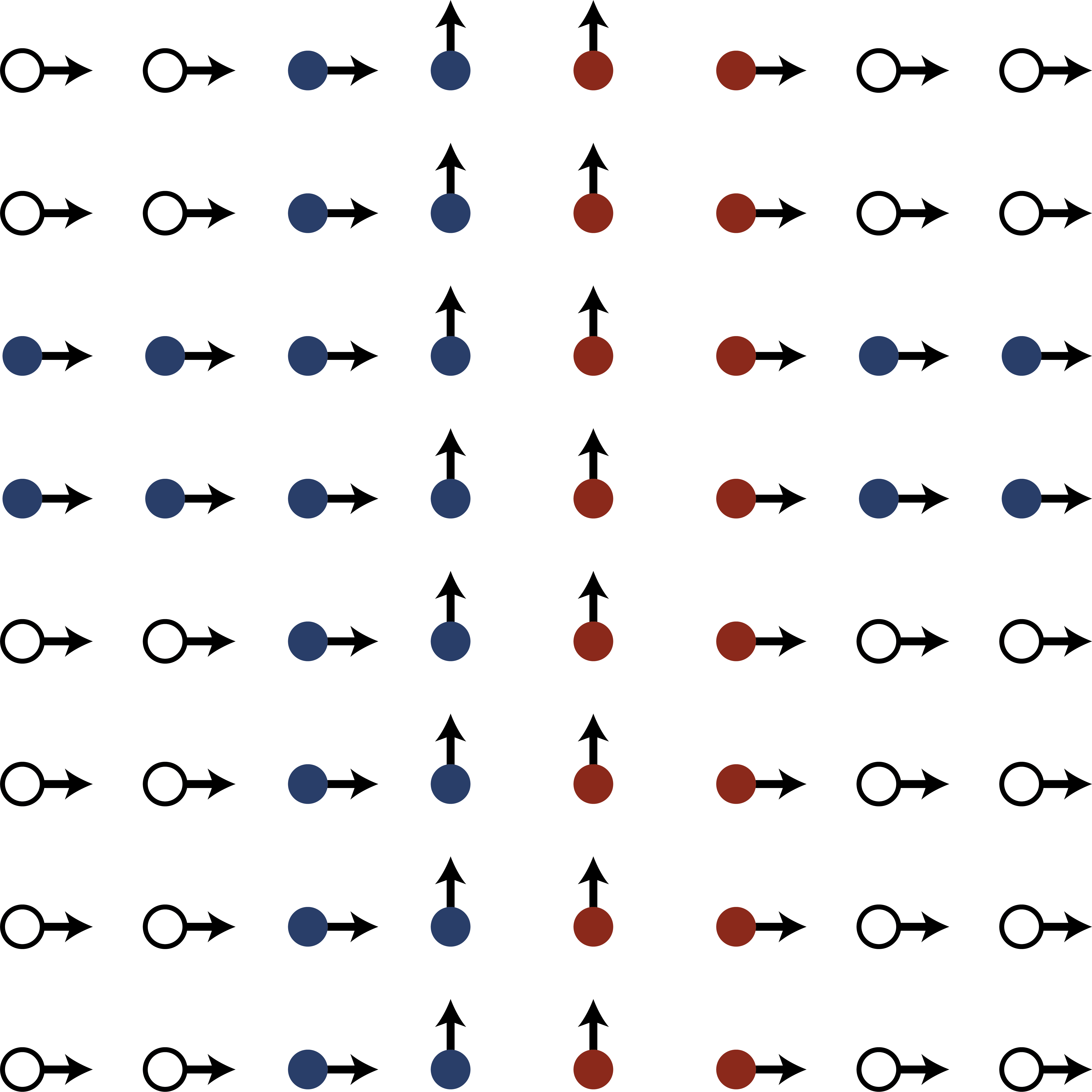}
    \caption{The small boxes of the region I shown in Figure \ref{fig:PLSboxAAnalysis}.}
    \label{fig:I}
\end{figure}

\paragraph{J.} In region J, as we can see in Figure \ref{fig:J}, all the small boxes have appeared before in regions A - G and hence we can directly conclude that there are no solutions in J.

\begin{figure}
    \centering
    \includegraphics[scale=0.2]{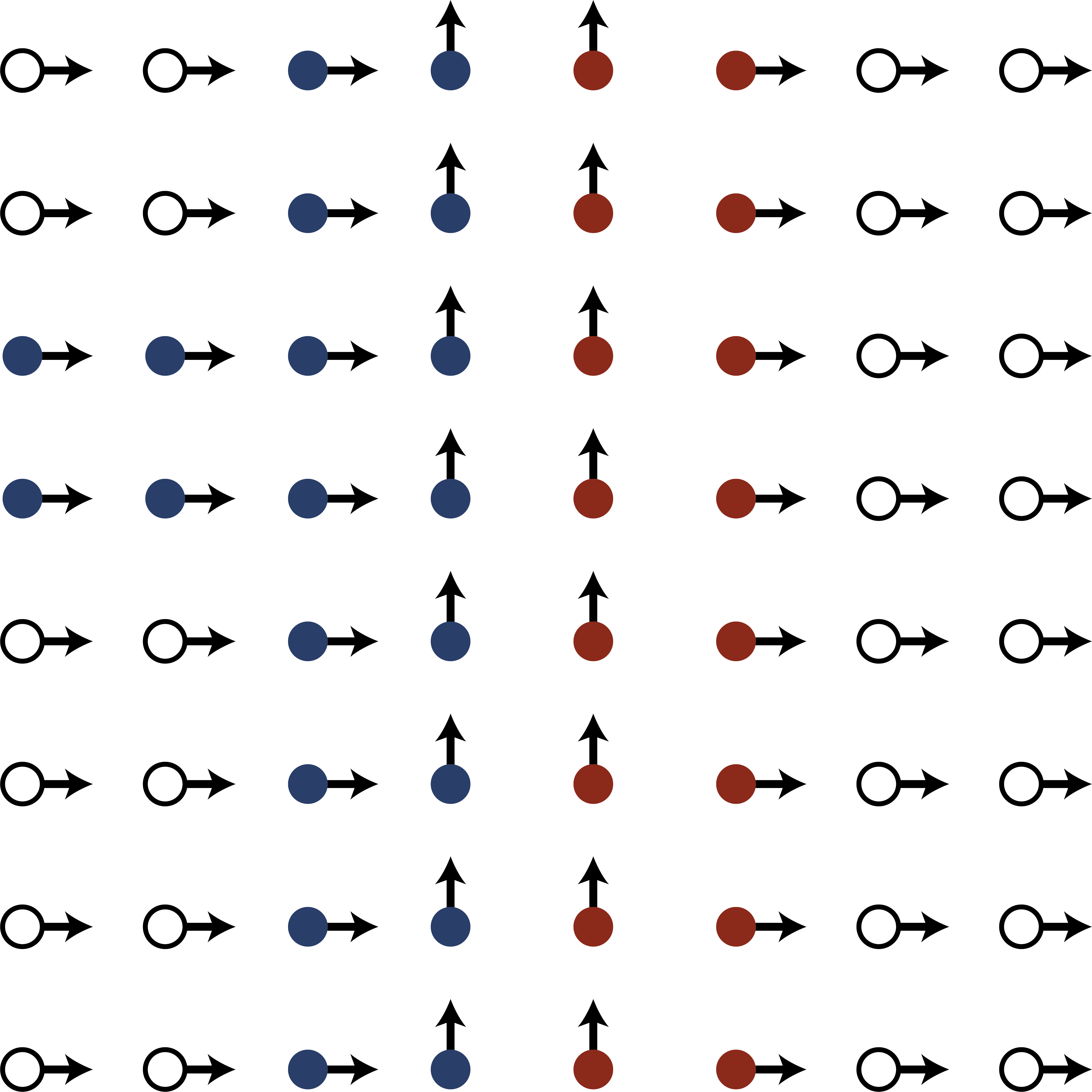}
    \caption{The small boxes of the region J shown in Figure \ref{fig:PLSboxAAnalysis}.}
    \label{fig:J}
\end{figure}

\paragraph{K.} In region K, as we can see in Figure \ref{fig:K}, all the small boxes have appeared before in regions A - G and hence we can directly conclude that there are no solutions in K.

\begin{figure}
    \centering
    \includegraphics[scale=0.2]{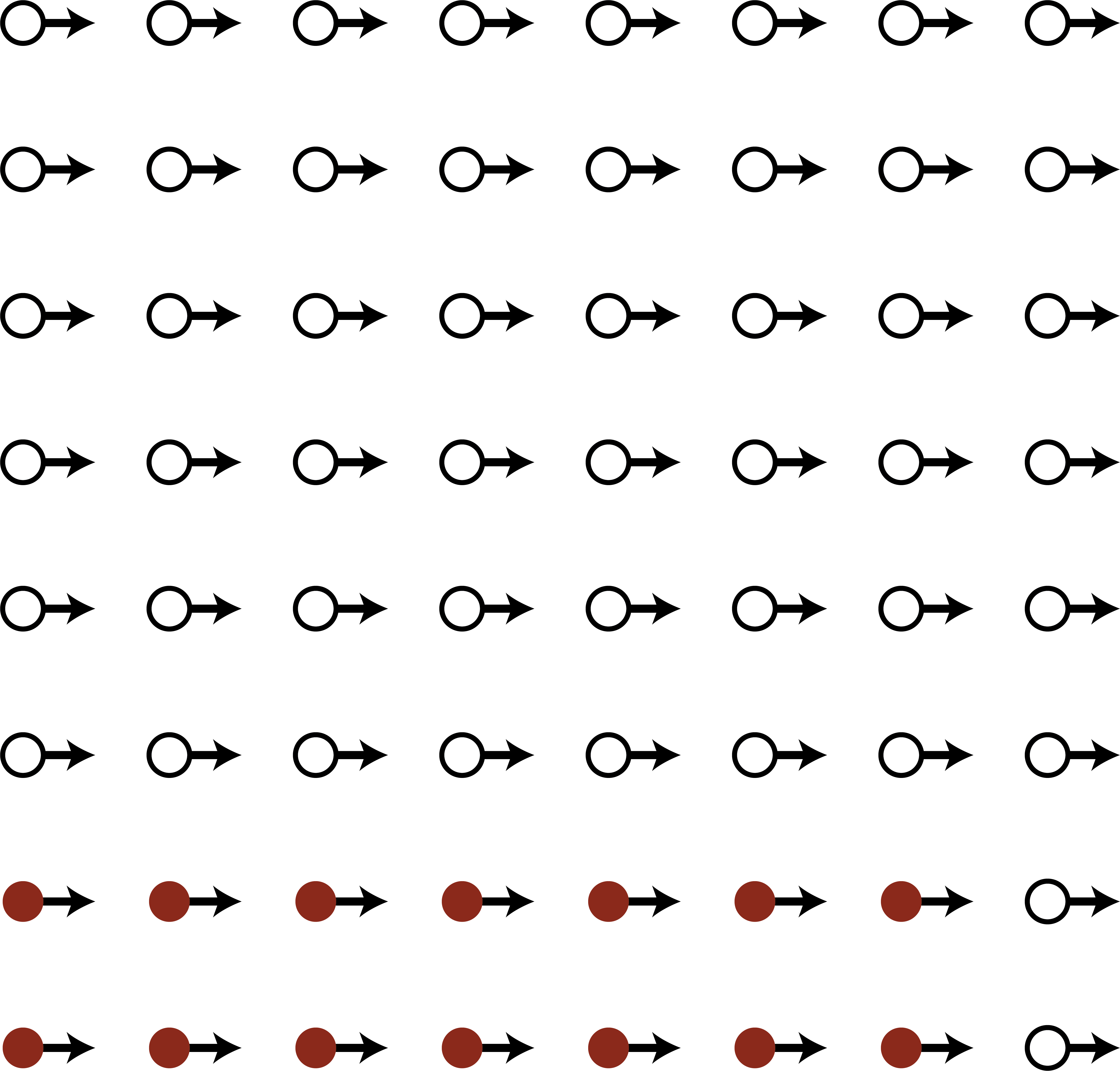}
    \caption{The small boxes of the region K shown in Figure \ref{fig:PLSboxAAnalysis}.}
    \label{fig:K}
\end{figure}

\paragraph{L.} We start with a figure of the region L in Figure \ref{fig:L}, where we indicate all the small boxes with colors and gradients that have not appeared in A - K.

\begin{figure}
    \centering
    \includegraphics[scale=0.2]{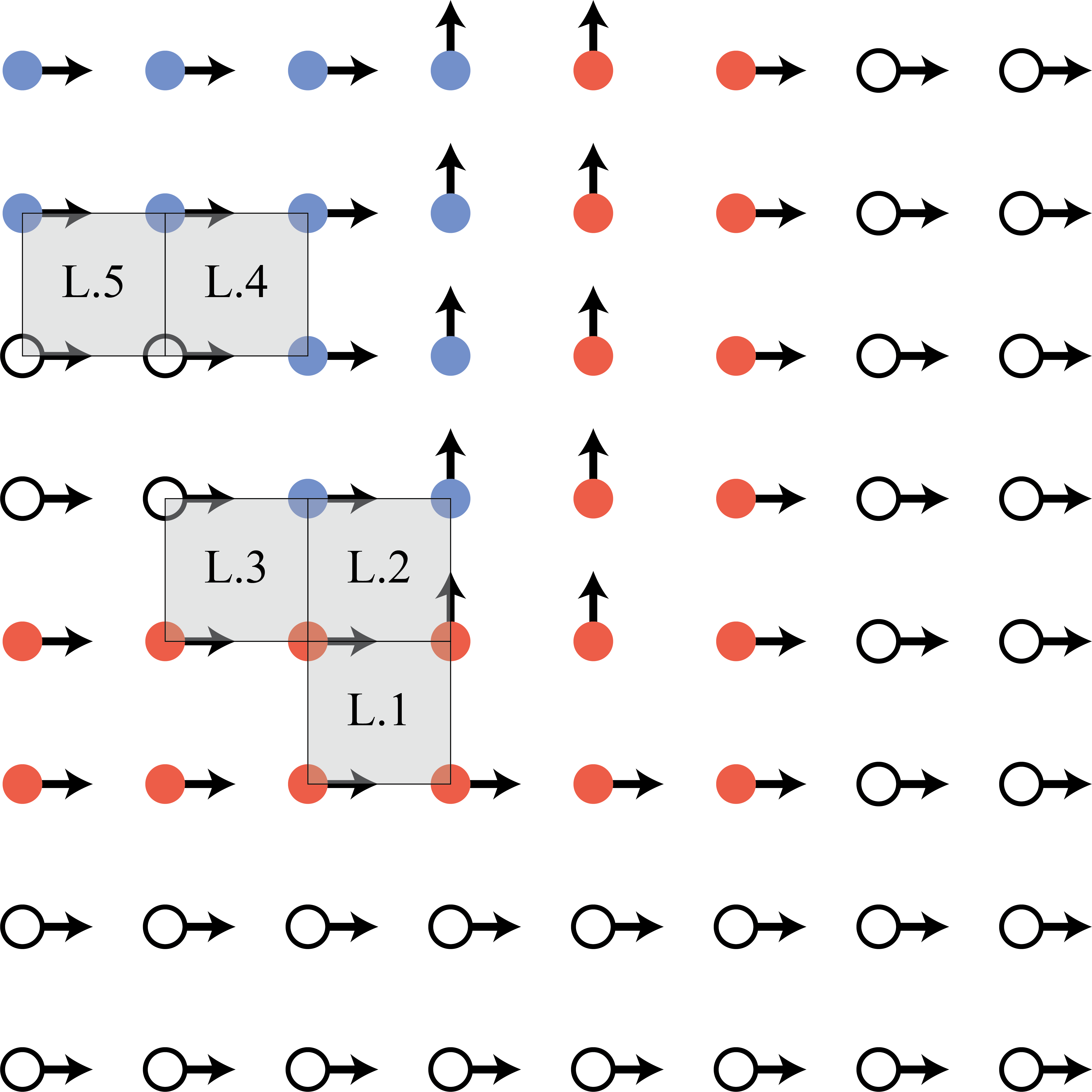}
    \caption{The small boxes of the region L shown in Figure \ref{fig:PLSboxBAnalysis}.}
    \label{fig:L}
\end{figure}

Most of the small boxes have appeared before except from the following:
\begin{itemize}
  \item[$\blacktriangleright$] The small box L.1 follows from Group 4 after applying a $y = - x$ reflection and negation.
    \item[$\blacktriangleright$] The small box L.2 follows from Group 2 after applying a $y = x$ reflection.
  \item[$\blacktriangleright$] The small boxes L.3, L.4, L.5 follow from plain application of Group 1 using that: (1) all the colors decrease linearly as the $x$ coordinate increases, and (2) light blue is everywhere at least $1$ smaller than the background.
\end{itemize}

So no solution appears in L.

\paragraph{M.} We start with a figure of the region M in Figure \ref{fig:M}, where we indicate all the small boxes with colors and gradients that have not appeared in A - L.

\begin{figure}
    \centering
    \includegraphics[scale=0.2]{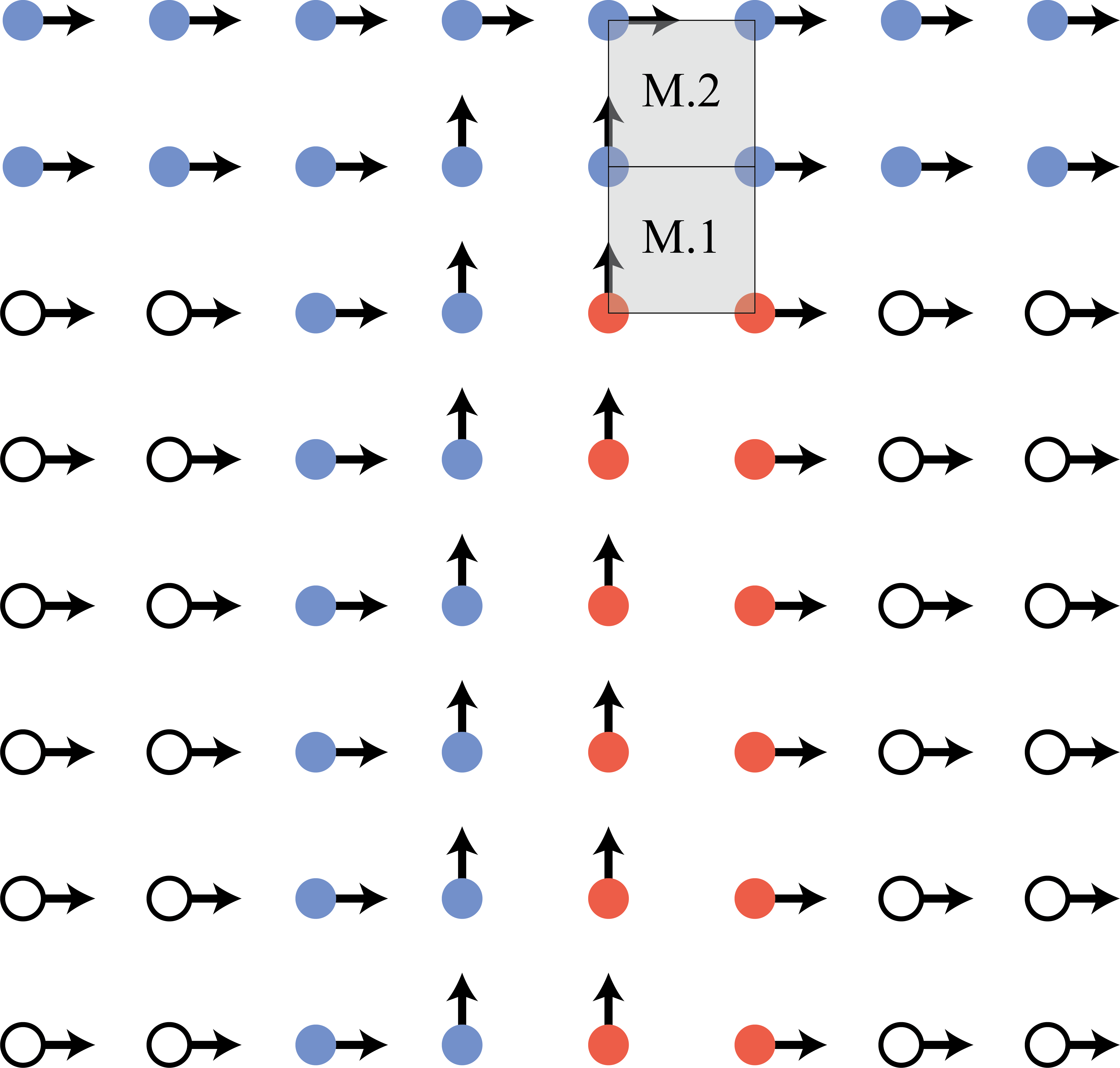}
    \caption{The small boxes of the region M shown in Figure \ref{fig:PLSboxBAnalysis}.}
    \label{fig:M}
\end{figure}

Most of the small boxes have appeared before except from the following:
\begin{itemize}
  \item[$\blacktriangleright$] The small box M.1 follows from Group 2 after applying a $y = - x$ reflection and negation.
  \item[$\blacktriangleright$] The small box M.2 follows from Group 4 after applying a $y = x$ reflection.
\end{itemize}

So no solution appears in M.

\paragraph{N.} In region N, as we can see in Figure \ref{fig:N}, all the small boxes have appeared before in regions A - M and hence we can directly conclude that there are no solutions in N.

\begin{figure}
    \centering
    \includegraphics[scale=0.2]{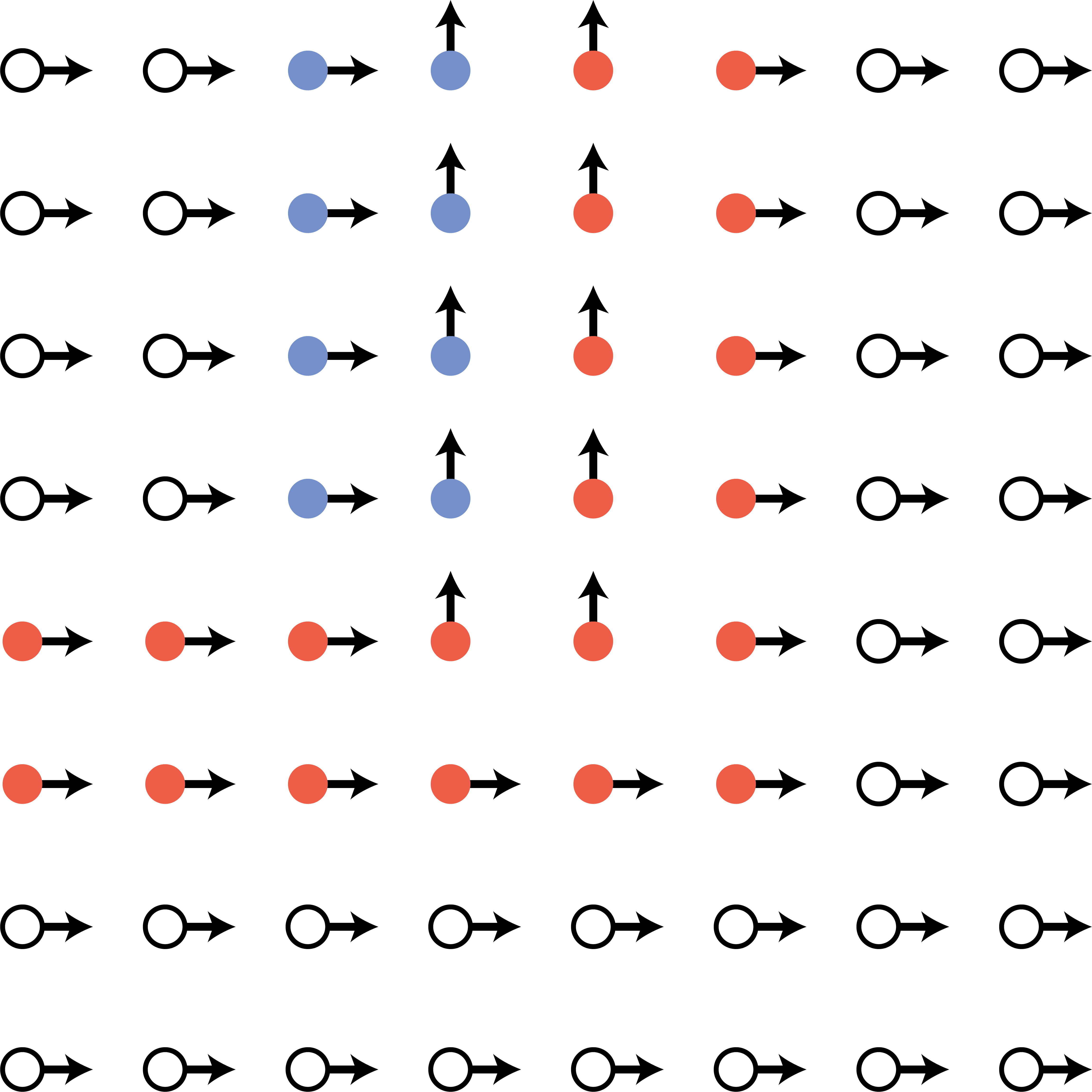}
    \caption{The small boxes of the region N shown in Figure \ref{fig:PLSboxBAnalysis}.}
    \label{fig:N}
\end{figure}

\paragraph{O.} In region O, as we can see in Figure \ref{fig:O}, all the small boxes have appeared before in regions A - M and hence we can directly conclude that there are no solutions in O.

\begin{figure}
    \centering
    \includegraphics[scale=0.2]{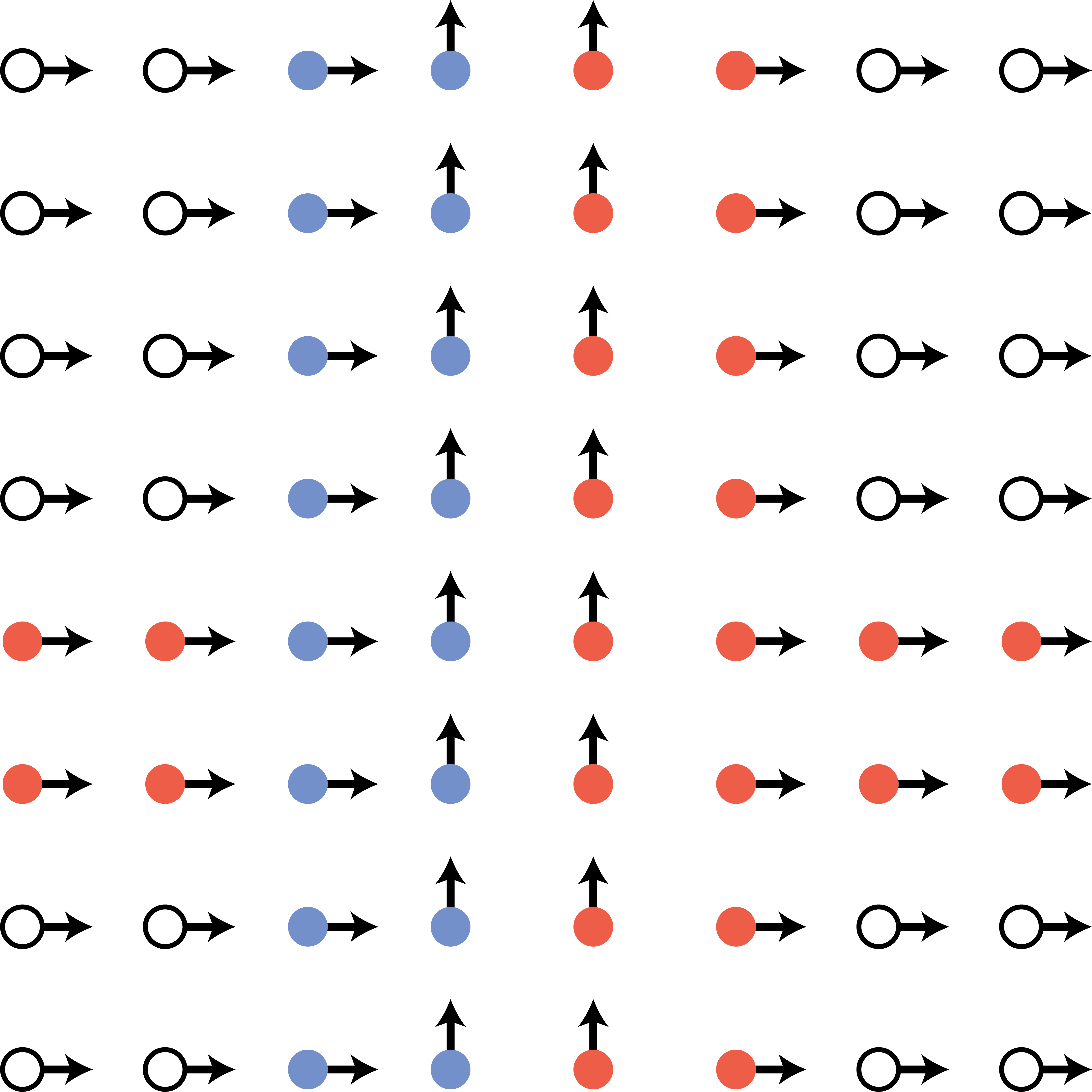}
    \caption{The small boxes of the region O shown in Figure \ref{fig:PLSboxBAnalysis}.}
    \label{fig:O}
\end{figure}

\paragraph{P.} In region P, as we can see in Figure \ref{fig:P}, all the small boxes have appeared before in regions A - M and hence we can directly conclude that there are no solutions in P.

\begin{figure}
    \centering
    \includegraphics[scale=0.2]{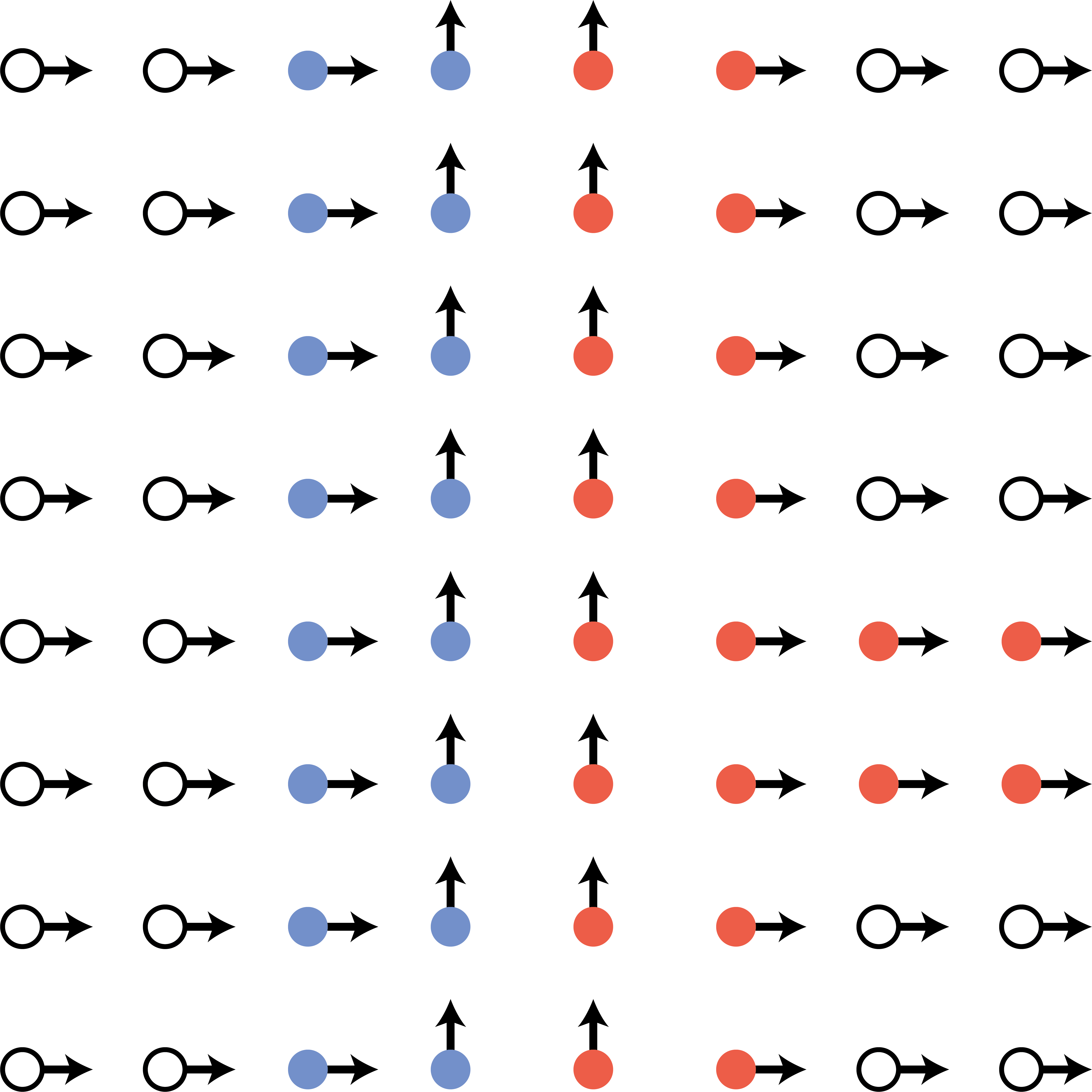}
    \caption{The small boxes of the region P shown in Figure \ref{fig:PLSboxBAnalysis}.}
    \label{fig:P}
\end{figure}

\paragraph{R.} In region R, as we can see in Figure \ref{fig:R}, all the small boxes have appeared before in regions A - M and hence we can directly conclude that there are no solutions in R.

\begin{figure}
    \centering
    \includegraphics[scale=0.2]{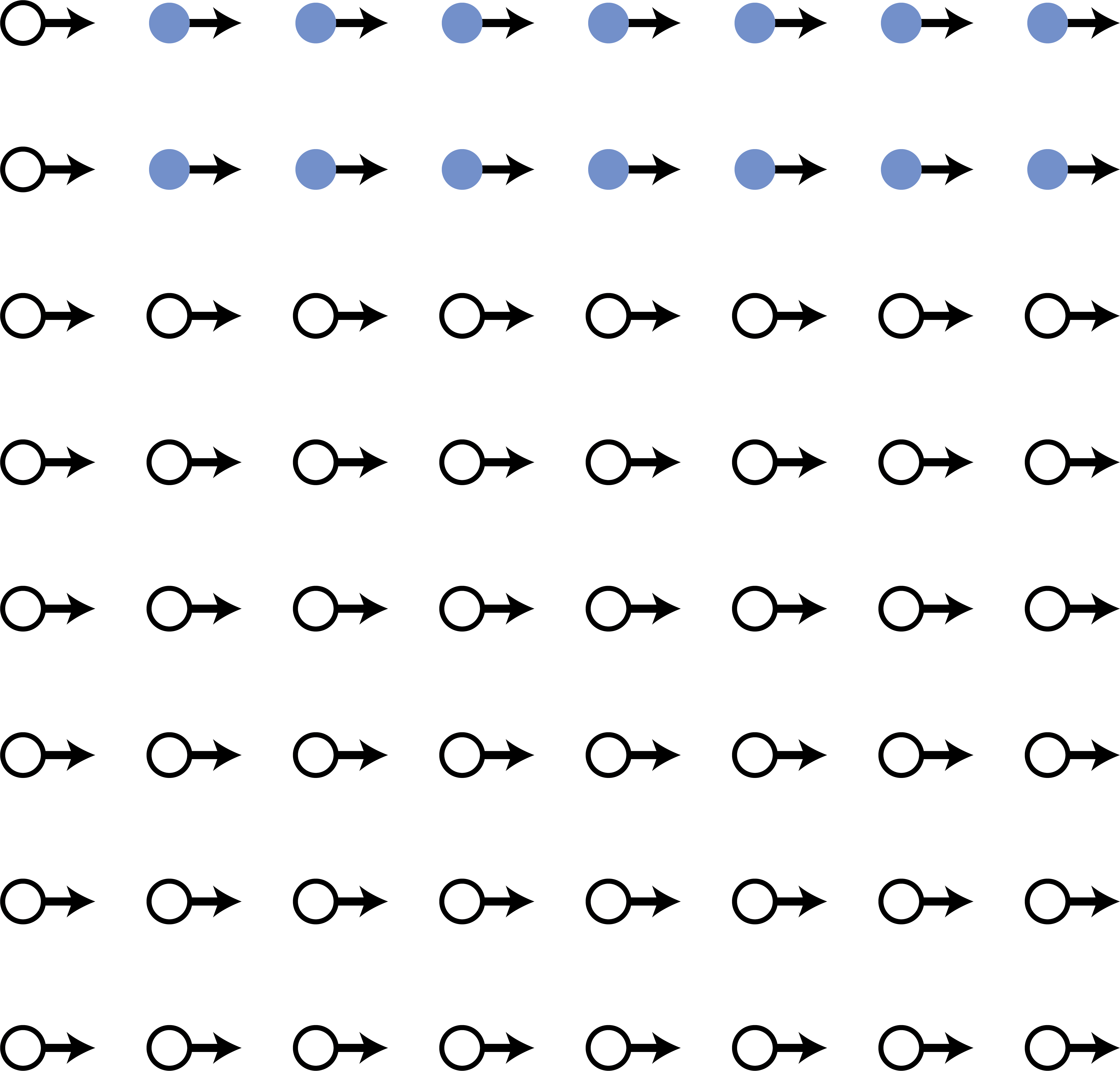}
    \caption{The small boxes of the region R shown in Figure \ref{fig:PLSboxBAnalysis}.}
    \label{fig:R}
\end{figure}

\noindent Since no solutions have appeared in any of F - R and there are no other types of small boxes that appear inside the PLS Box (A) and PLS Box (B), we conclude that Lemma \ref{lem:noSolutionsInsidePLSbox} follows.
\end{proof}

Now combining Lemma \ref{lem:noSolutionsLemma} and Lemma \ref{lem:noSolutionsInsidePLSbox} we conclude that using our construction any $0.01$-stationary point can only appear in places that correspond to solutions of the original \iter/ instance. To finish our proof of Lemma \ref{lem:hardness} it remains: (1) to understand the boundness, Lipschitzness, and smoothness of $g$, and (2) to show that the complexity of computing $g$ at every point $(x, y)$ is polynomial in the representations of $x$, $y$ and in the size of the binary circuit $C$ from the \iter/ instance. We proceed with these goals in the following sections.

\subsection{Function Values, Lipschitzness, and Smoothness of $g$} \label{sec:Lipschitzness}

Our goal in this section is to prove the following lemma that provides bounds on  the function values, the Lipschitzness, and the smoothness of $g$ and hence of $f$. The proof is essentially the same, almost verbatim, with the proof of Lemma 4.2 of \cite{FearnleyGHS22-gradient} with some small changes to be applied in our case.

\begin{lemma}\label{lem:functionProperties}
	The function $f: \R^2 \to \mathbb{R}$ defined in \eqref{eq:definitionPeriodic} has the following properties:
	\begin{enumerate}
		\item It is continuously differentiable on $\R^2$.
		\item $f$ and its gradient $\nabla f$ are Lipschitz-continuous on $\R^2$ with Lipschitz-constant $L = 2^{18} M$.
        \item It holds that $\abs{f(x, y)} \le 2^{14} M$ for all $x, y \in \R^2$.
	\end{enumerate}
\end{lemma}

\begin{proof}
  The first point is a simple consequence of \cite{russell1995polynomial}.

  Next we prove the properties 2.~and 3.
	
  \paragraph{\bf Lipschitz-continuity.} In order to prove the second point, we first show that $g$ and $\nabla g$ are $L$-Lipschitz-continuous in every small box of the grid $G_M$. Consider any small box. In our construction, the values of $g_x,g_y$ at the corners of the box are upper bounded by $\delta = 1/2$. The value of $g$ at the corners is upper bounded by the maximum value of the light red function $h_{LR}$ which is everywhere less than $6 \cdot M$. Furthermore, the value of $g$ at the corners is lower bounded by the minimum value of the dark blue function $h_{DB}$ which is everywhere at least $- 8 \cdot M$. So we conclude that $\abs{g(x, y)} \le 8 M$ for all $(x, y) \in G_M$. Thus, using \cref{eq:bicubic-coefficients}, it is easy to check that $|a_{ij}| \leq 2^{10} M$ for all $i,j \in \{0,1,2,3\}$. Furthermore, the partial derivatives of $g$ inside the small box can be written as:
  \begin{equation}\label{eq:bicubic-gradient}
	\frac{\partial g}{\partial x}(x,y) = \sum_{i=1}^3 \sum_{j=0}^3 i \cdot a_{ij} x^{i-1} y^j
	\qquad
	\frac{\partial g}{\partial y}(x,y) = \sum_{i=0}^3 \sum_{j=1}^3 j \cdot a_{ij} x^i y^{j-1}
  \end{equation}
  where in the above expressions $(x,y) \in [0,1]^2$ and correspond to the local coordinates inside the small box. Finally, it is easy to see that the monomials $x^i y^j$, $i,j \in \{0,1,2,3\}$, are $6$-Lipschitz continuous over $[0,1]^2$. Now, using \cref{eq:bicubic} and \cref{eq:bicubic-gradient}, we obtain that $g$ and $\nabla g$ are Lipschitz-continuous (w.r.t.\ the $\ell_2$-norm) with Lipschitz constant $L = 2^{18} M$ inside the small box.

  Now, since $g$ and $\nabla g$ are $L$-Lipschitz-continuous inside every small box and continuous over all of $[0, M]^2$, a simple argument shows that they are also $L$-Lipschitz-continuous over all of $[0, M]^2$ (e.g., see the proof of Lemma~4.2 in \citep{FearnleyGHS22-gradient}). Since $f$ is itself a tiling of $g$ we can use the same argument together with Lemma \ref{lem:boundaryConditions} to conclude that $f$ is also $L$-Lipschitz and $L$-smooth.

  \paragraph{Bounded Value.} As shown above we have $\abs{g(x, y)} \le 8 M$ for all $(x, y) \in G_M$, and $|a_{ij}| \leq 2^{10} M$ for all $i,j \in \{0,1,2,3\}$. It follows from \cref{eq:bicubic} that $\abs{g(x, y)} \le 2^{14} M$ for all $(x,y) \in [0,M]^2$. Finally, since $f$ is everywhere equal to the output of $g$ at some point we get that $\abs{f(x, y)} \le 2^{14} M$ over the whole domain.
\end{proof}

\subsection{Turing Machine that Evaluates $f$} \label{sec:computationF}

The last part of the proof of Lemma \ref{lem:hardness} is to show that there exist a Turing machine $\calC_f$ such that given two numbers $x, y \in [0, 1]$ with bit complexity $b \ge \len(x)$ and $b \ge \len(y)$ we can compute the value and the gradient of $f$ at any point in time that is polynomial in $b$ and in the size of the boolean circuit $C$ that we are given as an input that describes the \iter/ instance that we are reducing from. It is easy to see that the computation $z \mapsto z - M \cdot \floor{\frac{z}{M}}$ can be done in time polynomial in $\len(z)$ and $\len(M)$. $\len(z)$ will be bounded by $b$ and $M$ is a natural number such that $M = O(2^n)$, so $\len(M) = O(n)$. Since $C$ is a circuit with $n$ inputs and $n$ outputs, $O(n)$ is certainly polynomial in the size of $C$.

Since $z \mapsto z - M \cdot \floor{\frac{z}{M}}$ can be done in polynomial time, it suffices to show that we can compute $g(x, y)$ in polynomial time. To do that we first need to identify the type of the small box where $(x, y)$ belongs. To do this outside the PLS boxes, we need time linear in $b$. Inside the PLS boxes, on the other hand, we need to evaluate the circuit $C$ with input $u$ that corresponds to the column of the medium box that $(x, y)$ belongs to, and with input $v$ that corresponds to the row of the medium box that $(x, y)$ belongs to. So we need to evaluate $C(u), C(v)$ as well as $C(C(u)), C(C(v))$ and if we specify these then we can identify the type of the small box that $(x, y)$ belongs to. So we need to evaluate $C$ four times, which takes linear time in the size of $C$. Finally, once we identify the small box, we need to compute the bi-cubic interpolation which involves solving a linear system and computing a third degree polynomial with numbers that use at most $\max\set{b, \len(M)}$ bits. It is well known that both of these can be done in time polynomial in the description of the number and hence we conclude that there exists an efficient Turing machine that computes $g$ which implies an efficient Turing machine that computes $f$.

\subsection{Proof of Lemma \ref{lem:hardness}} \label{app:lem:hardness}

To show Lemma \ref{lem:hardness} we combine Lemma \ref{lem:noSolutionsLemma}, Lemma \ref{lem:noSolutionsInsidePLSbox}, Lemma \ref{lem:boundaryConditions}, Lemma \ref{lem:functionProperties} and the discussion of Section \ref{sec:computationF} and Lemma \ref{lem:hardness} follows.

\section{Query Lower Bound for 2D -- Proof of Theorem \ref{thm:query2D}} \label{app:query2D}

The first observation to show this theorem is that the black box version of \iter/ has a query lower bound of $2^n$. Consider any algorithm $\calA$ that solves \iter/ and has only query access to the input circuit $C$. At every step $t$ of the algorithm $\calA$, there exists a sequence $S_t = (u_0, u_1, \dots, u_{i_t})$ such that $u_0 = 1$, $u_j = C(u_{j - 1})$, for all $j \in [i_t]$ and the algorithm $\calA$ has queried all the nodes $u \in S_t$. In words, $S_t$ corresponds to the longest path that $\calA$ has discovered start from $1$. The adversary against $\calA$ is very simple
\begin{itemize}
  \item[$\triangleright$] if the algorithm queries to learn $C(v)$ for some $v \neq u_{i+t}$ then the adversary replies $C(v) = v$,
  \item[$\triangleright$] if the algorithm queries to learn $C(u_t)$ then the adversary replies with the smallest node that has never been queried before.
\end{itemize}
Hence, the only way for $\calA$ find a solution at any node $v$ is to have queried everything before $v$, otherwise the adversary will just continue the path $S$ without giving a solution. The same way, the only way that $\calA$ can find a solution at any node $v$ is that $\calA$ has queried everything after $v$ as well because if there are still available nodes after $v$ the adversary will continue the path $S$. This implies that any algorithm $\calA$ has to query every node in the worst case before it finds a solution. So in the worst case any algorithm $\calA$ will take time $2^n$.

Now as we discussed in the proof of Lemma \ref{lem:hardness} the function $f$ that we construct in Section \ref{app:hardness} has the properties: (1) to evaluate $f$ at any given point we need at most $4$ queries to $C$, (2) any stationary point of $f$ reveals a solution to \iter/ with input $C$. These two combined give us that any algorithm that finds a $0.01$-stationary point in functions that we construct will in the worst-case take time at least $2^{n - 2}$. Now observe that the parameter $M$ that we use to construct $f$ in Section \ref{app:hardness} satisfies $M = \Theta(2^n)$ which means that any algorithm that finds a $0.01$-stationary point in functions that we construct will in the worst-case take time at least $\Omega(M)$. The last thing to fix is the parameters of the algorithm. Currently, the function $f$ that we construct has boundedness $B = O(M)$, and Lipschitzness/smoothness $L = O(M)$. If we multiply the function value by $1/M$, i.e., we define $\tilde{f}(x) = \frac{1}{M} \cdot f(x)$, then to find a $0.01$-stationary point for $f$ we need to find a $\frac{1}{100 M}$-stationary point of $\tilde{f}$ so our new target $\eps$ is $\frac{1}{100 M}$. At the same time $\tilde{f}$ has now boundedness $B = O(1)$, and Lipschitzness/smoothness $L = O(1)$ and of course any algorithm would still need $\Omega(M)$ time to find a $\frac{1}{100 M}$-stationary point. Hence, any algorithm that finds an $\eps$-stationary point for $\tilde{f}$ will need in the worst case $\Omega(1/\eps)$ queries and Theorem \ref{thm:query2D} follows.

\bibliographystyle{plainnat}
\bibliography{references}

\begin{thebibliography}{18}
\providecommand{\natexlab}[1]{#1}
\providecommand{\url}[1]{\texttt{#1}}
\expandafter\ifx\csname urlstyle\endcsname\relax
  \providecommand{\doi}[1]{doi: #1}\else
  \providecommand{\doi}{doi: \begingroup \urlstyle{rm}\Url}\fi

\bibitem[Babichenko and Rubinstein(2021)]{BabichenkoR21-congestion}
Yakov Babichenko and Aviad Rubinstein.
\newblock Settling the complexity of {N}ash equilibrium in congestion games.
\newblock In \emph{Proceedings of the 53rd ACM Symposium on Theory of Computing
  (STOC)}, pages 1426--1437, 2021.
\newblock \doi{10.1145/3406325.3451039}.

\bibitem[Beame et~al.(1998)Beame, Cook, Edmonds, Impagliazzo, and
  Pitassi]{BeameCEIP98-NP-search}
Paul Beame, Stephen Cook, Jeff Edmonds, Russell Impagliazzo, and Toniann
  Pitassi.
\newblock The relative complexity of {NP} search problems.
\newblock \emph{Journal of Computer and System Sciences}, 57\penalty0
  (1):\penalty0 3--19, 1998.
\newblock \doi{10.1006/jcss.1998.1575}.

\bibitem[Boyd et~al.(2004)Boyd, Boyd, and Vandenberghe]{boyd2004convex}
Stephen Boyd, Stephen~P Boyd, and Lieven Vandenberghe.
\newblock \emph{Convex optimization}.
\newblock Cambridge University Press, 2004.

\bibitem[Bubeck and Mikulincer(2020)]{BubeckM20-trap-gradient}
S\'ebastien Bubeck and Dan Mikulincer.
\newblock How to trap a gradient flow.
\newblock In \emph{Proceedings of the 33rd Conference on Learning Theory
  (COLT)}, pages 940--960, 2020.
\newblock URL \url{http://proceedings.mlr.press/v125/bubeck20b.html}.

\bibitem[Buresh-Oppenheim and Morioka(2004)]{BureshM04-NP-search-problems}
Joshua Buresh-Oppenheim and Tsuyoshi Morioka.
\newblock Relativized {NP} search problems and propositional proof systems.
\newblock In \emph{Proceedings of the 19th IEEE Conference on Computational
  Complexity (CCC)}, pages 54--67, 2004.
\newblock \doi{10.1109/CCC.2004.1313795}.

\bibitem[Buss and Johnson(2012)]{BussJ12-propositional-NP-search}
Samuel~R. Buss and Alan~S. Johnson.
\newblock Propositional proofs and reductions between {NP} search problems.
\newblock \emph{Annals of Pure and Applied Logic}, 163\penalty0 (9):\penalty0
  1163--1182, 2012.
\newblock \doi{10.1016/j.apal.2012.01.015}.

\bibitem[Carmon et~al.(2020)Carmon, Duchi, Hinder, and
  Sidford]{CarmonDHS20-stationary}
Yair Carmon, John~C. Duchi, Oliver Hinder, and Aaron Sidford.
\newblock Lower bounds for finding stationary points {I}.
\newblock \emph{Mathematical Programming}, 184:\penalty0 71--120, 2020.
\newblock \doi{10.1007/s10107-019-01406-y}.

\bibitem[Chewi et~al.(2023)Chewi, Bubeck, and Salim]{chewi23a}
Sinho Chewi, S\'ebastien Bubeck, and Adil Salim.
\newblock On the complexity of finding stationary points of smooth functions in
  one dimension.
\newblock In \emph{Proceedings of the 34th International Conference on
  Algorithmic Learning Theory (ALT)}, pages 358--374, 2023.
\newblock URL \url{https://proceedings.mlr.press/v201/chewi23a.html}.

\bibitem[Daskalakis and Papadimitriou(2011)]{DaskalakisP11-CLS}
Constantinos Daskalakis and Christos Papadimitriou.
\newblock Continuous local search.
\newblock In \emph{Proceedings of the 22nd ACM-SIAM Symposium on Discrete
  Algorithms (SODA)}, pages 790--804, 2011.
\newblock \doi{10.1137/1.9781611973082.62}.

\bibitem[Daskalakis et~al.(2021)Daskalakis, Skoulakis, and
  Zampetakis]{DaskalakisSZ21-min-max}
Constantinos Daskalakis, Stratis Skoulakis, and Manolis Zampetakis.
\newblock The complexity of constrained min-max optimization.
\newblock In \emph{Proceedings of the 53rd ACM Symposium on Theory of Computing
  (STOC)}, pages 1466--1478, 2021.
\newblock \doi{10.1145/3406325.3451125}.

\bibitem[Fearnley et~al.(2022)Fearnley, Goldberg, Hollender, and
  Savani]{FearnleyGHS22-gradient}
John Fearnley, Paul Goldberg, Alexandros Hollender, and Rahul Savani.
\newblock The complexity of gradient descent: {CLS} = {PPAD} $\cap$ {PLS}.
\newblock \emph{Journal of the ACM}, 70\penalty0 (1):\penalty0 1--74, 2022.
\newblock \doi{10.1145/3568163}.

\bibitem[G{\"o}{\"o}s et~al.(2022)G{\"o}{\"o}s, Hollender, Jain, Maystre,
  Pires, Robere, and Tao]{GoosHJMPRT22-separations}
Mika G{\"o}{\"o}s, Alexandros Hollender, Siddhartha Jain, Gilbert Maystre,
  William Pires, Robert Robere, and Ran Tao.
\newblock Separations in proof complexity and {TFNP}.
\newblock In \emph{Proceedings of the 63rd Symposium on Foundations of Computer
  Science (FOCS)}, pages 1150--1161, 2022.
\newblock \doi{10.1109/focs54457.2022.00111}.

\bibitem[Johnson et~al.(1988)Johnson, Papadimitriou, and
  Yannakakis]{JohnsonPY88-PLS}
David~S. Johnson, Christos~H. Papadimitriou, and Mihalis Yannakakis.
\newblock How easy is local search?
\newblock \emph{Journal of Computer and System Sciences}, 37\penalty0
  (1):\penalty0 79--100, 1988.
\newblock \doi{10.1016/0022-0000(88)90046-3}.

\bibitem[Morioka(2001)]{Morioka01-Mthesis-PLS}
Tsuyoshi Morioka.
\newblock Classification of search problems and their definability in bounded
  arithmetic.
\newblock Master's thesis, University of Toronto, 2001.
\newblock URL
  \url{https://www.collectionscanada.ca/obj/s4/f2/dsk3/ftp04/MQ58775.pdf}.

\bibitem[Nemirovskij and Yudin(1983)]{nemirovskij1983problem}
Arkadij~Semenovi{\v{c}} Nemirovskij and David~Borisovich Yudin.
\newblock \emph{Problem complexity and method efficiency in optimization}.
\newblock Wiley-Interscience, 1983.

\bibitem[Nesterov(2003)]{nesterov2003introductory}
Yurii Nesterov.
\newblock \emph{Introductory lectures on convex optimization: A basic course}.
\newblock Springer Science \& Business Media, 2003.

\bibitem[Russell(1995)]{russell1995polynomial}
William~S. Russell.
\newblock Polynomial interpolation schemes for internal derivative
  distributions on structured grids.
\newblock \emph{Applied Numerical Mathematics}, 17\penalty0 (2):\penalty0
  129--171, 1995.
\newblock \doi{10.1016/0168-9274(95)00014-L}.

\bibitem[Vavasis(1993)]{Vavasis93-local-min}
Stephen~A. Vavasis.
\newblock Black-box complexity of local minimization.
\newblock \emph{SIAM Journal on Optimization}, 3\penalty0 (1):\penalty0 60--80,
  1993.
\newblock \doi{10.1137/0803004}.

\end{thebibliography}

\end{document}